\documentclass{amsart}

\usepackage{hyperref}

\usepackage{amssymb} 
\usepackage{enumerate, amsfonts, latexsym,epsfig, color}
\usepackage{epstopdf, amscd}

\usepackage{mathptmx}
\usepackage{thm-restate}
\usepackage{xspace}

\input xy
\xyoption{all}

\newtheorem {theorem}{Theorem}[section]
\newtheorem {lemma} [theorem] {Lemma}
\newtheorem {proposition} [theorem] {Proposition}

\newtheorem {corollary} [theorem] {Corollary}
\newtheorem {definition} [theorem] {Definition}

\newtheorem {remark} [theorem] {Remark}

\newtheorem {question} {Question}
\newtheorem {notation} [theorem] {Notation}
\newtheorem {terminology} [theorem] {Terminology}

\newtheorem{assumption} [theorem] {Standing Assumption}

\newtheorem {claim}{Claim}

\newtheorem{thmA}{Theorem}

\newtheorem{corA}[thmA]{Corollary}

\usepackage{amsmath}

\def\Man {\mc{M}^\pi}
\def\MgenM {\mc{M}_{\operatorname{Gen}}}
\def \Mgen {\MgenM^\pi}
\def \Mgeo {\mc{M}_{\operatorname{Geo}}^\pi}
\def \Mirr {\mc{M}_{\operatorname{Irr}}^\pi}
\def \Mori {\mc{M}_{\operatorname{Ori}}^\pi}

\def \Mlsf {\mc{M}_{\operatorname{SFH}}^\pi}
\def\orb {\mc{M}_{\operatorname{Orb}}^\pi}

\def \Tcut {T_{\operatorname{cut}}}

\def\Cdiv {$\mathcal{C}$--divergent\xspace}
\def\Tdiv {$\mathcal{T}$--divergent\xspace}

\def\LSF { SFH-type }

\def\Cay {\operatorname{Cay}}

\def\co{\colon\thinspace}

\def\mc {\mathcal}

\def\R {\mathbb R}
\def\bH {\mathbb H}
\def\bC {\mathbb C}

\def\N {\mathbb N}
\def\Z {\mathbb Z}

\def\bp {\mathfrak o}
\def\Gd {{{\mathbb{G}}}} 
\def\Gdf {\mathbb{G}^r} 
\def\V {V} 
\def\D {\mathbb{D}} 
\def\E {\mathbb{E}} 

\def\Boun {\mc{B}}  

\def\Stab {\operatorname{Stab}}
\def\Fix {\operatorname{Fix}}
\def\Aut {\operatorname{Aut}}

\def\SC {Z^\omega} 

\def\Hom {{\operatorname{Hom}}}

\def\len {\ell}

\def\ker {\operatorname{Ker}}

\def\wker {\ker^\omega}
\def\walmost {$\omega$--almost }
\def\wapprox {$\omega$--approximation }
\def\wapproxs {$\omega$--approximations }
\def\wlim {{\lim}^\omega}

\def\onto {\twoheadrightarrow}

\def\Mod {\operatorname{Mod}}

\def\geom {{\operatorname{geom}}}

\newcommand{\Coll}{\mc{C}} 

\begin{document}

\title{Homomorphisms to $3$--manifold groups}

\author[Daniel Groves]{Daniel Groves}
\address{Department of Mathematics, Statistics, and Computer Science,
University of Illinois at Chicago,
322 Science and Engineering Offices (M/C 249),
851 S. Morgan St.,
Chicago, IL 60607-7045}
\email{groves@math.uic.edu}

\author[Michael Hull]{Michael Hull}

\address{Department of Mathematics \& Statistics,
University of North Carolina at Greensboro, 
116 Petty Building,
Greensboro, NC 27402-6170}
\email{mbhull@uncg.edu}

\author[Hao Liang]{Hao Liang}

\address{School of Mathematics and big data, Foshan University, Foshan, Guangdong, 528000, People's Republic of China}
\email{lianghao1019@hotmail.com}

\thanks{The first author is supported in part by NSF grant DMS-1904913. The third author is supported by NSFC (No.11701581 and No.11521101).
}

\begin{abstract}
We prove foundational results about the set of homomorphisms from a finitely generated group to the collection of all fundamental groups of compact $3$--manifolds and answer questions of Agol--Liu \cite{Agol-Liu} and Reid--Wang--Zhou \cite{RWZ}.  
\end{abstract}

\date{\today}

\maketitle

\setcounter{tocdepth}{1}
\tableofcontents

\section{Introduction}

There is a well-developed structure theory for compact $3$--manifolds based on the Geometrization Theorem of Perelman (conjectured by Thurston). In this paper we develop a structure theory for the collection of all maps between $3$--manifolds, from the point of view of their fundamental groups.  Let $\Man$ be the set of all isomorphism classes of fundamental groups of $3$--manifolds  (see \cite{AFW} for background).  Our main result gives a positive answer to a question of Agol--Liu \cite[Question 10.3]{Agol-Liu}.

\begin{thmA}\label{t:main}
For any finitely generated group $G$ there is a finitely presented group $G_0$ and an epimorphism $\alpha \co G_0 \onto G$ so that for every $\Gamma \in \Man$
 the map 
\[	\alpha^* \co \Hom(G,\Gamma) \to \Hom(G_0,\Gamma)	\]
induced by precomposition with $\alpha$ is a bijection.
\end{thmA}

Of course, if $G$ is finitely presented, then one can take $G_0 = G$, and there is nothing to prove.  Since
finitely generated $3$--manifold groups are finitely presented \cite{Scott:Coherence}, it might appear that Theorem~\ref{t:main} is not relevant to the study of homomorphisms between $3$--manifold groups.  On the contrary, Theorem~\ref{t:main} is an important tool, and in particular implies the following theorem, answering a question of Reid--Wang--Zhou \cite[Question 3.1.(C2)]{RWZ}.

\begin{thmA} \label{t:DCC}
Let $\left( M_i \right)_{i \ge 1}$ be a sequence of compact $3$--manifolds (possibly with boundary). Every infinite sequence of epimorphisms
$\pi_1(M_1) \onto \cdots \onto \pi_1(M_n) \onto \cdots$ contains an isomorphism.
\end{thmA}
It is worth remarking that Reid, Wang and Zhou ask their question about closed, orientable, aspherical $3$--manifolds, whereas the only assumption we need is that they are compact. Theorem~\ref{t:DCC} proves the descending chain condition for the following partial order on $3$--manifold groups: $G_1 \ge G_2$ if there is an epimorphism $\phi \co G_1 \to G_2$.  Per Calegari \cite{Calegari:mathscinet}, understanding this order is an important question in $3$--manifold topology.  

Theorem~\ref{t:DCC} is proved in Section~\ref{s:outline of proof} as a consequence of Theorem~\ref{t:main}.  The proof of Theorem~\ref{t:DCC} is by contradiction, and the group $G$ from Theorem~\ref{t:main} used in this proof arises as a direct limit of the sequence of maps $\pi_1(M_i) \onto \pi_1(M_{i+1})$. 

Partial results about Reid, Wang and Zhou's question were already known.  Reid--Wang--Zhou \cite[Theorem 3.4]{RWZ} gave a positive answer to this question in case all $M_i$ are (closed, orientable, aspherical) Seifert $3$-manifolds. A key step of their proof is showing that epimorphisms between fundamental groups of aspherical Seifert 3-manifolds with the same $\pi_1$ rank are realized by non-zero degree maps, which does not hold in general for epimorphisms between closed aspherical 3-manifolds groups of the same rank (see \cite{Gonzalez-Acuna-Ramirez2003}).  

Soma gave a positive answer to Reid, Wang and Zhou's question in case all $M_i$ are hyperbolic \cite[Theorem 1]{Soma}.  Since every hyperbolic 3-manifold group embeds into $\mathrm{PSL}(2, \mathbb C)$, Soma's result can be derived using the classical Hilbert Basis Theorem. To make this approach work in the general case, one would need to not only show that all 3-manifold groups are linear (which is still open for certain graph manifolds), but also to find a \emph{single} linear group into which \emph{every} 3-manifold group embeds. Whether this can be done is an open question which seems to be out of reach of current methods. 

In order to solve Reid, Wang and Zhou's question in full generality, a new approach was needed.  We avoid questions of linearity and assumptions of positive degree by using Theorem~\ref{t:main} as a replacement for the Hilbert Basis Theorem.  From the point of view of equations over groups, a homomorphism $h \co G \to H$ between groups corresponds to a solution in $H$ to a system of equations (where the generators of $G$ are variables, and the relators of $G$ give equations).  In this language, Theorem~\ref{t:main} is an analogue of the Hilbert  Basis Theorem, uniformly for all groups in $\Man$.  In the language of \cite{GH}, Theorem~\ref{t:main} says that $\Man$ is an \emph{equationally noetherian} family of groups. 

Many other people have studied collections of maps between $3$--manifolds or between $3$--manifold groups. Examples of previous work include \cite{Rong92, Boileau-Wang-JDG, Hayat-legrand-Wang-Zieschang, Reid-Wang99, Soma-JDG-98, Soma2000, Wang-Zhou2002, Derbez2003, Derbez2007, Boileau-Boyer-Wang2008, Boileau-Rubinstein-Wang2014, Liu, Agol-Liu}.  For an overview of the state of the art in 2002 about positive degree maps, see Wang's ICM address \cite{Wang:ICM}.  Highlights of work since then have been Agol and Liu's solution to Simon's Conjecture (about epimorphisms between knot groups) \cite{Agol-Liu}, and Liu's proof that for fixed $k \ge 1$, a given closed $3$-manifold admits a degree $k$ map onto only finitely many other closed $3$--manifolds \cite{Liu}.  However, as far as we are aware, all previous work in this area has made assumptions about the $3$--manifolds or the types of maps considered, whereas Theorem \ref{t:main} makes no such assumptions.

Since the Borel Conjecture is true in dimension $3$ (see for instance \cite[Section 5]{KreckLuck}), we also obtain:

\begin{corA}\label{cor:dcc_homeo}
Let $(M_i)_{i \in \N}$ be a sequence of closed, aspherical $3$--manifolds and let  $f_i\colon M_i\to M_{i+1}$ be  $\pi_1$--surjective maps. There exists $N \in \mathbb{N}$ so $f_i$ is homotopic to a homeomorphism for all $i\geq N$.
\end{corA}

Note that the asphericity assumption in Corollary~\ref{cor:dcc_homeo} cannot be dropped, as can be seen using all $M_i = \mathbb{S}^3$.

We now turn to our approach to proving Theorem~\ref{t:main}.  We fix a finitely generated group $G$ and consider the collection of all homomorphisms from $G$ to $\Gamma$, for all $\Gamma \in \Man$.  Our basic approach, inspired by Sela \cite{sela:dio1,Sela:hyp} and his work on limit groups, etc., is to consider a sequence of homomorphisms and try to extract limits.  In this theory, the goal is usually to find a limiting $\R$-tree from a \emph{divergent} sequence of homomorphisms.  At this point, Sela's shortening argument is used to reduce to the case of non-divergent homomorphisms which are usually constant modulo conjugation. 

In order to find limiting $\R$--trees, one way is to use actions on $\delta$--hyperbolic spaces. For $3$--manifold groups, there are the Bass--Serre trees arising from the Kneser--Milnor decomposition or the geometric decomposition.  This is the context of previous work by the first two authors \cite{GH}, who used limiting $\R$--trees and a version of Sela's shortening argument to reduce many questions to the case where a sequence of homomorphisms does not diverge on these trees.  This is the starting point of this paper.  However, this `non-divergent' case is still highly nontrivial.  We do obtain a limiting simplicial action on a tree in this setting, which induces a splitting of the associated limit group $L$. We call this splitting the \emph{geometric decomposition} of $L$. For each vertex group $V$ of this splitting, there is a corresponding sequence of pieces in the geometric decompositions of the corresponding sequence of $3$--manifold groups. For each such piece, we construct a new space called the {\em collapsed space} which is a variant of a construction of Farb from \cite{farb:relhyp} (Section~\ref{s:collapsing Divergent}). This sequence of spaces induces a new limiting action of $V$ on an $\R$-tree (see Section~\ref{s:R-trees}), and the collapsed spaces are carefully tuned so that the limiting actions behaves sufficiently well and so we can repeat the shortening argument for this new action (see Section~\ref{s:res-short}).   

Even having applied the shortening argument twice, and having a new more restrictive notion of non-divergent, the non-divergent sequences are still far from being constant.  Thus, we have yet another layer to our analysis.   In this case, we construct a new splitting of our limit group $L$ dual to a tree that we construct as a quotient of the Cayley graph of $L$. This splitting has an algebraic property which we call {edge-twisted} (see Definition~\ref{def:edge-twist}). We prove some technical results about edge-twisted splittings in Appendix~\ref{app:edge-twist}, and applying these to our splitting of $L$ allows us to conclude that each vertex group in the geometric decomposition of $L$ is itself a limit group coming from a sequence of homomorphisms to geometric 3-manifolds. Finally, we can complete the proof by using the results of the third author \cite{Liang} to understand these limit groups over geometric 3-manifolds.

This paper is in some sense a culmination of the prior work \cite{GH,Liang} of the authors.  On the other hand, we expect the tools in this paper to be useful for many other questions about maps between $3$--manifold groups. In particular, although Theorems~\ref{t:main} and \ref{t:DCC} are immediate from the Hilbert Basis Theorem in the case of Kleinian groups, we believe that the tools built in this paper will be useful for other questions about maps between and to Kleinian groups, and we intend to pursue this direction in future work.

We now outline the contents of this paper.  In Section~\ref{s:outline of proof} we outline the proof of Theorem~\ref{t:main} and use this to derive Theorem~\ref{t:DCC}.  In Section~\ref{s:outline of proof} Theorem~\ref{t:main} is reduced to four results, proved later in the paper.  In Section~\ref{s:reduction} we prove Theorem~\ref{t:reduction}, which allows us to focus on a restricted class of $3$--manifold groups.   In Section~\ref{s:geo decomp} we consider various splittings of limit groups.    
Later sections are dedicated to the proof of the main technical result, Theorem \ref{t:collapsing Divergent}.

\section{Outline of proofs} \label{s:outline of proof}

In this section we explain the proof of Theorem~\ref{t:main}.  Specifically, we explain the outline of the proof, and also state four results: Theorem~\ref{t:reduction}, Theorem~\ref{t:GH}, Theorem~\ref{t:fingenedges}, and Theorem~\ref{t:collapsing Divergent} (proved in later sections).
Assuming these results, we give
a complete proof of Theorem~\ref{t:main}.  We deduce Theorem~\ref{t:DCC} from Theorem~\ref{t:main}.

Our first reduction uses the basic structure of $3$--manifolds and $3$--manifold groups to restrict the class of $3$--manifolds we need to consider.  We refer to \cite{AFW} for background about $3$--manifolds, their fundamental groups, etc.

\begin{definition} \label{d:gen type}
Let $\MgenM$ be the set of homeomorphism classes of $3$--manifolds $M$ so that:
\begin{enumerate}
\item $M$ is closed, orientable, and irreducible;
\item All geometric pieces are hyperbolic or Seifert-fibered with hyperbolic base orbifold; and
\item The base orbifold of each Seifert-fibered piece is orientable.
\end{enumerate}
Let $\Mgen$ denote the set of isomorphism classes  of fundamental groups of elements of $\MgenM$.
\end{definition}

Throughout this paper, we fix a non-principal ultrafilter $\omega$ on $\mathbb{N}$ and use the concepts of $\omega$--limits, etc.  See \cite{vdD-W} for the background and basic results about ultrafilters, ultralimits, etc.
Recall that a \emph{non-principal} ultrafilter $\omega$ is a finitely additive probability measure $\omega\colon 2^\N\to\{0, 1\}$ such that $\omega(F)=0$ for any finite set $F$. A statement $P(i)$ depending on an index $i$ holds {\em \walmost surely} if $\omega(\{i\;|\; P(i) \text{ holds}\})=1$.

\begin{definition} \label{def:limit group}
Suppose $\mc{G}$ is a family of groups and $G$ is a finitely generated group.  Let $\Hom(G,\mc{G})$ denote the set of all homomorphisms from $G$ to an element of $\mc{G}$.  Let $(\phi_i)$ be a sequence of homomorphisms from $\Hom(G, \mc{G})$. Associated to $(\phi_i)$ is the \emph{stable kernel} $\wker(\phi_i)=\{g\in G \mid \text{ \walmost surely } \phi_i(g)=1 \}$ and a  \emph{$\mc{G}$--limit group} defined by $L:=G/\wker(\phi_i)$. Let $\phi_\infty$ denote the quotient map $G\onto L$. 
We refer to the sequence $(\phi_i)$ as a \emph{defining sequence} of homomorphisms for the limit group $L$.

We say that the sequence $(\phi_i)$ \emph{$\omega$--factors through the limit} if $\phi_i$ \walmost surely factors through $\phi_\infty$.
\end{definition}
The notion of an \emph{equationally noetherian} family of groups was introduced by the first and second authors in \cite[Definition A]{GH}.
That the following definition is equivalent to the one from \cite{GH} is \cite[Theorem 3.6]{GH}.

\begin{definition}\label{def:eqchar}
The family $\mc{G}$ is \emph{equationally noetherian} if for any finitely generated group $G$ there is a finitely presented group $G_0$ and an epimorphism $\alpha \co G_0 \onto G$ so that for every $\Gamma \in \mc{G}$
 the map 
\[	\alpha^* \co \Hom(G,\Gamma) \to \Hom(G_0,\Gamma)	\]
induced by precomposition with $\alpha$ is a bijection.
\end{definition}

Note that with this definition, Theorem~\ref{t:main} is the claim that $\Man$ is equationally noetherian.
In \cite[Theorem 3.7]{GH} a further characterization of equationally noetherian is given, which is the one we use in the proof of Theorem~\ref{t:main}.

\begin{theorem}\label{t:EN def}
The family $\mc{G}$ is equationally noetherian if and only if for every finitely generated group $G$, every sequence from $\Hom(G, \mc{G})$ $\omega$--factors through the limit.
\end{theorem}

Our first reduction towards the proof of Theorem~\ref{t:main} is the following, proved in Section~\ref{s:reduction}.
\begin{restatable}{theorem}{reductiontheorem}\label{t:reduction}
If $\Mgen$ is equationally noetherian then so is $\Man$.
\end{restatable}

Our proof of Theorem~\ref{t:main} proceeds by verifying the criterion from Thereom~\ref{t:EN def}.
One feature a sequence $(\phi_i \co G \to \Gamma_i)$ from $\Hom(G,\Mgen)$ might have is being \emph{\Tdiv} (see Definition~\ref{d:Tdiv} -- roughly speaking the actions of $\phi_i$ on the trees dual to the geometric decompositions diverge).  The main result of \cite{GH} implies the following theorem, proved in Section~\ref{s:geo decomp}.

\begin{restatable}{theorem}{groveshull}\label{t:GH}
Suppose that for every finitely generated group $G$, every sequence from $\Hom(G,\Mgen)$ which is not \Tdiv $\omega$--factors through the limit.  Then for every finitely generated group $G$ \emph{every} sequence from $\Hom(G,\Mgen)$ $\omega$--factors through the limit.
\end{restatable}

Thus, we fix a non-\Tdiv sequence $(\phi_i \co G \to \Gamma_i)$ defining an $\Mgen$--limit group $L$.  The $G$--actions on the trees associated to the $\Gamma_i$ converge to an $L$--action on a (simplicial) tree, inducing the {\em geometric decomposition} of $L$ (Definition~\ref{def:gd}).  A key set of subgroups are \emph{stably parabolic} subgroups (Definition~\ref{def:stably para}).
These are abelian (Lemma~\ref{l:sp abelian}) and edge groups of the geometric decomposition are stably parabolic.  The following is proved in Section~\ref{s:geo decomp}.

\begin{restatable}{theorem}{fingenedges}\label{t:fingenedges}
Let $L$ be an $\Mgen$--limit group defined by a non--\Tdiv sequence $(\phi_i)$ and suppose the edge groups of the geometric decomposition of $L$ with respect to $(\phi_i)$ are finitely generated.  Then $(\phi_i)$ $\omega$--factors through the limit.
\end{restatable}

At this point, the following result (together with Theorems \ref{t:reduction}, \ref{t:GH} and \ref{t:fingenedges}) completes the proof of Theorem~\ref{t:main}.

\begin{restatable}{theorem}{collapsingdivergent}\label{t:collapsing Divergent}
Let $L$ be an $\Mgen$--limit group defined by a non-\Tdiv sequence $(\phi_i)$.  The edge groups of the geometric decomposition of $L$ with respect to $(\phi_i)$ are finitely generated.
\end{restatable}

Our approach to proving Theorem~\ref{t:collapsing Divergent} is to analyze
 the geometric decomposition of $L$ using the \emph{collapsed spaces} defined in Section~\ref{s:collapsing Divergent}.  This leads to the notion of a \emph{\Cdiv} sequence (see Definition~\ref{d:P-divergent}).  In Section~\ref{s:non-divergent} we prove Theorem~\ref{t:NEW non-divergent} (modulo the technical Theorem~\ref{t:fin gen} proved in Appendix~\ref{app:edge-twist}), which deals with the case where the sequence is not \Cdiv.
After build-up in Sections ~\ref{s:collapsing Divergent},~\ref{s:R-trees} and~\ref{s:JSJ}, we prove Theorem~~\ref{t:collapsing Divergent} in Section~\ref{s:res-short}.

We finish by explicitly deducing Theorem~\ref{t:main} from the results in this section.  

\begin{proof}[Proof of Theorem~\ref{t:main}] 
Suppose that $L$ is an $\Mgen$--limit group defined by a non--\Tdiv sequence.  Edge groups of the geometric decomposition of $L$ are finitely generated by Theorem \ref{t:collapsing Divergent}, so by Theorem \ref{t:fingenedges} the defining sequence for $L$ $\omega$--factors through the limit.  Theorem~\ref{t:GH} implies that \emph{all} sequences of homomorphisms to $\Mgen$ $\omega$--factor through the limit, so by Theorem~\ref{t:EN def} $\Mgen$ is equationally noetherian.  Therefore, by Theorem~\ref{t:reduction} $\Man$ is equationally neotherian, as required.
\end{proof}

\subsection{Proof of Theorem~\ref{t:DCC}}

We now deduce Theorem~\ref{t:DCC} from Theorem~\ref{t:main}.
Suppose that $( \Gamma_i )_{i=1}^\infty$ is a sequence from $\Man$, and that for each $i \ge 1$ there is a surjection $\tau_i \co \Gamma_i \to \Gamma_{i+1}$.
Define $(\rho_i)$ from $\Hom(\Gamma_1,\Man)$ by $\rho_i = \tau_{i} \circ \cdots \circ \tau_1\co \Gamma_1 \to \Gamma_{i+1}$.  By Theorem~\ref{t:main} $\Man$ is equationally noetherian, so by Thereom~\ref{t:EN def} $(\rho_i)$ $\omega$--factors through the limit map $\rho_\infty$.

Now, $\ker(\rho_{i+1}) \subseteq \ker(\rho_i)$, so 
\[	\ker(\rho_\infty) = \bigcup_{i=1}^\infty \ker(\rho_i)	.	\]
By Theorem~\ref{t:main} for $\omega$--almost all $j$ the homomorphism $\rho_j$ factors through $\rho_\infty$.  Fix such a $j$.  Then $\ker(\rho_\infty) \subseteq \ker(\rho_j)$, and since $\ker(\rho_j) \subset \ker(\rho_\infty)$, the two are equal.  Thus, the ascending sequence of kernels stabilizes after $\rho_j$, and for all $k \ge j$ the map $\tau_k$ is an isomorphism, completing the proof.

\subsection{Notation and conventions}\label{sec:not_cov}

Throughout this paper a \emph{$\delta$--hyperbolic space} is a geodesic metric space in which all geodesic triangles are \emph{$\delta$--thin} in the sense of \cite[Definition III.H.1.16]{bridhaef:book}.  We remark that if a geodesic metric space has $\nu$--slim triangles (in the sense of \cite[Definition III.H.1.1]{bridhaef:book}) then it has $4\nu$--thin triangles. If $X$ is a geodesic metric space and $x,y \in X$ then we denote any geodesic between $x$ and $y$ by $[x,y]$.

Let $G$ be a group, and $\mc{A}$ and $\mc{H}$ families of subgroups closed under conjugation.  An \emph{$(\mc{A}, \mc{H})$--splitting} of $G$ is an identification of $G$ with the fundamental group of a graph of groups in which all edge groups are in $\mc{A}$ and all groups in $\mc{H}$ are conjugate into vertex groups. When we consider a group as being the fundamental group of a graph of groups, we consider it to come with a specific choice of vertex groups and edge groups, rather than just conjugacy classes arising from an action on a tree. We also consider the underlying graph to have a fixed spanning tree, which, together with the choice of vertex and edge groups, induces a specific Bass-Serre presentation for the fundamental group of the graph of groups. This allows us to identify the fundamental group of any connected sub-(graph of groups) with a subgroup of the fundamental group of the original graph of groups. A graph of groups is {\em minimal} if there are no proper invariant subtrees of the Bass--Serre tree.

Given a sequence of real numbers $(a_i)$, we denote the $\omega$--limit of the sequence by $\wlim a_i$. See \cite{vdD-W} or \cite[Section 4.1]{GH} for precise definitions. If $X_i$ is a sequences of metric spaces with fixed basepoints $o_i\in X_i$, then we denote the ultra-limit of the metric spaces $X_i$ with basepoints $o_i$ by $\wlim(X_i, o_i)$. This ultra-limit is another metric space, see \cite[Section 4.1]{GH} for the precise definition. In all of our applications of ultra-limits of metric spaces, the spaces $X_i$ are $\delta_i$--hyperbolic with $\wlim \delta_i=0$. In these cases, the corresponding ultra-limit is an $\R$--trees by \cite[Lemma 4.1]{GH}.

\section{Reduction to $\Mgen$} \label{s:reduction}
The purpose of this section is to record some basic facts about (limit groups over the family of) $3$--manifold groups, and in particular to prove Theorem~\ref{t:reduction}.

\subsection{Equationally noetherian families}
In this subsection, we record some facts about equationally noetherian families not particular to $3$--manifold groups. The first follows from \cite[Theorem B1]{BMR99} and \cite[Lemma 3.9]{GH}.

\begin{lemma}\label{lem:linear}
For any field $K$ and $n\geq 1$, if $\mc{G}$ is a family of groups so each $\Gamma\in \mc{G}$ is a subgroup of $\operatorname{GL}(n, K)$, then $\mc{G}$ is equationally noetherian. 
\end{lemma}

\begin{lemma}\label{lem:relfplim}\cite[Lemma 3.20]{GH}
Suppose $(\phi_i)$ is a sequence from $\Hom(G, \mc{G})$ and the associated limit group is finitely presented relative to subgroups $\{P_1,..., P_n\}$. Suppose for each $P_j$, there is $\widetilde{P}_j\leq G$ so that $\phi_\infty(\widetilde{P}_j)=P_j$ and that $\phi_i|_{\widetilde{P}_j}$ \walmost surely factors through $\phi_\infty|_{\widetilde{P}_j}$. Then $(\phi_i)$ $\omega$--factors through the limit.
\end{lemma}

Theorem 1 of \cite{BMR} states that if $H$ is a finite index subgroup of $G$ and $H$ is equationally noetherian, then $G$ is equationally noetherian. We generalize this to a family of groups.

\begin{definition}
Let $\mc{G}$ be a family of groups and let $n \ge 1$. Then $\mc{G}^{\le n}$ is the family of groups containing a subgroup from $\mc{G}$ of index at most $n$.
\end{definition}

\begin{proposition} \label{prop:fi}
Suppose that $\mc{G}$ is an equationally noetherian family of groups.  For any $n \ge 1$ the family $\mc{G}^{\le n}$ is equationally noetherian.
\end{proposition}

\begin{proof}
Let $G$ be finitely generated and let $(\phi_i\colon G\to \Gamma_i)$ be a sequence with $\Gamma_i\in\mc{G}^{\le n}$. By Thereom~\ref{t:EN def} it suffices to show that \walmost surely $\ker(\phi_\infty)\subseteq \ker(\phi_i)$.
So, let $\Gamma_i^\prime \in \mc{G}$ be so that $|\Gamma_i : \Gamma_i^\prime| \le n$. Let $H$ be the intersection of all subgroups of $G$ of index at most $n$.  Then $\phi_i(H) \subseteq \Gamma_i^\prime$, so $\phi_i|_H\in \Hom(H, \mc{G})$. Since $\mc{G}$ is equationally noetherian by Thereom~\ref{t:EN def} \walmost surely $\ker(\phi_\infty|_H)\subseteq \ker(\phi_i|_H)$.

Let $g\in G$. If $\phi_\infty(g)\notin\phi_\infty(H)$, then for any $h\in H$ $gh\notin\ker(\phi_\infty)$. Suppose now that $\phi_\infty(g)\in\phi_\infty(H)$. Fix $k\in H$ such that $\phi_\infty(g)=\phi_\infty(k)$. Now, \walmost surely $gk^{-1}\in\ker(\phi_i)$ and  $\ker(\phi_\infty|_H)\subseteq \ker(\phi_i|_H)$. For any such $i$ and for any $h\in H$, if $gh\in\ker(\phi_\infty)$ then $kh\in\ker(\phi_\infty)\subseteq \ker(\phi_i)$. Hence $gh=gk^{-1}kh\in\ker(\phi_i)$.
 Putting together these two cases, we see that for any $g \in G$ \walmost surely $\ker(\phi_\infty)\cap gH\subseteq \ker(\phi_i)$. Repeating this argument over a finite of coset representatives of $H$ in $G$ proves that \walmost surely $\ker(\phi_\infty)\subseteq \ker(\phi_i)$.
\end{proof}

\begin{remark}
The bound on the index in Proposition~\ref{prop:fi} is essential.  Indeed, it is straightforward to see that if one considers the family $\mc{G}^{\mathrm{fin}}$ of groups which have a finite index subgroup contained in $\mc{G}$ then $\mc{G}$ may be equationally noetherian while $\mc{G}^{\mathrm{fin}}$ is not.  This is true, for example, in case $\mc{G}$ is the family consisting of the trivial group, since then $\mc{G}^{\mathrm{fin}}$ is the family of finite groups, which is not equationally noetherian \cite[Example 3.15]{GH}.
\end{remark}

\begin{definition}
If $\mc{G}$ is a family of groups, then let $\mc{G}^\ast$ be the family of all free products of elements of $\mc{G}$.
\end{definition}

\begin{theorem}\cite[Cor. C]{GH} \label{t:free products}
If $\mc{G}$ is equationally noetherian then so is $\mc{G}^\ast$.
\end{theorem}

\subsection{Reduction to closed, orientable, irreducible $3$--manifolds}

Recall that $\Man$ is the set of all fundamental groups of all $3$--manifolds. Recall that we are trying to prove Theorem~\ref{t:reduction}, so we are interested in homomorphisms from a finitely generated group $G$ to elements of $\Man$.  Of course, all such homomorphisms have finitely generated image, and since $\Man$ is closed under passing to subgroups there is no harm in restricting to the finitely generated elements of $\Man$.  See \cite[Lemma 3.11]{GH}.

By Scott's Theorem \cite[Theorem 2.1]{Scott:Coherence}, every finitely generated $3$--manifold group is in fact finitely presented, and moreover is the fundamental group of a \emph{compact} $3$--manifold (see also \cite{Scott:Coherence} or \cite{Scott:core}).  Doubling along the boundary embeds any compact $3$--manifold with boundary into a closed $3$--manifold as a retract.  Since it is a retract, the fundamental group of the manifold with boundary injects into the fundamental group of the closed $3$--manifold.  Therefore, on the level of groups, we may assume that any homomorphism from $G$ to an element of $\Man$ is in fact a homomorphism to the fundamental group of a closed $3$--manifold.  In summary, we have the following,  where $\mc{M}_{\mathrm{cl}}$ is the set of fundamental groups of closed $3$--manifolds.

\begin{theorem} \label{th:closed}
 If $\mc{M}_{\mathrm{cl}}$ is equationally noetherian then so is $\Man$.
\end{theorem}

We now move on to consider how to prove $\mc{M}_{\mathrm{cl}}$ is equationally noetherian, by considering the structure of closed $3$--manifolds.  We first make another reduction to the closed, orientable, irreducible case, and then we consider the geometric decomposition of these $3$--manifolds.

Here are some pertinent classes of $3$--manifolds.

\begin{definition}
Let $\Mori$ be set of the fundamental groups of closed, orientable $3$--manifolds.
Let $\Mirr$ be the set of fundamental groups of closed, orientable and irreducible $3$--manifolds.
Let $\Mgeo$ be the set of fundamental groups of compact geometric $3$--manifolds with all boundary components tori.
\end{definition}

\begin{lemma} \label{c:red to ori}
If $\Mori$ is equationally noetherian then so is $\Man$.
\end{lemma}
\begin{proof}
By Theorem~\ref{th:closed}, it suffices to consider only closed manifolds. 
For each closed 3--manifold $M$, $\pi_1(M)$ has an index at most $2$ subgroup belonging to $\Mori$, so the result follows from Proposition~\ref{prop:fi}.
\end{proof}

Each closed, orientable 3--manifold $M$ has a unique Milnor--Kneser decomposition which induces a free product decomposition of $\pi_1(M)$:
\[
\pi_1(M)\cong\pi_1(M_1)\ast...\ast\pi_1(M_k)\ast F_n
\]
where each $\pi_1(M_i) \in \Mirr$ (and $F_n$ is a free group of rank $n$). 
The following is thus an immediate consequence of Theorem~\ref{t:free products}
\begin{corollary} \label{c:irr EN then all EN}
If $\Mirr$ is equationally noetherian then so is $\Mori$.
\end{corollary}

\subsection{The geometric decomposition} \label{ss:geometric}
If $M$ is a closed, orientable, irreducible $3$--manifold then the Geometrization Theorem (see \cite{Per02, Per03a, Per03b, CaoZhu, KleinerLott, MorganTian}) gives a (possibly trivial) decomposition of $M$ along incompressible tori into geometric pieces. This induces a splitting of $\pi_1(M)$ which we call the \emph{geometric decomposition}.  
We refer to \cite{AFW}, \cite{scott:geom}, or \cite{thurston:notes} for the definition of a geometric $3$--manifold, and to \cite{scott:geom} for details about each of the $8$ Thurston geometries.

Assume $M$ is a closed, orientable, irreducible 3--manifold and also that $M$ is not a geometric 3--manifold (in particular, $M$ is not a torus bundle over a circle). Then each piece of the geometric decomposition of $M$ is either hyperbolic, Seifert-fibered with hyperbolic base orbifold, or a twisted $I$--bundle over a Klein bottle.  Replacing $M$ by a double cover as in \cite[Lemma 2.4]{WiltonZalesskii} we may assume there are no twisted $I$--bundles over a Klein bottle. By a similar argument, there is a double cover so all the base orbifolds of Seifert-fibered pieces with hyperbolic base orbifold are orientable.  Thus:

\begin{lemma} \label{l:cover to become gen}
Let $M$ be a closed, orientable, irreducible, non-geometric $3$--manifold.  Then $M$ has a cover of degree at most $4$ which lies in $\MgenM$.
\end{lemma}

An important property of elements of $\Mgen$ is the following result, which follows immediately from the proof of \cite[Lemma 2.4]{WiltonZalesskii}.  Recall a group action on a tree $T$ is \emph{$k$--acylindrical} if the stabilizer of any segment in $T$ of length at least $k+1$ is trivial.

\begin{lemma} \label{l:2-acyl}
 Suppose  $\Gamma \in \Mgen$ and that $T$ is the Bass--Serre tree dual to the geometric decomposition of $\Gamma$. The $\Gamma$--action on $T$ is $2$--acylindrical.
\end{lemma}

\begin{definition}
    Suppose that $\Gamma \in \Mgen$ and that $T$ is the Bass--Serre tree dual to the geometric decomposition of $\Gamma$.  A vertex group of the associated graph of groups decomposition of $\Gamma$ is of \emph{hyperbolic type} if the associated sub-manifold in the geometric decomposition is hyperbolic, and of \emph{\LSF} if the associated sub-manifold is of Seifert-fibered type (with orientable hyperbolic base orbifold).
\end{definition}

\subsection{Reduction to $\Mgen$}

Consider the case $\Mlsf$ of fundamental groups of Seifert--fibered manifolds with orientable hyperbolic base orbifold.  Let $\orb$ be the set of fundamental groups of orientable hyperbolic $2$--orbifolds of finite-type.  Each $\Gamma \in \Mlsf$ admits a short exact sequence which is a central extension:
\[	1 \to \Z \to \Gamma \to \Gamma_B \to 1	,	\]
for some $\Gamma_B \in \orb$ (the fundamental group of the base orbifold).  From this, the following result follows quickly.

\begin{lemma} \label{lem:central}
Let $L$ be an $\Mlsf$--limit group. Then there is a central extension
\[	1 \to A \to L \to B \to 1	,	\]
where $A$ is abelian and $B$ is an $\orb$--limit group.
\end{lemma}

\begin{definition} \label{d:A-slender}
 A group $G$ is {\em $\mc{A}$--slender} if every abelian subgroup of $G$ is finitely generated.
\end{definition}

Our analysis of $\Mlsf$--limit groups is based on combining Lemma \ref{lem:central} with the following theorem of the third author.

\begin{theorem}\cite[Theorem 1.2]{Liang}\label{thm:Liang}
Let $B$ be an $\orb$--limit group. Then $B$ is finitely presented and $\mc{A}$--slender.
\end{theorem}

\begin{corollary}\label{cor:SF limit fp+A-slender}
Let $L$ be an $\Mlsf$--limit group. Then $L$ is finitely presented and $\mc{A}$--slender.
\end{corollary}

\begin{proof}
Consider the central extension from Lemma~\ref{lem:central}. Since $B$ is finitely presented by Theorem \ref{thm:Liang}, $A$ is finitely generated, for example by \cite[Lemma 12.1]{GMW}.  It is now easy to see $L$ is finitely presented. The fact that $A$ is finitely generated and abelian subgroups of $B$ are finitely generated implies that all abelian subgroups of $L$ are finitely generated.
\end{proof}

\begin{lemma}\label{l:abelian orb implies SF}
Let $M$ be a Seifert fibered $3$--manifold with orientable hyperbolic base orbifold $O$. Let $K$ be the kernel of the natural quotient map $\pi_1(M) \to \pi_1(O)$. Let $g, h\in \pi_1(M)$. If $[g, h]\in K$, then $[g, h]=1\in \pi_1(M)$. 
\end{lemma}
\begin{proof}
All abelian subgroups of $\pi_1(O)$ are cyclic, so the images of $g$ and $h$ in $\pi_1(O)$ generate a cyclic subgroup.  A central extension of a cyclic group is abelian. 
\end{proof}

The following result follows quickly from the fact that finite subgroups of elements of $\orb$ are cyclic.

\begin{lemma} \label{lem:fin cyclic}
Let $B$ be an $\orb$--limit group.  Any finite subgroup of $B$ is cyclic.
\end{lemma}

The following result is now immediate from Lemmas~\ref{l:abelian orb implies SF} and \ref{lem:fin cyclic}.

\begin{corollary} \label{lem:ab}
Suppose that $L$ is an $\Mlsf$--limit group, and let $1 \to A \to L \to B \to 1$ be as in Lemma~\ref{lem:central}.  Suppose  that $E \le L$ has finite image in $B$.  Then $E$ is abelian.
\end{corollary}

\begin{lemma}\label{JSJ properties}
Let $\Gamma \in \Mgen$, and let $\Gamma_e$ be an edge group of the geometric decomposition of $\Gamma$, with adjacent vertex groups $\Gamma_v$, $\Gamma_w$.  The pre-images in $\Gamma_e$ of the centers $Z(\Gamma_v)$ and $Z(\Gamma_w)$ intersect trivially.
\end{lemma}
\begin{proof}
The only way that $Z(\Gamma_v)$ and $Z(\Gamma_w)$ can both be nontrivial is if the corresponding pieces are both Seifert--fibered.  The centers of adjacent Seifert-fibered pieces intersect trivially, since otherwise the connecting torus is not part of the characteristic sub-manifold.
\end{proof}

\begin{theorem} \label{t:geo EN}
$\Mgeo$ is an equationally noetherian family of groups.
\end{theorem}
\begin{proof}
It is a straightforward observation that any finite union of equationally noetherian families is equationally noetherian \cite[Lemma 3.8]{GH}. Hence it suffices to deal with each geometry individually.

Suppose first that $M$ supports one of the following geometries:  $\mathbb{S}^3$, $\mathbb{S}^2 \times \R$, $\mathbb{E}^3$, $\operatorname{NIL}$ or $\operatorname{SOL}$.  Then (as, for example, explained in \cite[$\S7-11$]{GMW}) $\pi_1(M)$ has a subgroup of index at most $240$ which is either cyclic (finite or infinite), $\Z^3$, or some semi-direct product $\Z^2 \rtimes \Z$.  Each of these groups embeds in $\operatorname{SL}(3, \Z)$.  Therefore, the union of the set of fundamental groups of manifolds supporting these five geometries in equationally noetherian by Proposition~\ref{prop:fi} and Lemma~\ref{lem:linear}.

The fundamental group of any hyperbolic $3$--manifold is (isomorphic to) a subgroup of $\operatorname{PSL}(2, \bC)$ (and hence also contained in $\operatorname{SL}(3, \bC)$ -- see \cite[Lemma 12.9]{GMW}). Thus, the family of fundamental groups of hyperbolic 3--manifolds is equationally noetherian by Lemma~\ref{lem:linear}.

The two remaining geometries correspond to \emph{Seifert fibered manifolds with hyperbolic base orbifolds}.  Since we can pass to an index $2$ subgroup to lie in $\Mlsf$, by Proposition~\ref{prop:fi} it suffices to consider $\Mlsf$--limit groups.  If $L$ is an $\Mlsf$--limit group then $L$ is finitely presented by~\ref{cor:SF limit fp+A-slender}. It follows that the sequence of maps defining $L$ $\omega$--factors through the limit (see Lemma~\ref{lem:relfplim}). Thus the family $\Mlsf$ is equationally noetherian by Thereom~\ref{t:EN def}.
\end{proof}

Summarizing the above discussion and the previous results in the section, we have the following result from Section~\ref{s:outline of proof}.

\reductiontheorem*
\begin{proof}
Suppose that $\Mgen$ is equationally noetherian.  By Lemmas~\ref{c:red to ori} and~\ref{c:irr EN then all EN}, in order to prove that $\Man$ is equationally noetherian it is enough to prove that $\Mirr$ is.  By Theorem~\ref{t:geo EN}, $\Mgeo$ is equationally noetherian, so it is enough to consider closed, orientable, irreduciable, non-geometric $3$--manifolds.  By Lemma~\ref{l:cover to become gen} any such $3$--manifold has a degree at most $4$ cover which lies in $\MgenM$.  Therefore, under the assumption that $\Mgen$ is equationally noetherian, we see that $\Man$ is equationally noetherian, as required.
\end{proof}

\section{Generalized geometric decompositions} \label{s:geo decomp}

In order to prove Theorem~\ref{t:main}, it remains to prove Theorems~\ref{t:GH}, ~\ref{t:fingenedges} and~\ref{t:collapsing Divergent}.  In this section, we provide the setting and some basic definitions, and prove Theorems~\ref{t:GH} and~\ref{t:fingenedges}.

We fix the following setup.
Let $G = \langle S \rangle$ be a finitely generated group, and let 
$\left( \phi_i\colon G\to \Gamma_i\right)$ be a sequence from $\Hom(G, \Mgen)$. Let $L=G/\wker(\phi_i)$ be the $\Mgen$--limit group associated to $(\phi_i)$, with associated limit map $\phi_\infty \co G \onto L$  (see Definition~\ref{def:limit group}). 

\begin{definition}
Suppose that $G, (\phi_i)$ and $L$ are as above, and that $g \in L$.  An \emph{\wapprox}to $g$ is a sequence $(g_i)$ so that each $g_i \in \Gamma_i$ and there is an element $\widetilde{g} \in G$ so that $\phi_\infty(\widetilde{g}) = g$ and for each $i$ we have $\phi_i(\widetilde{g}) = g_i$.

Similarly, if $F \subset L$ is a finite ordered set then an \emph{\wapprox}to $F$ is a sequence of (ordered) tuples $(F_i)$ from $\Gamma_i$ so that for some lift $\widetilde{F}$ of $F$ to $G$ (so $\phi_\infty$ restricts to an ordered bijection from $\widetilde{F}$ to $F$) we have $\phi_i(\widetilde{F}) = F_i$ (again, as an ordered bijection).
\end{definition}

We often apply this definition to a finite set $F$ without explicitly mentioning the ordering. The ordering in these cases can be chosen arbitrarily, the point is that for each $g\in F$ and each $i$, there is a fixed $g_i\in F_i$ such that $(g_i)$ is an \wapprox to $g$.

The following lemmas are immediate.
\begin{lemma}
Suppose that $g \in L$ and that $(g_i)$ and $(g_i')$ are \wapproxs to $g$.  Then \walmost surely $g_i = g_i'$.
Similarly, if $F \subset L$ is finite and $(F_i)$, $(F_i')$ are \wapproxs to $F$ then \walmost surely $F_i = F_i'$.
\end{lemma}

\begin{lemma}\label{l:appro relation}
Suppose $g^1, \dots, g^s\in L$ and for each $1\leq j\leq s$, $(g_i^j)$ is an \wapprox to $g^j$. For any word $w(x_1, \dots, x_s)$ in $\{ x_i^\pm \}$ we have $w(g^1,\ldots , g^s) = 1$ in $L$ if and only if for \walmost all $i$ $w(g_i^1,\ldots , g_i^s) = 1$ in $\Gamma_i$.
\end{lemma}

We are only concerned with properties of \wapproxs that hold \walmost surely for terms in the sequence, and so it is always irrelevant which \wapprox we choose for an element (or finite subset) of $L$.  We use this observation frequently without mention.

Let $S$ be a finite generating set for $G$ and $S_i=\phi_i(S)$. Let $T_i$ be the tree dual to the geometric decomposition of $\Gamma_i$. Let
\[	\| \phi_i \|_{T_i} = \inf_{t \in T_i} \max_{s \in S_i} d_{T_i}\left( t, s.t \right) 	.	\]

\begin{definition}\label{d:Tdiv}
The sequence $\left( \phi_i \right)$ is \emph{\Tdiv} if $\wlim \| \phi_i \|_{T_i} = \infty$.
\end{definition}

We are now ready to prove the following theorem from Section~\ref{s:outline of proof}.

\groveshull*

\begin{proof}
The action of any $\Gamma \in \Mgen$ on its geometric tree is $2$--acylindrical (see Lemma~\ref{l:2-acyl}). In particular, $\Mgen$ is (in the terminology of \cite{GH}) a uniformly acylindrical family of groups (see, for example, \cite[p.7142]{GH}).  Moreover, any sequence from $\Hom(G,\Mgen)$ which is not \Tdiv is `non-divergent' in the sense of \cite{GH}.  Thus, Theorem~\ref{t:GH} is an immediate consequence of \cite[Theorem B]{GH}.
\end{proof}

It now remains to prove Theorems~\ref{t:fingenedges} and~\ref{t:collapsing Divergent}.  Thus, we henceforth assume that the defining sequence $(\phi_i)$ is not \Tdiv.

\begin{definition} \label{def:stably para}
Let $g \in L$, and let $(g_i)$ be an \wapprox to $g$.  Then $g$ is \emph{stably parabolic with respect to $(\phi_i)$} if \walmost surely $g_i$ fixes an edge in $T_i$.
A subgroup $H \le L$ is \emph{stably parabolic with respect to $(\phi_i)$} if for any finite $F \subset H$ and any \wapprox $(F_i)$ of $F$ \walmost surely there is an edge in $T_i$ fixed by each $f \in F_i$.
Given $V \le L$, the set of stably parabolic subgroups of $V$ with respect to $(\phi_i)$ is $\mc{H}_{V,(\phi_i)}$, or just $\mc{H}_V$ when $(\phi_i)$ is implied.
\end{definition}

Since each subgroup of $\Gamma_i$ which is the stabilizer of an edge in $T_i$ is isomorphic to $\Z^2$, the following result is immediate from Lemma \ref{l:appro relation}.

\begin{lemma} \label{l:sp abelian}
Any stably parabolic subgroup of $L$ is abelian.
\end{lemma}

\begin{definition}\label{def:sub_geom}
A subgroup $H$ of $L$ is \emph{$\omega$--geometric} if for any finite subset $F \subset H$ and any \wapprox $(F_i)$ of $F$, \walmost surely there is a vertex $v_i$ of $T_i$ so that $F_i$ fixes $v_i$. Note that $v_i$ is unique when $H$ is non-abelian. When $H$ is abelian and $F_i$ fixes multiple vertices of $T_i$, we choose one and denote it by $v_i$.  
\end{definition}

\begin{terminology}
Let $H$ be an $\omega$--geometric subgroup of $L$ and $v_i$ be as in Definition~\ref{def:sub_geom}.  Then either
\begin{enumerate}
\item \walmost surely $v_i$ is of hyperbolic type; or
\item \walmost surely $v_i$ is of \LSF.
\end{enumerate}
In the first case $H$ is \emph{hyperbolic-type} and in the second $H$ is \emph{\LSF}. 
\end{terminology}

\begin{definition} \label{def:stable center}
Let $H$ be an $\omega$--geometric subgroup of $L$, with associated sequence $v_i$ of vertices of $T_i$, and let $\Gamma_{i,v_i}$ be the associated vertex group of $\Gamma_i$.  The {\em stable center} of $H$, denoted $\SC(H)$, is the set of $g \in H$ so that for any \wapprox $(g_i)$ of $g$, \walmost surely $g_i \in Z(\Gamma_{i,v_i})$.
\end{definition}
Observe that $\SC(H)$ is a subgroup and that if $H$ is a hyperbolic-type $\omega$--geometric subgroup of $L$ then $\SC(H) = \{ 1 \}$.  Observe also that all elements of the stable center are stably parabolic.

As we see in Definition~\ref{def:gd} below, $L$ admits a splitting arising from a limiting action on the trees $T_i$.  This geometric decomposition is the natural splitting associated to $L$, but for technical reasons in Section~\ref{s:res-short} we need a more general class of splittings of $L$.

\begin{definition}\label{def:ggd}
A minimal abelian splitting $\Gd$ of $L$ is a {\em generalized geometric decomposition (or GGD) with respect to $(\phi_i)$} if 
\begin{enumerate}
\item\label{ggd:vertices} Each vertex group of $\Gd$ is $\omega$--geometric.
\item\label{ggd:edges} For any adjacent vertices $v$ and $w$ of $\Gd$ with sequences $(v_i)$ and $(w_i)$ as in the definition of $\omega$--geometric (Definition \ref{def:sub_geom}), \walmost surely $v_i$ and $w_i$ are adjacent in $T_i$.
\item\label{ggd:sp} $\Gd$ is an $\left(\mc{H}_{L,(\phi_i)} , \mc{H}_{L,(\phi_i)} \right)$--splitting.  That is to say, the edge groups are stably parabolic, and all stably parabolic subgroups are elliptic.
\item For each \LSF vertex group $\Gd_v$ of $\Gd$, $\SC(\Gd_v)$ is contained in each edge group of $\Gd$ adjacent to $\Gd_v$.
\end{enumerate}
\end{definition}

For the rest of this subsection, fix the above notation, and let $\Gd$ be a GGD of $L$ with respect to $(\phi_i)$.  The sequences $(v_i)$ and $(w_i)$ in Definition~\ref{def:ggd} are implicitly fixed.

\begin{lemma}\label{l:properties of stable center}
If $\Gd_v$ is an \LSF vertex group of $\Gd$ then $\SC(\Gd_v)$ is central in $\Gd_v$.  
\end{lemma}
\begin{proof}
Let $g\in \SC(\Gd_v)$ and $h\in \Gd_v$.  If $(g_i)$ and $(h_i)$ are \wapproxs to $g$ and $h$, respectively then \walmost surely there is an \LSF vertex group $\Gamma_{i,v_i}$ of $\Gamma_i$ so that $g_i \in Z(\Gamma_{i,v_i})$ and $h_i \in \Gamma_{i,v_i}$.  Therefore \walmost surely $[g_i,h_i] = 1$ in $\Gamma_i$.  By Lemma~\ref{l:appro relation} $[g,h] = 1$ in $L$. 
\end{proof}

Let $V$ be a vertex group of $\Gd$.  By Lemma~\ref{l:properties of stable center}, $\SC(V) \unlhd V$, and we define $\overline{V} = V / \SC(V)$. Let $\pi: V \rightarrow \overline{V}$ be the quotient map. 

\begin{lemma} \label{l:preimage abelian}
Let $V$ be an \LSF vertex group of $\Gd$. If $H$ is an abelian subgroup of $\overline{V}$ then $\pi^{-1}(H)$ is an abelian subgroup of $V$. 
\end{lemma}
\begin{proof}
Let $\bar g, \bar h\in H$, let $g\in \pi^{-1}(\bar g)$ and $h\in \pi^{-1}(\bar h)$, and let $(g_i)$ and $(h_i)$ be \wapproxs to $g$ and $h$, respectively. Since $[\bar g, \bar h] = 1$, we have $[g, h]\in \SC(V)$, and so \walmost surely $[g_i, h_i]\in Z(\Gamma_{v_i})$. Since $Z(\Gamma_{v_i})$ is the kernel of $p_i: \Gamma_{v_i} \twoheadrightarrow \pi_1(O_i)$, where $O_i$ is the base orbifold of the Seifert-fibered manifold corresponding to $\Gamma_{v_i}$,  by Lemma~\ref{l:abelian orb implies SF}, we have $[g_i, h_i]=1$ \walmost surely.  By Lemma~\ref{l:appro relation} $[g, h]=1$, as required. 
\end{proof}

Let $V$, $\overline{V}$ be as above, let $\mc{E}$ be the set of edge groups of $\Gd$ adjacent to $V$, and $\overline{\mc{E}}$ the image of $\mc{E}$ in $\overline{V}$.  Since $L$ is finitely generated, $V$ is finitely generated relative to $\mc{E}$. Hence $\overline{V}$ is finitely generated relative to $\overline{\mc{E}}$. Therefore, by Corollary~\ref{cor:lindecomp}, $\overline{V}$ admits a graph of groups decomposition $\overline{\mathbb{L}}$ rel $\overline{\mc{E}}$ so that all edge groups of $\overline{\mathbb{L}}$ have cardinality at most $4$, and no vertex group of $\overline{\mathbb{L}}$ splits rel $\overline{\mc{E}}$ over a subgroup of size at most $4$.  We call $\overline{\mathbb{L}}$ the \emph{relative $4$--Linnell decomposition of $\overline{V}$ rel $\overline{\mc{E}}$}.
Notice that by Corollary~\ref{lem:ab} the preimages in $V$ of edges in $\overline{\mathbb{L}}$ are abelian.

\begin{definition} \label{def:refined ggd}
Let $V$ be an \LSF vertex group of $\Gd$.  Let $\overline{V} = V / \SC(V)$.  Let 
$\overline{\mathbb{L}}$ be the relative $4$--Linnell decomposition of $\overline{V}$ relative to (images of) edge groups of $\Gd$ adjacent to $V$, and $\mathbb{L}$ the induced abelian splitting of $V$ relative to adjacent edge groups.  The {\em Linell refinement of $\Gd$}, denoted $\Gdf$, is a refinement of $\Gd$ obtained by replacing each \LSF vertex group $V$ by $\mathbb{L}$.
\end{definition}

\begin{remark}
Suppose that $V$ is a vertex group of $\Gd$.  The edge groups of $\Gd$ adjacent to $V$ are elliptic the relative $4$--Linnell decomposition $\mathbb{L}$ of $V$ by definition.  However, such an edge group might not be conjugate into a unique vertex group of $\mathbb{L}$.  Therefore, there is \emph{a} refinement $\Gdf$ of $\Gd$ as described in Definition~\ref{def:refined ggd}, but it might not be unique.  None of our constructions or results in this paper depend on which refinement is chosen, so we henceforth ignore this ambiguity.
\end{remark}

Observe that the hyperbolic type vertex groups of $\Gd$ and $\Gdf$ are the same.

\begin{definition}\label{def:good generating set}
Suppose $\V$ is a vertex group of $\Gdf$.  A {\em good relative generating set} for $\V$ is a finite set $A$ so that
\begin{enumerate}
\item $A$ together with the adjacent edge groups generates $\V$; and
\item For each edge group $E\supsetneq \SC(V)$ adjacent to $V$, there exists  $a\in A \cap E\setminus\SC(V)$.
\end{enumerate}
\end{definition}

The existence of good relative generating sets is clear.

\subsection{The geometric decomposition}

Let $( \phi_i \co G \to \Gamma_i)$ be a sequence from $\Hom(G,\Mgen)$ which is not \Tdiv and let $T_i$ be the minimal $\phi_i(G)$--invariant subtree of the tree dual to the geometric splitting of $\Gamma_i$.
Let $o_i$ be a point in $T_i$ so $\max\limits_{s\in S} d_{T_i}(o_i, \phi_i(s)o_i)=\|\phi_i\|_{T_i}$.

Let $T_\infty = \wlim (T_i, o_i)$. Since each $T_i$ is a tree in which edges have length one, $T_\infty$ is a simplicial tree. Moreover, the actions of $G$ on $T_i$ induce a $2$--acylindrical action of $L$ on $T_\infty$; see \cite[Proposition 6.1]{GH}.   

\begin{definition}\label{def:gd}
Let $G$, $S$, $\left( \phi_i \right)$, $T_i$, $L$ and $T_\infty$ be as above.  The {\em geometric tree} of $L$, denoted $T_\geom$, is the minimal $L$--invariant subtree of $T_\infty$.

The \emph{geometric decomposition} of $L$ is the splitting dual to $T_\geom$.
\end{definition}

It is straightforward to check that the geometric decomposition of $L$ is in fact a GGD in the sense of Definition \ref{def:ggd}, which we record in the following lemma.
\begin{lemma}\label{l:gd is ggd}
 Let $L$ be an $\Mgen$--limit group defined by a sequence $( \phi_i \co G \to \Gamma_i)$ (so $L = G / \wker(\phi_i)$).  If $(\phi_i)$ is not \Tdiv then the geometric splitting of $L$ is a GGD with respect to $(\phi_i)$.
\end{lemma}

We now prove the following result from Section~\ref{s:outline of proof}.

\fingenedges*
\begin{proof}
Let $\Gd$ be the geometric decomposition of $L$.  The edge groups of $\Gd$ are finitely generated, so the vertex groups are also and hence each vertex group of $\Gd$ is an $\Mgeo$--limit  group. By Theorem~\ref{t:geo EN}, the defining sequence for each vertex group $\omega$--factors through the limit.  The edge groups of $\Gd$ are finitely generated and abelian, so $L$ is finitely presented relative to the vertex groups of $\Gd$, and by Lemma~\ref{lem:relfplim} the defining sequence for $L$ $\omega$--factors through the limit, as required.
\end{proof}

Given Theorem~\ref{t:fingenedges}, our goal is to prove that the edge groups of $\Gd$ are finitely generated. One way to achieve this is to show that the following question has a positive answer.

\begin{question}
Suppose the edge groups in an acylindrical splitting of a finitely generated group are abelian. Moreover, all vertex groups of this splitting have the property that their finitely generated subgroups are $\mc{A}$--slender. Are the edge groups of such a splitting always finitely generated? 
\end{question}

However it is not clear to us how to answer the above question. Hence we obtain more information about the geometric decomposition of $L$ before we show that its edge groups are finitely generated.

\subsection{Assumptions}

It remains to prove Theorem~\ref{t:collapsing Divergent}, and the remainder of the paper is dedicated to its proof.  Let $G_0$ be a finitely presented group and $\pi: G_0\rightarrow G$ be a surjective map. Then  the sequence of homomorphisms $\{\phi_i\circ \pi\}$ defines the same limit group $L$ as $\{\phi_i\}$. Since Theorem~\ref{t:collapsing Divergent} is about properties of $L$, we can assume $G$ to be finitely presented in its proof without loss of generality.  
For the remainder of the paper we make the following assumptions, which are key for the proof of Theorem~\ref{t:collapsing Divergent}.

\begin{assumption} \label{ass:fixed things}
Let $G$ be a finitely presented group with finite generating set $S$.  Let $(\phi_i \co G \to \Gamma_i = \pi_1(M_i))$ be a sequence from $\Hom(G,\Mgen)$ and let $T_i $ be the minimal $\phi_i(G)$--invariant subtree of the geometric tree associated to $M_i$.  Suppose $(\phi_i)$ is not \Tdiv.
Let $L = G / \wker(\phi_i)$ be the $\Mgen$--limit group, and $\phi_\infty \co G \onto L$ the limit map.  Let  $\Gd$ be a GGD of $L$ wrt $(\phi_i)$ and $T_\Gd$ the Bass--Serre tree.
Let $\Gdf$ be the Linnell refinement of $\Gd$ (Definition~\ref{def:refined ggd}).  Fix a vertex group $\V$ of $\Gdf$ associated to the vertex $v$, and a good relative generating set $A$ of $V$
(Definition~\ref{def:good generating set}).  Let $(A_i)$ be an \wapprox to $A$.  Denote by $\mc{H}_V$ the set of stably parabolic subgroups of $V$.
\end{assumption}

\subsection{Outline of the proof of Theorem~\ref{t:collapsing Divergent}}

As explained above, in order to prove Theorem~\ref{t:main}, it remains to prove Theorem~\ref{t:collapsing Divergent}.
In this subsection, we provide an outline of the proof of Theorem~\ref{t:collapsing Divergent}.

First, recall the statement:

\collapsingdivergent*

Recall from Definition~\ref{def:ggd} and the fact that the geometric decomposition is a GGD that the edge groups are stably parabolic.  We prove that all stably parabolic subgroups are finitely generated.  Stably parabolic subgroups were defined in Definition~\ref{def:stably para} and in Lemma~\ref{l:sp abelian} we observed that stably parabolic subgroups are abelian.  Therefore, in order to prove that stably parabolic subgroups are finitely generated, it suffices to embed them in a finitely generated abelian group. This is achieved via the combination of the following two results.  Recall that in Standing Assumption~\ref{ass:fixed things} we fixed a defining sequence $(\phi_i : G \to \Gamma_i)$ of the $\Mgen$--limit group $L$ and a GGD $\Gd$ of $L$ with respect to $(\phi_i)$.  In Definition~\ref{d:P-divergent} below we define what it means for $(\phi_i)$ to be \Cdiv with respect to $\Gd$.  However, for the logical structure of the proof of Theorem~\ref{t:collapsing Divergent}, the definition is not important, and we just need the dichotomy that a defining sequence either is or is not \Cdiv with respect to a given GGD.  Theorem~\ref{t:NEW non-divergent} is proved in Section~\ref{s:non-divergent}, modulo the technical result Theorem~\ref{t:fin gen}, proved in Appendix~\ref{app:edge-twist}.
Theorem~\ref{t:technical outline} is proved in Section~\ref{s:res-short}, building on the work in Sections~\ref{s:R-trees} and~\ref{s:JSJ}.

\begin{restatable}{theorem}{nondivtheorem} \label{t:NEW non-divergent}
Let $L$ be an $\Mgen$--limit group defined by a non-\Tdiv sequence $(\phi_i)$.  If there is a GGD $\Gd$ of $L$ with respect to which $(\phi_i)$ is not \Cdiv then all stably parabolic subgroups of $L$ are finitely generated.
\end{restatable}

\begin{restatable}{theorem}{technical} \label{t:technical outline}
Suppose that $L$ is an $\Mgen$--limit group whose defining sequence of homomorphisms is not \Tdiv.  There exists $k\geq 0$ and a sequence of epimorphisms:
\[	L=S_0 \overset{\eta_1}{\onto} S_1 \overset{\eta_2}{\onto} S_2 \overset{\eta_3}{\onto} \cdots \overset{\eta_k}{\onto} S_k\overset{\eta_{k+1}}{\onto} S_{k+1}	,	\]
so that
\begin{enumerate}
\item Each $S_i$ is an $\Mgen$--limit group;
\item For each $1\leq i\leq k$, the map $\eta_i : S_{i-1} \onto S_i$ is a proper quotient map and $\eta_{k+1}$ is an isomorphism;
\item\label{item:sp inject} For each $1\leq i\leq k+1$, $\eta_i$ injectively maps the stably parabolic subgroups of $S_{i-1}$ into stably parabolic subgroups of $S_i$; and
\item There is a GGD for $S_{k+1}$ with respect to which $S_{k+1}$ is not \Cdiv.
\end{enumerate}
\end{restatable}

Given Theorems~\ref{t:NEW non-divergent} and~\ref{t:technical outline}, Theorem~\ref{t:collapsing Divergent} is easily proved.

\begin{proof}[Proof of Theorem~\ref{t:collapsing Divergent}]
Recall that we have an $\Mgen$--limit group $L$ defined by a non-\Tdiv sequence $(\phi_i)$, and we have to prove that the edge groups of the geometric decomposition of $L$ with respect to $(\phi_i)$ are finitely generated.  We apply Theorem~\ref{t:technical outline} to $L$ and $(\phi_i)$.  By Condition~\eqref{item:sp inject} of Theorem~\ref{t:technical outline} and induction, any stably parabolic subgroup $H$ of $L$ embeds into a stably parabolic subgroup of the group $S_{k+1}$ given by Theorem~\ref{t:technical outline}.  Moreover, by Theorem~\ref{t:NEW non-divergent}, the stably parabolic subgroups of $S_{k+1}$ are finitely generated.  The stably parabolic subgroups of $S_k$ are abelian by Lemma~\ref{l:sp abelian}, so the stably parabolic subgroups of $L$ are subgroups of finitely generated abelian groups, and hence are finitely generated.  The geometric decomposition of $L$ is a GGD by Lemma~\ref{l:gd is ggd}, and so the edge groups of the geometric decomposition of $L$ are stably parabolic by this lemma and Definition~\ref{def:ggd}.\eqref{ggd:sp}, completing the proof of Theorem~\ref{t:collapsing Divergent}.
\end{proof}

So, it finally remains to prove Theorem~\ref{t:NEW non-divergent} and Theorem~\ref{t:technical outline}.  We briefly comment on the tools we use to prove these two results.

In Section~\ref{s:collapsing Divergent} we define a new space, the \emph{collapsed space}.  This space arises from $\bH^2$ or $\bH^3$ by collapsing an invariant collection of tubes around geodesics to lines, and an invariant collection of balls and horoballs to points.  This space is defined because we do not know how to make the shortening argument work for the actions of different geometric pieces of $3$--manifold groups on $\bH^2$ and $\bH^3$.  Therefore, we construct the collapsed space carefully in order to make limiting $\R$--trees and the shortening argument work in ``nearly" the usual way.  Section~\ref{s:R-trees} is devoted to analysing limiting $\R$--trees that arise from degenerations of actions on the collapsed space.  Section~\ref{s:JSJ} is devoted to proving that the appropriate JSJ-decomposition works for $\Mgen$--limit groups.  Section~\ref{s:res-short} is devoted to the shortening argument, and the proof of Theorem~\ref{t:technical outline}.  Though there is substantial work to do in all of these sections, these will be relatively familiar to the experts.  In fact, this is the benefit of the collapsed space -- it is designed to make these arguments run in as ``standard" a way as possible.

The downside to the use of the collapsed space is that it is possible for interesting sequences of homomorphisms to not lead to a divergent sequence of actions on the collapsed space(s).  Section~\ref{s:non-divergent} is devoted to dealing with this issue.  We define what it means for a sequence of homomorphisms to be \Cdiv (Definition~\ref{d:P-divergent}), which is when a sequence of homomorphisms leads to actions on the collapsed space which diverge.  The goal of Section~\ref{s:non-divergent} is to prove Theorem~\ref{t:NEW non-divergent}.  This is achieved by finding yet another limiting tree, the \emph{cut point tree} (see Subsection~\ref{ss:cut point}), and introducing and analyzing a new kind of graph of groups, which we call \emph{edge-twisted} (see Definition~\ref{def:edge-twist}).  These definitions are carefully tuned to the situation of $\Mgen$-limits groups and the associated actions on the collapsed spaces, and designed to prove Theorem~\ref{t:NEW non-divergent}.

\section{The Collapsed Space} \label{s:collapsing Divergent}

Suppose $\Delta \in \Mgen$, and consider the geometric decomposition of $\Delta$.  There are two kinds of vertex groups: (i) Hyperbolic vertex groups which act on $\bH^3$; and (ii) Seifert-fibered vertex groups with hyperbolic base orbifold which (through projection to the base orbifold and a choice of hyperbolic structure on this base orbifold) act on $\bH^2$.  Each of $\bH^3$ and $\bH^2$ are $\delta$--hyperbolic.  However, these actions do not have properties that are well suited for the arguments in this paper. 

In this section we build the {\em collapsed space}, a variation on Farb's construction of the ``Electric" space from \cite{farb:relhyp}.  The collapsed space is geodesic and $\delta$--hyperbolic for a uniform $\delta$ (see Proposition~\ref{p: geodesic} and Theorem~\ref{t:hyperbolicity}).  The key advantage of the collapsed space over $\bH^2$ and $\bH^3$ is given by Lemma ~\ref{l:almost equal trans len for para}.  As well as these results, we prove a collection of basic structural results which are used in later sections.

\subsection{The collapsed space for Kleinian and Fuchsian groups} \label{ss:collapsed geometric}

Let $\Gamma$ be a finitely generated Fuchsian or torsion-free Kleinian group.  In case $\Gamma$ is Kleinian, let $\bH^\ast = \bH^3$, and in case $\Gamma$ is Fuchsian let $\bH^\ast = \bH^2$.  Fix $\delta_\bH$ so that all geodesic triangles in $\bH^2$ and $\bH^3$ are $\delta_{\bH}$--thin.
 Let $\epsilon$ be the $3$--dimensional (or $2$--dimensional) Margulis constant (see \cite{BP} for more information). Define the \emph{$\epsilon$--thin part} of $\bH^\ast$ with respect to $\Gamma$:
\begin{equation*}
\mc{U}_{\epsilon}(\Gamma)=\{x\in \bH^\ast \mid d(x, \gamma x)\leq \epsilon \text{ for some } \gamma\neq 1\in \Gamma\}.
\end{equation*}
The set $\mc{U}_{\epsilon}(\Gamma)$ is a disjoint union of a collection of horoballs, balls and neighborhoods of geodesics  (see, for example, \cite[Theorem D.3.3]{BP}, though we work in the universal cover).  We refer to the neighborhoods of geodesics and geodesic segments in the thin part of $\bH^\ast$ as \emph{tubes}. 

\begin{proposition} \label{p:props of U^K}
For any $K > 0$ there exists $D > 0$ and a $\Gamma$--invariant open set $\mc{U}^{K}(\Gamma) \subseteq \mc{U}_{\epsilon}(\Gamma)$ so that $\mc{U}^K(\Gamma)$ satisfies the following:
\begin{enumerate}
\item The components of $\mc{U}^K(\Gamma)$ are tubes, horoballs and balls;
\item Distinct components of $\mc{U}^K(\Gamma)$ are at least $K$ apart;
\item\label{eq:lower bound} Suppose $x\in \partial U$ for some connected component of $U$ of $\mc{U}^K(\Gamma)$. Then every non-trivial $g\in\Gamma$ moves $x$ by at least $D$;
\item\label{eq:large radius} Tubes and balls in $\mc{U}^K(\Gamma)$ have radius at least $K$; and
\item Any horoball in $\mc{U}_\epsilon(\Gamma)$, and every tube and ball of radius at least $2K$ is contained in the $K$--neighborhood of $\mc{U}^K(\Gamma)$.
\end{enumerate}
\end{proposition}

\begin{proof}
 The set $\mc{U}^K(\Gamma)$ is defined by taking $\mc{U}_\epsilon(\Gamma)$,  shrinking the components until they are distance $K$ apart, and then keeping only the tubes and balls which remain that are of radius at least $K$.

Given this construction, Item (1) follows from the Margulis Lemma. Items (2), (4) and (5) are true by the construction of $\mc{U}^K(\Gamma)$.  We now prove Item (3): By definition, $U$ is obtained by shrinking the radius of a component $U'$ of $\mc{U}_\epsilon(\Gamma)$ by $K$.  By definition of $\mc{U}_\epsilon(\Gamma)$, any non-trivial $g\in \Gamma$ preserving $U'$ moves points on $\partial U'$ by $\epsilon$. Hence $g$ moves points on $\partial U$ by an amount $D'$  depending only on $K$.  If $g\in \Gamma$ does not preserve $U'$, then $g$ does not preserve $U$. Then $U$ and $gU$ are at least $K$ apart.  Therefore we can let $D$ be the minimum of $\{D', K\}$. 
\end{proof}

Given points $x, y$ in a tube $U$ around a geodesic $\gamma$ in $\bH^*$, let $\pi_\gamma \co U \to \gamma$ be the closest point projection and let $d_\gamma \left( x , y \right) = d_{\bH^*}(\pi_{\gamma}(x),\pi_{\gamma}(y) )$.  		
Suppose $K>1$ (a specific value for $K$ is fixed later). Define an ``electric" distance $d$ on $\bH^*$ by
\[
d_{el}(x,y) = \left\{
\begin{array}{ll} 
0 & \mbox{if $x,y$ lie in the same horoball or ball of $\mc{U}^K(\Gamma)$;}\\
d_\gamma (x,y) & \mbox{if $x,y$ both lie in a tube in $\mc{U}^K(\Gamma)$; and}\\
d_{\bH^\ast}(x,y) & \mbox{otherwise.}
\end{array}\right.
\]

Define a pseudo-metric $\widehat{d}$ on $\bH^*$ by setting 
\[	\widehat{d}(x,y) = \inf_{ \{ x_i \}_{i=0}^n}	\left\{ \sum_{i=1}^n d_{el}(x_i,x_{i-1}) \right\}	,		\]
where the infimum is taken over all finite sets $x_0 = x, \ldots , x_{n-1}, x_n = y$.

Denote by $(\Coll^K(\Gamma), d)$ the induced metric space of $(\bH^*,\widehat{d})$, and let $q^K_\Gamma \co \bH^\ast \to \Coll^K(\Gamma)$ be the canonical quotient map. Observe $(\Coll^K(\Gamma), d)$ is a length space. Denote the length of a path $p$ in $\Coll^K(\Gamma)$ by $\widehat{\len}(p)$.  

\begin{definition}
The {\em collapsing locus} of $\Coll^K(\Gamma)$ is the collection of points $x\in \Coll^K(\Gamma)$ so that $\left| (q^K_\Gamma)^{-1}(x) \right| > 1$.
\end{definition}

Let $\gamma$ be a path from $x$ to $y$ in $\Coll^K(\Gamma)$ so that the intersection of $\gamma$ with each component of the collapsing locus is connected. Let the components of the intersection of $\gamma$ with the collapsing locus of $\Coll^K(\Gamma)$ be $\mathbf{C}=\{C_1, \dots, C_k\}$.  The {\em collapsed set associated to $\gamma$} is $\left\{ U_i=(q^K_\Gamma)^{-1}(C_i) \right\}$, a disjoint collection of horoballs, balls and tubes around geodesic segments.  If $C_i$ does not contain $x$ or $y$ then $\gamma$ {\em penetrates} $U_i$. Let $\{\gamma_0,\dots, \gamma_k\}$ be the components of $\gamma \smallsetminus \mathbf{C}$.  Note that $\gamma_0$ and/or $\gamma_k$ may be empty. Let $\widetilde\gamma_i$ be the closure of $(q^K_\Gamma)^{-1}(\gamma_i)$ in $\bH^\ast$. The \emph{entry point} of $\gamma$ into $U_{i}$ is $e(\gamma, U_i)=\widetilde\gamma_{i-1}\cap U_{i}$ and the \emph{exit point} out of $U_i$ is $o(\gamma, U_i)=\widetilde\gamma_{i}\cap U_i$. Let $\widetilde\gamma_i^p$ be the geodesic between $e(\gamma, U_i)$ and $o(\gamma, U_i)$. The \emph{lift} of $\gamma$ is $\{\widetilde\gamma_0, \widetilde\gamma_1^p, \widetilde\gamma_1,\dots, \widetilde\gamma_k^p, \widetilde\gamma_k\}$.  If each $\widetilde\gamma_i$ is geodesic then $\gamma$ is a {\em local collapsed geodesic}.  

A local collapsed geodesic $\gamma$ between $x,y\in\Coll^K(\Gamma)$ is an {\em almost geodesic} if $\widehat{\len}(\gamma)\leq \widehat{d}(x,y)+1$. Any path $[x, y]$ can be replaced by a local collapsed geodesic connecting $x, y$ without increasing its length, so an almost geodesic always exists between any two points in $\Coll^K(\Gamma)$.

\begin{notation}
Suppose $(X,d_X)$ is a proper metric space, and  $A,B \subset X$ are closed.  If $x \in X$ then $\pi_A(x) = \left\{ y \in A \mid d_X(x,y) = d_X(x,A) \right\}$ and $\pi_A(B) = \bigcup\limits_{b \in B} \pi_A(b)$. Note that $\pi_A(x)$ and $\pi_A(B)$ are non-empty as $X$ is proper. 
\end{notation}
Here is a standard fact about hyperbolic spaces.

\begin{lemma}\label{l:visual size}
Let $(X,d_X)$ be a $\nu$--hyperbolic metric space and suppose that $W_1,W_2 \subset X$ are convex and that $d(W_1,W_2) > 6\nu$.  There exists $x \in X$ with $d_X(x,W_2) \le 3\nu$ so that $\pi_{W_1}(W_2) \subseteq \pi_{W_1}(B(x,2\nu))$.
\end{lemma}

\begin{lemma}\label{Bounded projection}
Suppose that $U_1$ and $U_2$ are convex subsets of $\bH^\ast$ that are at distance $t \ge 6\delta_\bH$ apart.   Then $\mathrm{Diam}\left( \pi_{U_1}(U_2) \right) \le 4\delta_{\bH} e^{-t + 5\delta_{\bH}}$.
\end{lemma}
\begin{proof}
This follows from Lemma~\ref{l:visual size} and the fact that for any $t'$ the length of paths outside of $N(U_2,t')$ decreases by a factor of at least $e^{-t'}$ under projection to $U_2$.
\end{proof}

The following lemma is similar to \cite[Lemma 4.5]{farb:relhyp}.  

\begin{lemma}\label{l:track geodesics}
There exists an absolute constant $l>2$ independent of $K$ with the following property: Let $\widetilde\gamma$ be the geodesic between $\widetilde x,\widetilde y\in \bH^\ast$ and let $\gamma=q^K_\Gamma(\widetilde\gamma)$. Let $\beta$ be any almost geodesic  in $\Coll^K(\Gamma)$ between $x=q^K_\Gamma(\widetilde x)$ and $y=q^K_\Gamma(\widetilde y)$. Any subsegment of $\beta$ which lies outside of  $N(\gamma, l)$ has length at most $4l+4$. In particular, any almost geodesic from $x$ to $y$ stays completely inside $N(\gamma, 3l+2)$.
\end{lemma}
\begin{proof}
We claim any $l$ satisfying the following conditions suffices:
\begin{enumerate}
\item $l \ge 6\delta_\bH>2$; and 
\item $[1-(4\delta_\bH+1)e^{-l+5\delta_\bH}]\geq 1/2 > 0$.
\end{enumerate}

Let $\beta'$ be a subsegment of $\beta$ lying completely outside $N(\gamma, l)$, and let $z$ and $w$ be the first point and the last point of $\beta'$, respectively.  Since $\beta'$ lies outside $N(\gamma, l)$, all collapsed sets penetrated by $\beta'$ are at least $l$ away from $\widetilde\gamma$. 
Let $\{U_1, \dots, U_k\}$ be the collapsed set for $\beta'$, and let 
\begin{equation*}
\widetilde\beta_0, \widetilde\beta_1^p, \widetilde\beta_1,\dots, \widetilde\beta_k^p, \widetilde\beta_k
\end{equation*}
be the lift of $\beta'$ to $\bH^\ast$. Then $q^K_\Gamma(\widetilde\beta_i)$ is subsegment of $\beta'$ and hence $\widetilde\beta_i$ is at least $l$ away from $\widetilde\gamma$. Since $U_i$ is at least $l$ away from $\widetilde\gamma$ and $l\geq 6\delta_\bH$,   the projection of $\widetilde\beta_i$ to $\widetilde\gamma$ has length $\len(\widetilde\beta_i)e^{-l}$. Since $\widetilde\beta^p_i\subset U_i$ and $U_i$ is $l$ away from $\widetilde\gamma$,  by Lemma~\ref{Bounded projection} the projection of $\widetilde\beta^p_i$ to $\widetilde\gamma$ has length at most $4\delta_{\bH}e^{-l + 5\delta_{\bH}}$.  Let $l' = l - 5\delta_{\bH}$.

Let $\widetilde z$ and $\widetilde w$ be the $q^K_\Gamma$--pre-image of $z$ and $w$, respectively, and let $\widetilde z'$ and $\widetilde w'$ denote the images of $\widetilde z$ and $\widetilde w$ under orthogonal projection onto $\widetilde\gamma$, respectively.  Let $z'=q^K_\Gamma(\widetilde z')$ and $w'=q^K_\Gamma(\widetilde w')$.  Then 
\begin{eqnarray*}
\widehat{\len}(\beta')&\leq& d( z, z')+d( z', w')+d(w', w)+1\\
&\leq& l+d(\widetilde z', \widetilde w')+l+1\\
&\leq& l+ \widehat{\len}(\beta')\cdot e^{-l'}+4\delta_\bH\cdot k\cdot e^{-l'}+l+1
\end{eqnarray*} 
Components  are at least $K$ apart (in $\bH^\ast$--distance) with $K>1$, so $k-1\leq \widehat{\len}(\beta')$, and 
\begin{eqnarray*}
\widehat{\len}(\beta')&\leq& 2l+ \widehat{\len}(\beta')\cdot e^{-l'}+4\delta_\bH\cdot (\widehat{\len}(\beta')+1)\cdot e^{-l'}+1
\end{eqnarray*} 
Therefore 
\begin{eqnarray*}
\widehat{\len}(\beta')\cdot [1-(4\delta_\bH+1)e^{-l'}]&\leq& 2l+4\delta_\bH\cdot e^{-l'}+1
\end{eqnarray*}
We chose $l$ so that $[1-(4\delta_\bH+1)e^{-l'}]\geq 1/2 > 0$,  hence $4\delta_\bH\cdot e^{-l'}<1$, so 
\begin{equation*}
\widehat{\len}(\beta')\leq 2(2l+4\delta_\bH\cdot e^{-l'}+1)\leq 4l+4,
\end{equation*}
as required.
\end{proof}

\begin{lemma}\label{basic properties of collapsing locus}
The collapsing locus of $\Coll(\Gamma)$ has the following properties:
\begin{enumerate}
\item A component of the collapsing locus is either a point or (homeomorphic to) a line. 
\item Different components of the collapsing locus are at least $K$ away from each other. 
\end{enumerate}
\end{lemma}

\begin{proof}
Item (1) follows directly from the construction of $\Coll(\Gamma)$. Item (2) follows the second statement of Proposition \ref{p:props of U^K} and the construction of $\Coll(\Gamma)$. 
\end{proof}

Let $l$ be as in Lemma~\ref{l:track geodesics}.

\begin{proposition}\label{p: geodesic}
If $K \ge 8l$ then $(\Coll^K(\Gamma), d)$ is a geodesic metric space.
\end{proposition}
\begin{proof}
Since $(\Coll^K(\Gamma), d)$ is a length space, for any $x, y\in\Coll^K(\Gamma)$ there is a sequence of almost geodesics $p_i$ from $x$ to $y$ such that $\widehat{\len}(p_i)\rightarrow d(x, y)$. Choose lifts $\widetilde{x}$ and $\widetilde{y}$ of $x$ and $y$ to $\bH^\ast$, let $\widetilde{\gamma}$ denote the $\bH^\ast$--geodesic from $\widetilde{x}$ to $\widetilde{y}$ and let $\gamma = q_\Gamma(\widetilde{\gamma})$.  
By Lemmas~\ref{l:track geodesics}, each of the $p_i$
 lies within $3l+2$ of $\gamma$. Since different components of the collapsing locus are at least $K$ away from each other and $K\geq 8l$ by Lemma \ref{basic properties of collapsing locus}, the $(3l+2)$-neighborhood of $\gamma$ intersects finitely many components of the collapsing locus. Therefore, there is a finite collection of components of the collapsing locus which any of the $p_i$ intersect, and passing to a subsequence we may assume that the $p_i$ intersect exactly the same collection of components of the collapsing locus, say $C_1, \ldots , C_n$.

The space $\Coll^K(\Gamma)$ is locally compact except at points in the collapsing locus corresponding to the $q_\Gamma$--image of horoballs in $\mc{U}^K(\Gamma)$.   Suppose that $C_i$ is a component corresponding to a horoball $U_i$.  Then an almost geodesic
travelling through $C_i$ is locally the concatenation of two paths, each of which lift to paths in $\bH^\ast$ which travel almost as quickly as possible towards $U_i$.  In particular, if $a_i$ is a point on $p_i$ at distance $4l$ from $C_i$, then $d(a_i,C_i) > 3l+2$.  The path from $a_i$ to $\gamma$ of length at most $3l+2$ cannot intersect $C_i$, and since $K \ge 8l$ it cannot intersect any other component of the collapsing locus either.  It follows that such $a_i$ lift to a bounded region of $\bH^\ast$, and hence all of $p_i$ lifts to a bounded region of $\bH^\ast$.
Since $\bH^\ast$ is locally compact, $(p_i)$ sub-converges to a geodesic $p$ from $x$ to $y$.
\end{proof}

\begin{theorem}\label{t:hyperbolicity}
Let $l$ be as in Lemma~\ref{l:track geodesics} and $\delta = 4 \left( 6l + \delta_{\bH} + 2 \right)$.  Then for any $K>8l$, the space $(\Coll^K(\Gamma), d)$ is $\delta$--hyperbolic.
\end{theorem}

\begin{proof}
In the same way that \cite[Proposition 4.6]{farb:relhyp} follows from  \cite[Lemma 4.5]{farb:relhyp}, one can use Lemma~\ref{l:track geodesics} to see that $(\Coll^K(\Gamma), d)$ has $\left( 6l + \delta_{\bH} + 2 \right)$--slim triangles. 
\end{proof}

\begin{notation}\label{not:Fix K}
 Let $\delta$ be as in the statement of Theorem~\ref{t:hyperbolicity}, and fix $K = 40\delta$. We write $\Coll(\Gamma) = \Coll^K(\Gamma)$ and call $\Coll(\Gamma)$ the {\em collapsed space} associated to $\Gamma$.
We drop the superscript $K$ from other constructions also. Therefore, for example, we have the map $q_\Gamma \co \bH^\ast \to \Coll(\Gamma)$, and when $\Gamma$ is assumed we just write $q \co \bH^\ast \to \Coll$, etc.
\end{notation}

\begin{lemma}
If a component of the collapsing locus is a line, then it is a geodesic in $\Coll(\Gamma)$. 
\end{lemma}

\begin{proof}
 Let $\gamma$ be such a component. Then $\gamma=q(\widetilde\gamma)$, where $\widetilde\gamma$ is a geodesic in $\bH^\ast$. Let $x, y\in \gamma$. Then by Proposition \ref{l:track geodesics} any geodesic $[x, y]$ is in a $3l+2$-neighborhood of $\gamma$. By (2) of Lemma \ref{basic properties of collapsing locus} and the fact that $K=40\delta$ is much bigger than $3l+2$, we know that $[x, y]$ does not intersect other components of the collapsing locus. This implies that $[x, y]$ is exactly the subsegment of $\gamma$ connecting $x$ and $y$. Therefore $\gamma$ is a geodesic. 
\end{proof}

\begin{lemma}\label{lem:w close}
Let $Z$ be a $\nu$--hyperbolic space, and suppose that $C$ is a convex subset of $Z$.  Suppose that $u,v \in Z$ are so that $d_Z(u,C) = d_Z(v,C) = 12\nu$, and that some geodesic $\sigma$ from $u$ to $v$ intersects the $6\nu$--neighborhood of $C$.  Let $w$ be the point in $\sigma$ lying at distance $2\nu$ from $u$.  Then $d_Z(w,C) \le 11\frac{1}{4}\nu$.
\end{lemma}
\begin{proof}
Choose a geodesic $\rho$ in $Z$ from $u$ to $C$ so that the length of $\rho$ is at most $d(u,C) + \frac{1}{4} = 12\frac{1}{4}\nu$.  Let $s \in C$ be the endpoint of $\rho$, and consider a geodesic triangle with two sides $\rho$ and $\sigma$.  Since $\sigma$ intersects the $6\nu$--neighborhood of $C$, it is easy to see that the Gromov product $(v \mid s)_u$ is at least $2\nu$.

Therefore, there exists a point $m \in \rho$ so that $d(u,m) = 2\nu$ and $d(m,w) \le \nu$.  Note that $d(m,C) \le d(m,s) \le 10\frac{1}{4}\nu$ (by traveling along $\rho$).  Now,
\[	d(w,C) \le d(w,s) \le d(w,m) + d(m,s) \le \nu + 10\frac{1}{4}\nu = 11\frac{1}{4}\nu	,	\]
as required.
\end{proof}

\begin{lemma}\label{lem:z far}
Suppose that $p$ is the image in $\mathcal{C}(\Gamma)$ of an $\mathbb{H}^\ast$--geodesic.  Suppose that $C$ is a component of the collapsing locus which lies within $6\delta$ of $p$, that $r$ is a point in $N_{6\delta}(C) \cap p$, that $u$ is a point on $p$ so that $d(u,C) = 12\delta$, and that $z$ is a point on $p$ so that $d(u,z) = 2\delta$, where the order of the points on $p$ is $(z,u,r)$.

Then $d(z,C) \ge 13\delta$.
\end{lemma}
\begin{proof}
Lift to points $\tilde{z}, \tilde{u}, \tilde{r}, \tilde{w}$, and a path $\tilde{p}$, all in $\mathbb{H}^\ast$.  Let $\tilde{C}$ be the set collapsed to yield $C$, and let $\tilde{\tau}$ be the shortest path from $\tilde{u}$ to $\tilde{C}$.  Let $\tilde{t}$ be the point on $\tilde{\tau}$ at distance $6\delta$ from $\tilde{C}$, and consider the geodesic triangle with vertices $\tilde{z}, \tilde{t}, \tilde{r}$.  Since neighborhoods of convex sets are convex in $\mathbb{H}^\ast$, and since $\tilde{t}, \tilde{r}$ lie in the $6\delta$--neighborhood of $\tilde{C}$, the geodesic $\left[ \tilde{t},\tilde{r} \right]$ also lies entirely in the $6\delta$--neighborhood of $\tilde{C}$.

Therefore, there are points $\tilde{a} \in \left[ \tilde{z}, \tilde{t} \right], \tilde{b} \in \left[ \tilde{z}, \tilde{r} \right]$ so that $d_{\mathbb{H}^\ast}\left( \tilde{a}, \tilde{b} \right) \le \delta$, $d(\tilde{z},\tilde{a}) = d(\tilde{z},\tilde{b})$, and $d_{\mathbb{H}^\ast}\left(\tilde{a},\tilde{C}\right) = 8\delta$.

Now, $d(\tilde{b},\tilde{C}) \le d(\tilde{a},\tilde{b}) + \d(\tilde{a},\tilde{C}) \le \delta + 8\delta = 9\delta$.  Since $d(\tilde{u},\tilde{C}) = 12\delta$, it follows that $d(\tilde{u},\tilde{b} \ge 3\delta$.

We now have that 
\begin{eqnarray*}
d(\tilde{z},\tilde{C}) &=& d(\tilde{z},\tilde{a}) + d(\tilde{a}, \tilde{C}) \\
&=& d(\tilde{z},\tilde{a}) + 8\delta \\
&=& d(\tilde{z},\tilde{b}) + 8\delta \\
&=& d(\tilde{z},\tilde{u}) + d(\tilde{u},\tilde{b}) + 8\delta \\
&\ge& 2\delta + 3\delta + 8\delta = 13\delta	,	
\end{eqnarray*}
as required.
\end{proof}

\begin{proposition}\label{p:Hausdorff distance}
Let $x, y\in \Coll(\Gamma)$ and let $\gamma_2=[x, y]$ be a $\Coll(\Gamma)$--geodesic.  Let  $\widetilde{x}, \widetilde{y} \in \bH^\ast$ be lifts of $x, y$, respectively.  Let $\widetilde{\gamma_1}=[\widetilde{x}, \widetilde{y}]$ be the $\bH^\ast$--geodesic and $\gamma_1=q_\Gamma(\widetilde{\gamma_1})$.  
The Hausdorff distance between $\gamma_1$ and $\gamma_2$ is at most $9\delta$.
\end{proposition}

\begin{proof}
By Lemma~\ref{l:track geodesics} $\gamma_2\subseteq N(\gamma_1, 3l+2)$, and $3l+2 \le 9\delta$.  So it remains to prove that $\gamma_1 \in N(\gamma_2,9\delta)$.

First note that any segment of $\gamma_1$ which lies outside the $4l$--neighborhood of the collapsing locus is a geodesic.  For, if not the geodesic must intersect the collapsing locus.  But the geodesic cannot intersect the collapsing locus, since by Lemma~\ref{l:track geodesics} the geodesic remains within distance $3l+2$ of $\gamma_1$ and $l>2$. 

Given a convex set $W \in \bH^\ast$ and a geodesic $\gamma$, the distance from $\gamma(t)$ to $W$ is a strictly convex function, so for any $r$ there are at most two points on $\gamma$ lying at distance exactly $r$ from $W$.

Now, let $C_1, \ldots , C_k$ denote the components of the collapsing locus that lie within $6\delta$ of $\gamma_2$. Note that $k<\infty$ as $K$ is much bigger than $\delta$. Suppose that $\gamma_1 = p_1 \cdot p_2 \cdots p_j$ (where $j \in \{ 2k-1, 2k , 2k+1 \}$) is so that the endpoints of each $p_i$ lie at distance exactly $12\delta$ from $C_i$.  Thus, the $p_i$ alternate between paths which lie outside the $12\delta$--neighborhood of the collapsing locus, and within $12\delta$ of some $C_l$.  Suppose that $p_i$ lies within $12\delta$ of $C_l$, and let $q$ be a $\Coll(\Gamma)$-geodesic with the same endpoints as $p_i$. We claim that $p_i$ is in the $2\delta$-neighborhood of $q$. We prove this claim in the next paragraph.  

First note that if $q$ does not intersect $C_l$ then $q=p_i$ and the claim follows. Now suppose $q$ intersects $C_l$. 
Let $\tilde{p_i}$ be the lift of $p_i$ to $\bH^\ast$ and denote the end points of $\tilde{p_i}$ by $\widetilde{x}$ and $\widetilde{y}$.  Let $\widetilde{x}'$ and $\widetilde{y}'$ be the closest point projections of $\widetilde{x}$ and $\widetilde{y}$ to $q_\Gamma^{-1}(C_l)$, respectively. By hyperbolicity of $\bH^\ast$, $\tilde{p_i}$ is contained in the $2\delta_\bH$-neighborhood of $[\widetilde{x}, \widetilde{x}']\cup[\widetilde{x}',\widetilde{y}']\cup [\widetilde{y}, \widetilde{y}']$. Since $q_\Gamma$ is distance decreasing, we see that $p_i$ is within $2\delta_\bH \le \delta$ of $q_\Gamma([\widetilde{x}, \widetilde{x}']\cup[\widetilde{x}',\widetilde{y}']\cup [\widetilde{y}, \widetilde{y}'])$. 

Since $q$ intersects $C_l$, we can write $q$ as the $q_\Gamma$ image of $[\widetilde{x}, \widetilde{x_1}]\cup [\widetilde{x_1}, \widetilde{y_1}]\cup [\widetilde{y_1}, \widetilde{y}]$.  Consider the geodesic triangle with vertices $\widetilde{x}, \widetilde{x_1}, \widetilde{x}'$. Since $q_\Gamma^{-1}(C_l)$ is convex and $\widetilde{x}'$ is the point in $q_\Gamma^{-1}(C_l)$ closest to $\widetilde{x}$, the Gromov product of $[\widetilde{x}, \widetilde{x}']$ and $[\widetilde{x}', \widetilde{x_1}]$ at $\widetilde{x}'$ is at most $\delta_\bH$. Hence $[\widetilde{x}, \widetilde{x}']\cup [\widetilde{x}', \widetilde{x_1}]$ is with in $2\delta_\bH\leq \delta$ of $[\widetilde{x}, \widetilde{x_1}]$. Similarly, $[\widetilde{y}, \widetilde{y}']\cup [\widetilde{y}', \widetilde{y_1}]$ is with in $2\delta_\bH\leq \delta$ of $[\widetilde{y}, \widetilde{y_1}]$. Therefore, as $q_\Gamma$ is distance decreasing, we have $q_\Gamma([\widetilde{x}, \widetilde{x}']\cup [\widetilde{x}',\widetilde{x_1}]\cup [\widetilde{x_1}, \widetilde{y_1}]\cup [\widetilde{y_1},\widetilde{y}']\cup [\widetilde{y}', \widetilde{y}])$ is within $\delta$ of $q_\Gamma([\widetilde{x}, \widetilde{x_1}]\cup [\widetilde{x_1}, \widetilde{y_1}]\cup [\widetilde{y_1}, \widetilde{y}])$. This implies that $p_i$ is within $2\delta$ of $q$.  

By the claim, if we replace each of the $p_i$ which lies in the $12\delta$--neighborhood of some $C_l$ by a $\Coll(\Gamma)$--geodesic, then we obtain a concatenation of $\Coll(\Gamma)$--geodesics whose $2\delta$-neighborhood contains $\gamma_1$.

Using Lemmas~\ref{lem:w close} and \ref{lem:z far} it is straightforward to check that the Gromov product at each point of concatenation is at most $2\delta$.  It is also clear that each of the paths in this concatenation, except possibly the first and last, have length greater than $12\delta$ (since $K \ge 40\delta > 24\delta$).  Therefore, by \cite[Lemma 4.9]{AGM_msqt} (with $l = 2\delta$), the Hausdorff distance between this concatenation and $\gamma_2$ is at most $7\delta$. This together with the last statement of the previous paragragh implies that $\gamma_1 \in N(\gamma_2,9\delta)$ and the proposition follows.
\end{proof}

The following is similar to \cite[Lemma 4.8]{farb:relhyp}.
\begin{lemma}\label{l:generalized Farb's 4.8}
There exists $D_1$ satisfying the following:
Let $x, y\in \Coll(\Gamma)$ and let $\gamma_2=[x, y]$ be $\Coll(\Gamma)$--geodesic and $\widetilde{\gamma_2}$ be its lift.  Let  $\widetilde{x}, \widetilde{y} \in \bH^\ast$ be lifts of $x, y$, respectively.  Let $\widetilde{\gamma_1}=[\widetilde{x}, \widetilde{y}]$ and $\gamma_1=q_\Gamma(\widetilde{\gamma_1})$.  If precisely one of $\{\gamma_1,\gamma_2\}$ penetrates a component $U$ of the collapsed set then $d_{\bH^\ast}\left(e(\gamma_i, U),o(\gamma_i, U)\right) \le D_1$. 
\end{lemma}

\begin{proof}
Suppose $\gamma_2$ penetrates $U$ and $\gamma_1$ does not. Let $\widetilde z$ be the point on $\widetilde{\gamma}_2$ between $\widetilde x$ and $e(\gamma_2, U)$ and $10l$ away from $U$, or $\widetilde{z}=\widetilde{x}$ if no such point exists.  Let $\widetilde w$ be the point on $\widetilde{\gamma}_2$ between $\widetilde y$ and $o(\gamma_2, U)$ and $10l$ away from $U$, or $\widetilde{w}=\widetilde{y}$ if no such point exists. Let $z=q_\Gamma(\widetilde z)$ and $w=q_\Gamma(\widetilde w)$. By Lemma~\ref{l:track geodesics}, there exist $z', w'$ on $\gamma_1$ with $d(z, z'), d(w, w')\leq 5l$. Let $\widetilde{z}'$ and $\widetilde{w}'$ be the lifts of $z'$ and $w'$, respectively. By construction $[z', z]$ and $[w', w]$ are disjoint from the collapsing locus. Hence $d_{\bH^{\ast}}(\widetilde{z}, \widetilde{z}'), d_{\bH^{\ast}}(\widetilde{w}, \widetilde{w}')\leq 5l$. Since these paths are in $\bH^\ast$, $[\widetilde{z}', \widetilde{w}'] \subseteq N_{5l}\left( [\widetilde z, \widetilde w] \right)$.  Thus,  $q_\Gamma([\widetilde{z}', \widetilde{w}']) \subseteq N_{5l}\left( [z,w] \right)$, since $[z, w]=q_\Gamma([\widetilde z, \widetilde w])$. By Lemma~\ref{l:track geodesics}, the subsegment $\beta$ of $\gamma_1$ between $z'$ and $w'$ lies in $N_{3l+1}\left( q_\Gamma([\widetilde{z}', \widetilde{w}']) \right)$, so $\beta \subseteq N_{8l+1} \left( [z,w] \right) \subseteq N_{18l+1} \left( q_\Gamma(U) \right)$. Components of the collapsing locus are $40l$--separated, so $\beta$ does not penetrate any  component of the collapsed set except $U$.  On the other hand, since $\beta$ is a subsegment of $\gamma_1$, it does not penetrate $U$. Therefore the lift of $\beta$ is $[\widetilde{z}', \widetilde{w}']$ and $[\widetilde{z}', \widetilde{w}']$ is disjoint from $U$, so the projection of $[\widetilde{z}', \widetilde{w}']$ onto $U$ is bounded by a uniform constant $D$. Since $[z,z']$ and $[w, w']$ do not penetrate any component of the collapsing locus, the projections of $[\widetilde z, \widetilde{z}']$ and $[\widetilde w, \widetilde{w}']$ onto $U$ are also bounded by $D$. By choice of $\widetilde z$ and $\widetilde w$, $d_{\bH^{\ast}}(\pi_U(\widetilde z), e(\gamma_2, U))$ and $d_{\bH^{\ast}}(\pi_U(\widetilde w), o(\gamma_2, U))$ are bounded by a uniform constant $C_2$. The above distance bounds imply $d_{{\bH^{\ast}}}(o(\gamma_2, U), e(\gamma_2, U))\leq 3C_1+2C_2$. 

The case where $\gamma_1$ penetrates $U$ and $\gamma_2$ does not is similar, the only difference being that instead of using Lemma \ref{l:track geodesics}, we use Proposition \ref{p:Hausdorff distance}. As a result, the constants are different but they are again uniform constants.   
\end{proof}

The following is similar to \cite[Lemma 4.9]{farb:relhyp}.
\begin{lemma}\label{l:generalized Farb's 4.9}
There exists $D_2$ satisfying:
Let $x,y, \widetilde{x},\widetilde{y},\gamma_1$ and $\gamma_2$ be as in Lemma~\ref{l:generalized Farb's 4.8}. If both $\gamma_1$ and $\gamma_2$ penetrate a component $U$ of the collapsed set then $d_{\bH^{\ast}}\left(e(\gamma_1, U), e(\gamma_2, U)\right)  \le D_2$.  The same is true for exit points.
\end{lemma}

\begin{proof}
Let $\widetilde z$ be the point on $\widetilde{\gamma}_2$ between $\widetilde x$ and $e(\gamma_2, U)$ and $10l$ away from $U$, or $\widetilde{z}=\widetilde{x}$ if no such point exists. So we have 
\begin{equation}\label{z U bounded}
d_{\bH^\ast}(\widetilde z, U)\leq 10l
\end{equation}
Note that the distance between $\widetilde z$ and $e(\gamma_2, U)$ can not be much bigger than $10l$ otherwise the path $[\widetilde z, e(\gamma_2, U)]$ intersects $U$ in a non-degenerate segment, contradicting the definition of $e(\gamma_2, U)$. Hence there is a number $B$ depending only on $l$ such that 
\begin{equation}\label{bound tilde z and entry point}
d_{\bH^{\ast}}\left(\widetilde z, e(\gamma_2, U)\right)  \le B.
\end{equation}
By Lemma~\ref{l:track geodesics}, there exists $w$ on $\gamma_1$, so that $d(w, z)\leq 5l$ ,where $z=q_\Gamma(\widetilde z)$. When $\widetilde z=\widetilde x$, we let $w=x=z$. Definition of $\widetilde z$ implies that we either have $d(z, q(U))=10l$ or $z=x$. Hence $[w, z]\subset N(U, 15l)-N(U, 5l)$ unless $z=x=w$. Since $20l<K$, there is no other connected component of the collapsing locus in the $20l$--neighborhood of $q_\Gamma(U)$, so the geodesic $[z,w]$ does not intersect any component of the collapsing locus. Let $\widetilde w=q_\Gamma^{-1}(w)$. The geodesic $[z,w]$ lifts to the geodesic $[\widetilde z, \widetilde w]$. Hence
\begin{equation}\label{z, w less than 5l}
d_{\bH^\ast}(\widetilde z, \widetilde w)\leq 5l.
\end{equation}
By  (\ref{z, w less than 5l}) and (\ref{z U bounded}), we have $d_{\bH^\ast}(\widetilde w, U)\leq 15l$. This implies that,  in the same way as (\ref{z U bounded}) implying (\ref{bound tilde z and entry point}), there is a number $B'$ depending only on $l$ such that 
\begin{equation}\label{bound tilde w and entry point 1}
d_{\bH^{\ast}}\left(\widetilde w, e(\gamma_1, U)\right)  \le B'.
\end{equation}
Therefore, by (\ref{bound tilde z and entry point}), (\ref{z, w less than 5l}) and (\ref{bound tilde w and entry point 1}), the distance between $d_{\bH^\ast}(e(\gamma_1, U), e(\gamma_2, U))\leq B+B'+5l$.  

The proof of the statement for the exit points is similar. 
\end{proof}

\begin{lemma}\label{l:almost equal trans len for para}
There exist uniform constants $B_1$ and $B_2$ satisfying:
Let $g\in\Gamma \smallsetminus \{ 1 \}$ and $x\in \Coll(\Gamma)$.  If $g$ is parabolic, let $\gamma$ be the component of the collapsing locus fixed by $g$. Otherwise let $\gamma$ be the $q_\Gamma$-image of the minimal invariant set of $g$ in $\bH^\ast$.  Then $2d(x, \gamma) \le d(x, gx)- B_1$ and if $g$ is elliptic or parabolic then $\left|d(x, gx)-2d(x, \gamma)\right| \le B_2$.
\end{lemma}

\begin{proof}
Let $\widetilde{x}\in q_\Gamma^{-1}(x)$. Let $U_1=q^{-1}_{\Gamma}(\gamma)$ and $\widetilde y=\pi_{U_1}(\widetilde x)$. Let $U_2$ be a horoball (or ball or tube)  with the same center (or axis) as $U_1$ so $g$ moves points on $\partial U_2$ by $5\delta+l_g$, where $l_g$ is the translation length of $g$ on $\bH^\ast$. Let $\widetilde w=\pi_{U_2}(\widetilde x)$. We claim that $d(\widetilde w, U_1)\leq D_1$ for some absolute constant $D_1$ . The claim is trivial if $U_2\subset U_1$, so suppose $U_1\subset U_2$. 

First suppose that $\gamma$ is a component of the collapsing locus. In this case by Proposition~\ref{p:props of U^K}.\eqref{eq:lower bound} $d(g \widetilde y, \widetilde y)\geq {\rm max}\{D, l_g\}$, so the length of $[\widetilde w, \widetilde y]$ is bounded above by an absolute constant.  

Now suppose $\gamma$ is \emph{not} a component of the collapsing locus. In this case, $g$ is either loxodromic and $l_g$ is bounded from below by an absolute positive constant, or $g$ is elliptic and the angle of rotation is bounded below by an absolute positive constant, so $U_2$ has uniformly bounded size (tube radius). 

The claim gives an absolute upper bound on $d(w, \gamma)$, where $w=q_\Gamma(\widetilde w)$. Since $\widetilde{w}$ is the point in $U_2$ closest to $\widetilde{x}$ and $[\widetilde  w, g \cdot  \widetilde w]$ is contained in $U_2$, the Gromov product of $[\widetilde  x,  \widetilde w]$ and $[\widetilde  w, g \cdot  \widetilde w]$ at $\widetilde{w}$ is at almost $\delta_\bH$. Similarly, the Gromov product of $[g\cdot\widetilde  x,  g\cdot\widetilde w]$ and $[\widetilde  w, g \cdot  \widetilde w]$ at $g\cdot\widetilde{w}$ is also at almost $\delta_\bH$. Hence, by hyperbolicity of $\bH^\ast$, the Hausdorff distance between $[\widetilde x, g \cdot \widetilde  x]$ and  $[\widetilde  x,  \widetilde w]\cup[\widetilde  w, g \cdot  \widetilde w]\cup[g\cdot\widetilde  w, g \cdot \widetilde  x]$ is at most $4\delta_{\bH}$. By  Lemma~\ref{l:track geodesics}, $[x, g\cdot x]$, $[x, w]$ and $[g\cdot x, g\cdot w]$ are in the $\delta$-neighborhoods of $q_{\Gamma}([\widetilde x, g \cdot \widetilde  x])$, $q_{\Gamma}([\widetilde x, \widetilde  w])$ and $q_{\Gamma}([g \cdot \widetilde x, g \cdot \widetilde  w])$, respectively. Hence the Hausdorff distance between $[x, g \cdot x]$ and $[ x,  w]\cup[ w, g \cdot w]\cup[g \cdot w, g \cdot x]$ is bounded above by a constant depending only on $\delta$. Thus, there is a constant $L_1$ depending only on $\delta$ so that $d(x,g \cdot x) > 2 d(x,w) - L_1$.
If $g$ is elliptic or parabolic, then $d(w, g \cdot w)\leq 5\delta$ and hence $\left| d(x, g \cdot x)-2d(x, w) \right|$ is bounded by a constant $L_2$ depending only on $\delta$.  Let  $B_2$ be the sum of $L_2$ and an absolute upper bound of $d(w, \gamma)$. Then the lemma follows. 
\end{proof}

\begin{lemma}\label{least shift}
Suppose $g\in \Gamma$ acts loxodromically on $\bH^\ast$ with axis $\widetilde \gamma$, and let $\gamma=q(\widetilde\gamma)$. If  $\gamma$ comes within $5\delta$ of the collapsing locus but is not contained in the collapsing locus then $g$ moves every point of $\gamma$ by at least $30\delta$. 
\end{lemma} 

\begin{proof}
By assumption, there exists $x$ on $\gamma$ within $5\delta$ of a component $C$ of the collapsing locus and $\gamma$ is not contained in $C$.  Hence $C$ and $g \cdot C$ are distinct components of the collapsing locus, which are separated by $K = 40\delta$. Hence $x$ and $gx$ are separated by a distance at least $40\delta-2 \left( 5\delta \right)=30\delta$. So the translation length of $g$ on $\widetilde\gamma$ is at least $30\delta$. Let $y$ be a point on $\gamma$. Suppose $d(y, gy)<30\delta$. Since the translation length of $g$ on $\widetilde\gamma$ is at least $30\delta$, $[y, gy]$ intersects some connected component $C'$ of the collapsing locus and we have $d(y, C')+d(C', gy)\leq 30\delta$. Hence $d(gy, gC')+d(C', gy)\leq 30\delta$, so $d(C', gC')\leq 30\delta$. Thus $gC'=C'$, so $C'$ is a line and $\gamma= C'$. This is a contradiction. 
\end{proof}

\section{Sequences which are not \Cdiv} \label{s:non-divergent}

Throughout this section make Standing Assumption~\ref{ass:fixed things}.  We apply the results in the previous section to the vertex groups of the geometric decomposition of a group in $\Mgen$.  If $\Gamma_v$ is such a vertex group, and it is hyperbolic, then it admits a complete, finite-volume hyperbolic structure, unique by Mostow--Prasad Rigidity, and this exhibits $\Gamma_v$ as a Kleinian group. We identify all conjugates of $\Gamma_v$ in the ambient 3-manifold group with the same Kleinian group. If $\Gamma_v$ is an \LSF vertex group, let $\overline{\Gamma_v}$ be the quotient (hyperbolic) $2$--orbifold group, and fix a (complete, finite-volume) hyperbolic structure on the orbifold, witnessing $\overline{\Gamma_v}$ as a Fuchsian group. There are many such hyperbolic structures, but any suffices for our purposes. The quotients of all conjugates of $\Gamma_v$ in the ambient 3-manifold group (modulo their centers) are identified with the same Fuchsian group.

The purpose of this section is to prove Theorem~\ref{t:NEW non-divergent}, required for the proof of Theorem~\ref{t:collapsing Divergent}.  The proof of Theorem~\ref{t:NEW non-divergent} is contingent on the technical result Theorem~\ref{t:fin gen}, proved in Appendix~\ref{app:edge-twist}.

Recall from Assumption~\ref{ass:fixed things} that $A_i$ is the \wapprox to the good relative generating set $A$ of the vertex group $V$ of the refined GGD $\Gdf$ of $L$.  By Definition~\ref{def:ggd} \walmost surely there are vertices $v_i$ of $T_i$ so $A_{i}$ fixes $v_i$.  In the above paragraph we identified $\Gamma_{v_i}$, the stabilizer of $v_i$ in $\Gamma_i$ (respectively $\overline{\Gamma_{v_i}}$, the quotient of $\Gamma_{v_i}$ by its center) with a specific Kleinian group (resp., a Fuchsian group). Therefore, there is a collapsed space associated to $\Gamma_{v_i}$, which we denote by $\mc{C}_i$ in both cases.   

\begin{definition} [\Cdiv]  \label{d:P-divergent}
Let $\Gamma_{v_i}$ be the vertex stabilizer of $v_i$  in $\Gamma_i$ and let $\mc{C}_i = \Coll(\Gamma_{v_i})$ (respectively, $\mc{C}_i = \Coll(\overline{\Gamma_{v_i}})$) be the associated collapsed space.  We set
\[	\| \phi_i \|_{\Coll, v} = \inf_{x\in \mc{C}_i} \max_{s \in A_{i}} d_i \left( s \cdot x, x \right)	,	\]
which is defined \walmost surely.
The sequence $\left( \phi_i \right)$ is {\em \Cdiv with respect to $\Gd$} if $\wlim\|\phi_i\|_{\Coll, v}=\infty$ for some vertex $v$ of $\Gdf$.
\end{definition}

We recall the statement of Theorem~\ref{t:NEW non-divergent}.

\nondivtheorem*

\begin{definition}
    A graph of groups is \emph{star-like} if its underlying graph is a bipartite graph where one color of vertices consists of a single vertex.
\end{definition}

Recall the definition of the stable center of $V$ from Definition~\ref{def:stable center}.
When $(\phi_i)$ is not \Cdiv with respect to $\Gd$, we obtain the following crucial information about the vertex group $V$ of $\Gdf$:

\begin{proposition}\label{p:cut point splitting}
Suppose $(\phi_i)$ is not \Cdiv with respect to $\Gd$. Then $\V$ admits a splitting $\D$ so:
\begin{enumerate}
\item\label{eq:star} $\D$ is star-like;
\item\label{eq:stable center} $\SC(\V)$ is contained in each edge group of $\D$;
\item\label{eq:non-central} The non-central vertices of $\D$ are in bijection with the edges  in $\Gdf$ adjacent to $\V$. This bijection induces isomorphisms of associated groups.  In particular the vertex groups associated to the non-central vertices are abelian, and the edge groups of $\D$ are all abelian; and
\item\label{eq:rank 2} For each the edge group $H$ of $\D$, $H/\SC(\V)$ has rank at most two. 
\end{enumerate}
Furthermore, if $\V$ is of hyperbolic type, and a stably parabolic subgroup $P$ of $V$ is not conjugate into an edge group of $\Gdf$ adjacent to $V$, then $P$ is finitely generated.
\end{proposition}

In Subsection~\ref{ss:proof of non-div} we prove Theorem~\ref{t:NEW non-divergent}, assuming Proposition~\ref{p:cut point splitting}.  Subsection~\ref{ss:cut point} is devoted to the proof of Proposition~\ref{p:cut point splitting}.

\subsection{The proof of Theorem~\ref{t:NEW non-divergent}}\label{ss:proof of non-div}
In this section, we assume that Proposition~\ref{p:cut point splitting}.

As a consequence of Proposition \ref{p:cut point splitting}, each vertex group of $\mathbb{G}$ also enjoys the properties \eqref{eq:star}--\eqref{eq:rank 2} from the conclusion of Proposition \ref{p:cut point splitting} enjoyed by $\V$:

\begin{corollary}\label{p:cut point splitting-c}
Suppose $(\phi_i)$ is not \Cdiv with respect to $\Gd$.  For any vertex group $\Gd_v$ of $\Gd$ there is a splitting $\D(v)$ of $\Gd_v$ so that
\begin{enumerate}
\item\label{eq2:star} $\D(v)$ is star-like;
\item $\SC(\Gd_v)$ is contained in each edge group of $\D(v)$;
\item\label{eq2:non-central} The non-central vertices of $\D(v)$ are in bijection with the edges  in $\Gd_v$ adjacent to $\Gd_v$. This bijection induces isomorphisms of associated groups.  In particular the vertex groups associated to the non-central vertices are abelian, and the edge groups of $\D$ are all abelian; and
\item\label{eq2:rank 2} For each the edge group $H$ of $\D(v)$, $H/\SC(\Gd_v)$ has rank at most two. 
\end{enumerate}
Furthermore, if $\Gd_v$ is of hyperbolic type, and a stably parabolic subgroup $P$ of $\Gd_v$ is not conjugate into an edge group of $\Gd$ adjacent to $\Gd_v$, then $P$ is finitely generated.
\end{corollary}

\begin{proof}
Notice that if $\Gd_v$ has hyperbolic type then it is a vertex group in $\Gdf$ as well as in $\Gd$ (by the construction of $\Gdf$ in Definition~\ref{def:refined ggd}), and therefore there is nothing to prove in this case.  Thus, we may assume that $\Gd_v$ is of \LSF type, and in particular the ``furthermore" part of the statement has been proved.

Let $\mathbb{R}(v)$ be relative $4$--Linnell decomposition of $\Gd_v$ (see Corollary~\ref{cor:lindecomp}).
Let $\V_1, \dots, \V_t$ be the vertex groups of $\mathbb{R}_v$.  Notice that since the splitting $\Gdf$ is obtained from $\Gd$ by refinement, the vertex groups $V_i$ are vertex groups of $\Gdf$ (recall the definition of the Linnell refinement $\Gdf$ from Definition~\ref{def:refined ggd}).  Therefore, Proposition~\ref{p:cut point splitting} applies to the groups $\V_1, \ldots , \V_t$.  Let $\D_1, \dots, \D_t$ be the corresponding splittings given by applying Proposition \ref{p:cut point splitting} to $\V_1, \ldots , \V_t$, respectively. Refine $\mathbb{R}(v)$ by $\D_1, \dots, \D_t$.  Notice that for each pair $(w,e)$ of vertex and edge in $\mathbb{R}(v)$, where $\V_i$ is the vertex group of $w$, we obtain an edge in $\D_i$ and a leaf vertex group whose associated group is the local group of $e$ in $\mathbb{R}(v)$.  Thus, each edge in $\mathbb{R}(v)$ becomes a path of three edges, and we collapse all such paths.  The resulting splitting is the required splitting $\mathbb{D}(v)$. 
\end{proof}

The following is the key definition required for our proof of Theorem~\ref{t:NEW non-divergent}.

\begin{restatable}{definition} {defedgetwist}\label{def:edge-twist}
A graph of groups $\E$ is {\em edge-twisted} if:

The underlying graph of $\E$ is bipartite with colors $A$ and $B$.  Type $A$ vertices have valence $2$, and abelian vertex groups (thus the edge groups of $\E$ are also abelian).  Let $W$ be a Type $A$ vertex group of $\E$ and let $E_1$ and $E_2$ be the images in $W$ of the adjacent edge groups.  There are subgroups $K_j \le E_j$ (for $j = 1,2$) so that
\begin{enumerate}
\item $K_1 \cap K_2 = \{ 1 \}$; and
\item For $j = 1,2$, the group $E_j/K_j$ is finitely generated.
\end{enumerate}
\end{restatable}

We now define a new splitting $\mathbb{K}$ of $L$.  The underlying graph of $\mathbb{K}$ is the barycentric subdivision of the underlying graph $\Lambda(\Gd)$ of $\Gd$, so the vertices correspond to cells in $\Lambda(\Gd)$, and edges correspond to pairs $(w,e)$, where $w$ is a vertex, $e$ is an edge, and $e$ is adjacent to $w$.  The vertex group of a vertex corresponding to a vertex of $\Lambda(\Gd)$ is the central vertex group of the splitting of the corresponding vertex group of $\Gd$ arising from Corollary~\ref{p:cut point splitting-c}. The vertex group of a vertex corresponding to an edge of $\Lambda(\Gd)$ is the corresponding edge group of $\Gd$.  For an edge corresponding to the pair $(w,e)$, there is a corresponding edge group in the splitting $\D(w)$ of $\Gd_w$ coming from Corollary~\ref{p:cut point splitting-c}, and this is the edge group in $\mathbb{K}$.  The edge-to-vertex maps naturally come from the splittings $\D(w)$.

\begin{proposition}\label{t:tree output}
The graph of groups $\mathbb{K}$ is an edge-twisted splitting, where the Type A vertices correspond to edges in $\Lambda(\Gd)$ and the Type B correspond to vertices in $\Lambda(\Gd)$. 
\end{proposition}
\begin{proof}
It is clear from the construction of $\mathbb{K}$ that Type A vertices are valence $2$
and that their vertex groups are abelian (since they are the edge groups of $\Gd$).
By Lemma~\ref{JSJ properties} and Definition~\ref{def:ggd}\eqref{ggd:edges} $\SC(W_1)\cap \SC(W_2)=\{1\}$ for adjacent vertex groups $W_1, W_2$ of $\Gd$.  For the $K_i$ from Definition~\ref{def:edge-twist} we choose the stable center of the vertex groups of $\Gd$.  With this choice, the first condition of Definition~\ref{def:edge-twist} is satisfied. The last condition of Definition~\ref{def:edge-twist} is satisfied because of the Property~\eqref{eq2:rank 2} from Corollary~\ref{p:cut point splitting-c}.
\end{proof}

The following Theorem~\ref{t:fin gen} is proved in Appendix~\ref{app:edge-twist}. 

\begin{restatable}{theorem}{edgetwistfg} \label{t:fin gen}
Let $\E$ be a finite edge-twisted graph of groups so that $\pi_1(\E)$ is finitely generated.  The Type B vertex groups of $\E$ are finitely generated.
\end{restatable}

We finish this subsection by proving Theorem~\ref{t:NEW non-divergent}, assuming Proposition~\ref{p:cut point splitting} and Theorem~\ref{t:fin gen}.

\begin{proof}[Proof of Theorem~\ref{t:NEW non-divergent}]
Recall that $L$ is an $\Mgen$--limit group defined by a sequence $(\phi_i)$ which is not \Tdiv, and that there is a GGD $\Gd$ of $L$ with respect to which $(\phi_i)$ is not \Cdiv.  We are required to prove that all stably parabolic subgroups of $L$ are finitely generated.

Let $\mathbb{K}$ be the splitting defined in the paragraph before Proposition~\ref{t:tree output}. By Proposition~\ref{t:tree output} and Theorem~\ref{t:fin gen}, Type B vertex groups of $\mathbb{K}$ are finitely generated.  Thus, each Type B vertex group  which is a subgroup of an \LSF vertex groups of $\Gd$ is an $\Mlsf$--limit group. By Corollary~\ref{cor:SF limit fp+A-slender}, all abelian subgroups of these vertex groups are finitely generated, so edge groups adjacent to these vertex groups are finitely generated. The other edge groups are adjacent to a hyperbolic type vertex group of $\mathbb{K}$ and are finitely generated by the third property of Corollary~\ref{p:cut point splitting-c}, so all vertex groups of $\mathbb{K}$ are finitely generated also. Hence all vertex groups of $\Gd$ are finitely generated.  A stably parabolic subgroup of $L$ is finitely generated by Corollary~\ref{cor:SF limit fp+A-slender} if it is in an \LSF vertex group or by Corollary~\ref{p:cut point splitting-c} if it is in a hyperbolic type vertex group. 
\end{proof}

\subsection{The cut point tree} \label{ss:cut point}

The rest of the section is devoted to the proof of Proposition~\ref{p:cut point splitting}.  

Recall that the assumption of Proposition~\ref{p:cut point splitting} is that the defining sequence $(\phi_i)$ for $L$ is not \Cdiv with respect to $\Gd$.  We consider how $\V$ maps into the collapsed spaces associated to the vertex groups of the geometric decompositions of the $\Gamma_i$ (with the chosen Kleinian/Fuchsian structures).  Recall that in Standing Assumption~\ref{ass:fixed things} we fixed a finite good relative generating set $A$of $V$, and an \wapprox $(A_i)$ to $A$.

By the choice of $A_i$, \walmost surely $A_{i}\subset\Gamma_{v_i}$.  Recall that we fixed $\Coll_i$ to be either $\Coll(\Gamma_{v_i})$ or $\Coll(\overline{\Gamma_{v_i}})$, depending on whether $v_i$ is of hyperbolic or \LSF type. For \walmost every $i$, fix $\bp_i \in \Coll_i$ satisfying
\[	\max_{g \in A_{i}} \left\{ d_i( g \cdot\bp_i, \bp_i) \right\} \le \| \phi_i \|_{\Coll, v} + \frac{1}{i}	.	\]
By Definition~\ref{def:good generating set} $\V$ is generated by $A$ and the adjacent edge groups.  
Let $S_v$ be a (possibly infinite) generating set for $\V$ consisting of $A$ and elements from the adjacent edges groups.  Choose a lift $\widetilde{S}_v \subset G$ of $S_v$.

Let $(\Coll_{\infty,v}, \bp)$ be the $\omega$--limit of the collapsed spaces $(\Coll_i, \bp_i)$. We next show that $\{\phi_i\}$ induces an isometric action of $V$ on $(\Coll_{\infty,v}, \bp)$. This construction is standard in the case where $V$ is finitely generated. Since we do not (yet) know that $V$ is finitely generated, we need the following.

\begin{lemma}\label{lem:lim_act}
Let $(X_i, d_i)$ be a sequence of metric spaces with basepoints $o_i\in X_i$. Suppose that a group $G$ acts by isometries on each $X_i$ and $G$ is generated by a (possibly infinite) subset $S$ such that for all $s\in S$,
\[
\wlim d_i(o_i, s\cdot o_i)<\infty.
\]

Then $G$ acts by isometries on $\wlim(X_i, o_i)$. 
\end{lemma}

\begin{proof}
By definition, a point in $\wlim(X_i, o_i)$ is the equivalence class of a sequence $(x_i)$ so that (i) each $x_i\in X_i$; and (ii) $\wlim d_i(o_i, x_i)<\infty$. Two such sequences $(x_i)$ and $(x_i')$ are equivalent if $\wlim d_i(x_i, x_i')=0$.

Let $g, h\in G$ be such that $\wlim d_i(o_i, g\cdot o_i)<\infty$ and $\wlim d_i(o_i, h\cdot o_i)<\infty$. Then 
\begin{align*}
\wlim d_i(o_i, gh\cdot o_i)&\leq \wlim d_i(o_i, g\cdot o_i)+\wlim d_i(g\cdot o_i, gh\cdot o_i)\\&=\wlim d_i(o_i, g\cdot o_i)+\wlim d_i(o_i, h\cdot o_i)<\infty.
\end{align*}

Since $S$ generates $G$, it follows by induction on word length that for all $g\in G$, $\wlim d_i(o_i, g\cdot o_i)<\infty$. 

Now let $g\in G$ and let $(x_i)$ be a sequence with $x_i\in X_i$ and such that $\wlim d_i(o_i, x_i)<\infty$. Then  
\begin{align*}
\wlim d_i(o_i, g\cdot x_i)&\leq \wlim d_i(o_i, g\cdot o_i)+\wlim d_i(g\cdot o_i, g\cdot x_i)\\ 
&=\wlim d_i(o_i, g\cdot o_i)+\wlim d_i(o_i, x_i)<\infty
\end{align*}

That is, the sequence $(g\cdot x_i)$ defines a point of $\wlim(X_i, o_i)$. If $(x_i')$ is a sequence which is equivalent to $(x_i)$, then 
\[
\wlim d_i(g\cdot x_i, g\cdot x_i')=\wlim d_i(x_i, x_i')=0.
\]

So the sequences $(g\cdot x_i)$ and $(g\cdot x_i')$ are equivalent. Thus, $(x_i)\to (g\cdot x_i)$ is a well-defined map on $\wlim(X_i, o_i)$. These maps clearly define an action of $G$ on $\wlim(X_i, o_i)$. Finally, we observe that the metric on $\wlim(X_i, o_i)$ is defined by $d((x_i), (y_i))=\wlim d_i(x_i, y_i)$. It then follows that 
\[
d\left((g\cdot x_i), (g\cdot y_i)\right)=\wlim d_i\left(g\cdot x_i, g\cdot y_i\right)=\wlim d_i\left(x_i, y_i\right)=d\left((x_i), (y_i)\right).
\]

Hence, $G$ acts on $\wlim(X_i, o_i)$ by isometries.
\end{proof}

\begin{lemma}
$\{\phi_i\}$ induces an isometric action of $V$ on $(\Coll_{\infty,v}, \bp)$. 
\end{lemma}

\begin{proof}
We apply Lemma \ref{lem:lim_act} to the generating set $A \cup \E(\V)$, where $\E(\V)$ are the edge groups of $\Gdf$ adjacent to $\V$.

Let $\bp=[\{\bp_i\}]\in \Coll_{\infty, v}$ be the basepoint of $\Coll_{\infty,v}$.   Since $\| \phi_i \|_{\Coll, v}$ has finite $\omega$-limit, for all $g \in A_{i}$ we have
\begin{equation}\tag{$\dagger$}\label{eq:nondiv6}	
\wlim d_i(\bp_i, g \cdot \bp_i) < \infty	.	\end{equation}
Let $H\in \E(\V)$, and suppose $h \in H$ and $g\in A\cap H$ (note that such a $g$ exists by Definition~\ref{def:good generating set}). Let $(h_i)$ and $(g_i)$ be \wapproxs of $h$ and $g$, respectively. Then \walmost surely $g_i$ and $h_i$ are parabolic fixing the same point in $\Coll_i$.  By Lemma~\ref{l:almost equal trans len for para}, there exists some constant $D$ so that \walmost surely $\left| d_i(\bp_i,h_i\bp_i) - d_i(\bp_i,g_i\bp_i) \right| < D$.  Therefore, by \eqref{eq:nondiv6}
\[	\wlim d_i(\bp_i, h_i \cdot \bp_i) < \infty	,	\]
as required.
\end{proof}

Recall that we fixed $S_v$ and $\widetilde{S_v}$ above Lemma~\ref{lem:lim_act}.
Let $\Cay(\V, S_v)$ be the (possibly locally infinite) Cayley graph of $\V$ with respect to $S_v$. For each $s \in S_v$, choose a geodesic $[\bp, s\cdot\bp]\subset \Coll_{\infty, v}$ and fix a homeomorphism between $[1, s]\subset \Cay(\V, S_v)$ and $[\bp, s\cdot\bp]$. Let $f_\infty: \Cay(\V, S_v)\rightarrow  \Coll_{\infty, v}$ be the $V$-equivariant map extending the above homeomorphisms. So $f_\infty$ maps $[g, gs]\subset \Cay(\V, S_v)$ to $g[\bp, s\cdot\bp]$ for any $g\in V$. 

Let $\mc{L}$ be the set of limits of points in $\Coll_{\infty, v}$ in the collapsing locus which are quotients of horoballs. Let $Q$ be the
quotient of $\Cay(\V, S_v)$ obtained by identifying points that are mapped by $f_\infty$ to the same point in $\mc{L}$. The $V$-action on $\Cay(\V, S_v)$ descends to an action
of $V$ on $Q$. The map $f_\infty$ induces a continuous map from $Q$ to $\Coll_{\infty, v}$, which we still denote by $f_\infty$. Note that $f_\infty: Q\rightarrow \Coll_{\infty, v}$ is $V$-equivariant and injective on $f_\infty^{-1}(\mc{L})$. Points in $f_\infty^{-1}(\mc{L})\subset Q$ are called \emph{marked points} of $Q$.

\begin{lemma}\label{l:1-1 map}
The map $f_\infty \co Q \to \Coll_{\infty, v}$ is well-defined, continuous and $\V$--equivariant.  The map $f_\infty$ restricts to an injective map from the marked points of $Q$ to the set of limits of parabolic collapsed points in $\Coll_{\infty,v}$.  Moreover, for any $p \in \Coll_{\infty,v}$ which is a limit of parabolic collapsed points, $f_\infty^{-1}(p)$ is either a single marked point in $Q$ or is empty.
\end{lemma}

Recall that $\mathcal{H}_{\V}$ is the set of stably parabolic subgroups of $\V$.  Suppose $H \in \mc{H}_{\V}$.  If $g \in H$ and $(g_i)$ is an \wapprox of $g$, then \walmost surely $g_i$ is contained in $\Gamma_{v_i}$ and fixes a parabolic collapsed point $\xi_i(g)$ in $\Coll(\Gamma_{v_i})$. Note that the point $\xi_i(g)$ \walmost surely depends only on $g$ and not on the choices of $(g_i)$.
In fact, so long $g \not\in \SC(\V)$, by Definition~\ref{def:stably para}, the point $\xi_i(g)$ depends only on $H$ (different elements of $H \smallsetminus \SC(\V)$ must \walmost surely fix the same edge in the geometric decomposition of $\Gamma_i$, and so have same parabolic fixed point).

\begin{lemma}\label{l: fixed point of H}
The sequence $\{\xi_i(g)\}$ defines a point $\xi(g)$ in $\Coll_{\infty, v}$ fixed by $g$. 
\end{lemma}

\begin{proof}
By Lemma~\ref{l:almost equal trans len for para} and the fact that the translation length of $\phi_i$ with respect to $S_v$ and $\bp_i$ is bounded, the distance between $\xi_i(g)$ and $\bp_i$ is \walmost surely bounded independent of $i$. The lemma follows. 
\end{proof}

\begin{definition}

An element $g \in H$ has \emph{Property $\Boun$} if for some $x \in Q$ there is a path $\tau$ in $Q$ between $x$ and $g \cdot x$ such that $\xi(g) \not\in f_\infty(\tau)$.\
\end{definition}

\begin{lemma}\label{l:bounded parabolic image}
Let $H \in \mc{H}_{\V}$ be so that every $h \in H$ has Property $\Boun$. Then
$H/H \cap \SC(\V)$ is a free abelian group of rank at most two. 
\end{lemma}

\begin{proof}
Let $g\in H$ and fix an \wapprox $(g_i)$ of $g$. By the definition of Property $\Boun$, there is a path $\tau$ in $Q$ from $x$ to $g\cdot x$ such that $\xi(g) \not\in f_\infty(\tau)$. Suppose $f_\infty(x)$ is represented by $\{x_i\in \mc{C}_i\}$.  Then $\tau$ is the limit of a
sequence of paths $\{\tau_i\}$, where $\tau_i$ is a path in $\mc{C}_i$ connecting $x_i$ and $g_i\cdot x_i$ and $\xi_i(g)$ is not on $\tau_i$ \walmost surely.
This path $\tau_i$ can be lifted to a path $\widetilde{\tau_i}$ in $\bH^\ast$ connecting $\widetilde x_i$ and $g_i.\widetilde x_i$, where is $\widetilde x_i$ is the pre-image of $x_i$ under the natural projection from $\bH^\ast$ to $\Coll(\Gamma_{v_i})$. Let $U_i \subset \bH^\ast$ be the maximal horoball in $\bH^\ast$ that is collapsed to make the collapsed point $\xi_i$. Let $\widetilde y_i$ be the closest point projection of $\widetilde x_i$ to $U_i$. Then $g_i.\widetilde{y}_i$ is the closest point projection of $g_i.\widetilde x_i$ to $U_i$. Moreover, the path $\widetilde{\tau_i}$ projects to a path in the boundary of $U_i$ connecting $\widetilde y_i$ and $g_i.\widetilde y_i$.   The length of $\tau_i$ is \walmost surely bounded.  Therefore, \walmost surely there is a bound on the number of components of the pre-image of the collapsing locus which $\widetilde{\tau_i}$ intersects.  Then \walmost surely $\widetilde{\tau_i}$ does not intersect $U_i$. If $C$ is a component of the pre-image of the collapsing locus other than $U_i$ then by Lemma~\ref{Bounded projection} the projection of $C$ to $U_i$ has uniformly bounded diameter. Thus there exists $N$ so that \walmost surely the projection of $\widetilde{\tau_i}$ to $U_i$ has length at most $N$.  Hence, the translation length of $g_i$ on $U_i$ is \walmost surely bounded.   

Let $P_i$ be the subgroup of $\Gamma_{v_i}$ fixing $\xi_i(g)$.  Then $P_i$ is isomorphic to either $\mathbb{Z}$ or to $\mathbb{Z}^2$.  Choose a minimal cardinality generating set (basis) for $P_i$ so that the sum of their translation lengths on the corresponding horoball/horocycle is as small as possible.  This choice gives a specified isomorphism $\psi_i$ from $P_i$ to $\Z$ or $\mathbb{Z}^2$.  

First suppose that $P_i \cong \Z^2$.  Using the above isomorphism, \walmost surely we have $\psi_i(\phi_i(g_i))\in\Z^2$. The upper bound on the translation length of $g_i$ on $U_i$ imply that there exist $a, b\in\mathbb{Z}$, independent of $i$, so that $\psi_i(\phi_i(g_i))=(a, b)$ \walmost surely. Note that $(a, b)$ does not depend on the choice of \wapprox $(g_i)$. 

By Proposition~\ref{p:props of U^K}.\eqref{eq:lower bound}, as long as $g\notin \SC(\V)$ there is a lower bound, independent of $i$, on the translation length of $g_i$ on $U_i$. 
Defining $\phi'(g)=(a, b)$ for each $g\in H$ induces a homomorphism $H\to\Z^2$, and the lower bound on translation length for $g\notin \SC(\V)$ implies that this homomorphism has kernel $\SC(\V)$.  

The case where $P_i \cong \Z$ is entirely similar, completing the proof.   
\end{proof}

\begin{lemma}\label{l:not mark point}
Let $H \in \mc{H}_{\V}$. Then either $H/H \cap \SC(\V)$ has rank at most two or $H$ stabilizes a marked point in $Q$. 
\end{lemma}
\begin{proof}
Let $g\in H$ and $\xi(g)$ be the point in $\Coll_{\infty, v}$ fixed by $g$ as in Lemma \ref{l: fixed point of H}.  Then $H$ stabilizes $\xi(g)$. If $\xi(g)$ is the image of a marked point under the map $f_\infty$, then by Lemma~\ref{l:1-1 map} $H$ is stabilizes a marked point of $Q$. Now suppose $\xi(g)$ is \emph{not} the image of any marked point under the map $f_\infty$. Then $\xi(g)$ is \emph{not} in the image of $f_\infty$ by the last assertion in Lemma~\ref{l:1-1 map}. Hence $\xi(g) \not\in f_\infty(\tau)$ for any path in $Q$, and so every element of $H$ has Property $\Boun$. Therefore, by Lemma~\ref{l:bounded parabolic image}, $H/H \cap \SC(\V)$ has rank at most two.  
\end{proof}

By Lemma~\ref{l:not mark point}, to further understand the elements of $\mc{H}_{\V}$, we study the case when they are stabilizers of marked points in $Q$.  

\begin{lemma}\label{l:not cut point}
Let $H \in \mc{H}_{\V}$. Suppose $H$ stabilizes a marked point $z\in Q$ which is not a cut point. Then every element of $H$ has Property $\Boun$ and hence $H/H \cap \SC(\V)$ has rank at most two .  
\end{lemma}
\begin{proof}
Let $g\in H$ and $x\in Q \smallsetminus \{z\}$.  Since $z$ is not a cut point, there is a path $\tau$ in $Q \smallsetminus \{z\}$ from $x$ to $g \cdot x$.  By Lemma~\ref{l:1-1  map}, $f_\infty(z) \not\in f_\infty(\tau)$. Note that $\xi=f_\infty(z)$ is the fixed point of $H$ in $\Coll_{\infty, v}$ given by Lemma~\ref{l: fixed point of H}. Hence each $g\in H$ has Property $\Boun$, and so by Lemma~\ref{l:bounded parabolic image} $H/H \cap \SC(\V)$ has rank at most two, as required.  
\end{proof}

We now consider stably parabolic subgroups which fix marked points that are not cut points.  To that end, be build a new tree.

Let $M\subset Q$ be the collection of all marked points that are cut points. Consider the {\em cut-point tree}  $\Tcut$  associated to $(Q, M)$.  The vertices of $\Tcut$ are:

\begin{enumerate}
\item Points in $M$; and
\item The maximum connected subsets of $Q$ not separated by an element of $M$,
\end{enumerate}
Vertices of the first type are cut points, and vertices of the second type are \emph{blocks}.  Edges of $\Tcut$ correspond to the inclusion of a cut point in a block. Clearly, $\V$ acts on $\Tcut$.

\begin{lemma}\label{l:cut point tree edge}
Let $z \in M$ and let $B$ be a block containing $z$.  Let $e=[z\subset B]$ be the corresponding edge of $\Tcut$, and let $H$ be the stabilizer in $\V$ of $e$. Then $H/\SC(\V)$ has rank at most two.   
\end{lemma}

\begin{proof}
Denote the stabilizer of $z$ by $H'$. By Lemma~\ref{l:1-1 map}, $H'$ fixes $\xi=f_\infty(z)\in \Coll_{\infty, v}$,  a limit point of parabolic points $\xi_i\in\Coll(\Gamma_{v_i})$. In particular, $H' \in \mc{H}_{\V}$, hence $H\in \mc{H}_{\V}$. By Lemma~\ref{l:bounded parabolic image}, it suffices to show that $H$ consists entirely of elements satisfying Property $\Boun$.  Let $g\in H$ and $x\in B$. Since $g$ fixes $B$ as a set, we know $g \cdot x\in B$. By definition of blocks, $B \smallsetminus \{z\}$ is connected. Hence there is a path $\tau$ connecting $x$ and $g \cdot x$ and $\tau$ does not intersect $z$. Therefore, $g$ has Property $\Boun$.     
\end{proof}

We are now ready to prove Proposition~\ref{p:cut point splitting}.  As we showed above, this is enough to complete the proof of Theorem~\ref{t:NEW non-divergent}, modulo the proof of Theorem~\ref{t:fin gen} in Appendix~\ref{app:edge-twist}.

\begin{proof}[Proof of Proposition~\ref{p:cut point splitting}]
Let $\D_0$ be the splitting of $\V$ dual to the action of $\V$ on $\Tcut$. Since$\SC(\V)$ acts trivially on $\Tcut$,  $\SC(\V)$ is contained in all the edge groups of $\D_0$. The underlying graph of $\D_0$ is bipartite, with one type of vertex corresponding to cut point vertices in $\Tcut$ and the other type corresponding to block vertices.  For each cut point vertex of $\D_0$, fold all the adjacent edges together to one edge, and denote the resulting splitting by $\D'$.  Now, each cut point vertex of $\D'$ has one adjacent edge and $\D'$ is still connected, so the underlying graph of $\D'$ is a star.  

Collapse the cut point vertices of $\D'$ whose stabilizers are not conjugate into edge groups  of $\Gdf$ adjacent to $\V$.  Add vertices to the resulting graph of groups corresponding to edge groups adjacent to $\V$ not corresponding to vertices of $\D'$ (with identical edge and vertex groups). It follows that the resulting graph of groups $\D$ satisfies \eqref{eq:non-central}.  Since $\D'$ satisfies \eqref{eq:star} and \eqref{eq:stable center}, so does $\D$. Statement \eqref{eq:rank 2} follows from Lemmas~\ref{l:not mark point},  ~\ref{l:not cut point} and ~\ref{l:cut point tree edge}.

It remains to prove the ``Furthermore" assertion of the proposition.  To that end, suppose $\V$ is of hyperbolic type and that $P \in \mc{H}_{\V}$ is not conjugate into an adjacent edge group.  Then $\SC(\V) = \{ 1 \}$.  If $P$ does not stabilize a cut point vertex of $\Tcut$ then $P$ is finitely generated by Lemma~\ref{l:not mark point} and Lemma \ref{l:not cut point}.  If $P$ does stabilize a cut point vertex of $\Tcut$, then $P$ is in a valence-one vertex group of $\D'$ whose adjacent edge group is finitely generated by Lemma \ref{l:cut point tree edge}. Refining $\Gdf$ by $\D'$ yields a splitting $\mathbb{K}'$ of $L$. Since $P$ is not contained in the conjugate of any edge group of $\Gdf$, the corresponding vertex group in $\mathbb{K}'$ containing $P$ is still valence-one with the same adjacent edge group, so it is finitely generated since $L$ is. In all cases, $P$ is finitely generated and the proof of Proposition~\ref{p:cut point splitting} is complete.
\end{proof}

\section*{Interlude: Summary of what remains}

The goal of the remainder of paper (other than Appendix~\ref{app:edge-twist}) is to prove Theorem~\ref{t:technical outline}.  If our defining sequence is not \Cdiv with respect to some GGD, then we may take $k=0$ in Theorem~\ref{t:technical outline}, and take the GGDs of $S_1 = S_0$ to be the same.  Therefore, there is only anything to prove in case our defining sequence $(\phi_i)$ of $L$ \emph{is} \Cdiv with respect to some GGD $\Gd$ of $L$.  

Our approach to proving Theorem~\ref{t:technical outline} will be familiar to the experts.  In Section~\ref{s:R-trees} we build limiting $\R$-trees from rescaled actions on the collapsed spaces, which allows us to use the Rips machine to find splittings of $L$.  In Section~\ref{s:JSJ} we recall and prove the required properties of JSJ decompositions and modular automorphisms.  In Section~\ref{s:res-short} we undertake the version of Sela's ``shortening argument" in this context, and finally prove Theorem~\ref{t:technical outline}.

As we stated before, the construction of the collapsed spaces was made to make these arguments work in as ``standard" a way as possible.  Thus, while there are many technical difficulties in the next three sections, not much in these sections will surprise the experts.  An exception to this are the ``squared Nielsen transformations'' in Section~\ref{s:JSJ}, which are a new kind of move required to shorten our homomorphisms.  These arise from our analysis of edge groups arising as limits of infinite dihedral subgroups of Fuchsian groups coming from an arc joining two cone points of order $2$, and how to shorten in the presence of such edge groups.

\section{Limits and $\R$--trees} \label{s:R-trees}

Throughout this section make Standing Assumption~\ref{ass:fixed things}. Suppose further that $(\phi_i)$ is \Cdiv with respect to $\Gd$ (see Definition~\ref{d:P-divergent}) and that the vertex $v$ of $\Gdf$ associated to $V$ is one for which $\wlim \| \phi_i\|_{\Coll,v} = \infty$.   The space $\Coll_{\infty,v}$ is defined in Definition~\ref{d:Pinfty} below.

The goal of this section is to prove the following theorem, which is the only result from this section needed in future sections.   We remind the reader that $\SC(\V)$ is the stable center of $\V$ (see Definition~\ref{def:stable center}) and $\overline{\V} = \V/\SC(\V)$.

\begin{theorem} \label{t:Rtree summary}
\ 
\begin{enumerate}
\item\label{eq:rtree} $\Coll_{\infty,v}$ is an $\R$--tree equipped with a nontrivial isometric $V$--action.  The stable center $\SC(\V)$ acts trivially on this tree, and so there is an induced action of $\overline{\V}$.
\item\label{eq:para unique fp} Let $H \in \mc{H}_{\V}$.  There exists $x_H \in \Coll_{\infty,v}$, fixed by $H$. Each nontrivial element of $H \smallsetminus \SC(\V)$ fixes only the point $x_H$.
\item\label{eq:ell unique fp}  If $g \in V$ and $(g_i)$ is an \wapprox to $g$ so \walmost surely $g_i$ is (nontrivial and) elliptic then $g$ fixes a unique point in $\Coll_{\infty,v}$.
\item\label{eq:triv trip} The $\overline{\V}$--action on $\Coll_{\infty,v}$ has trivial tripod stabilizers. 
\item\label{eq:abel seg} The $\overline{\V}$--action on $\Coll_{\infty,v}$ has abelian segment stabilizers.
\end{enumerate}
\end{theorem}

The above theorem, together with the Rips machine, implies that the $\overline{\V}$--action on $\Coll_{\infty,v}$ admits a graph of actions decomposition with simplicial, axial  and Seifert type vertex actions, which induces an abelian splitting of $\overline{\V}$. The splitting of $\V$ (rel $\mc{H}_{\V}$) induced by this splitting also has abelian edge groups. \footnote{See \cite{Guirardel:Rtrees} for the definitions of graphs of actions, simplicial vertex actions, axial vertex actions and vertex actions of Seifert type, as well as a statement of the Rips machine (Theorem 5.1 in that paper).}

The remainder of this section is dedicated to the proof of Theorem~\ref{t:Rtree summary}.  Recall the definitions from Section~\ref{s:geo decomp}, particularly Definitions~\ref{def:ggd} and~\ref{def:good generating set}.

Let $A$ be the good relative generating set for $\V$ from Assumption~\ref{ass:fixed things} and let $(A_{i})$ be the chosen \wapprox of $A$.  Let $v_i\in T_i$ be the sequence of vertices associated to $v$. Let $\Gamma_{v_i}\subset \Gamma_i$ be the stabilizer of $v_i$. Then since $\Gd$ is a GGD of $L$ with respect to $(\phi_i)$  \walmost surely $A_{i}\subset\Gamma_{v_i}$. As at the beginning of Section~\ref{s:non-divergent}, let $\Coll_i=\Coll(\Gamma_{v_i})$ (respectively $\mc{C}_i = \Coll(\overline{\Gamma_{v_i}})$) be the collapsed space associated to $\Gamma_{v_i}$ (depending on whether $v_i$ is of hyperbolic type or \LSF type. For \walmost every $i$, fix $\bp_i \in \Coll_i$ so that
\[	\max_{g \in A_{i}} \left\{ d_i( g . \bp_i, \bp_i) \right\} \le \| \phi_i \|_{\Coll, v} + \frac{1}{i}	.	\]

\begin{definition}
Let $\Coll_{\infty,v}^0$ be the ultra-limit (with respect to $\omega$) of the sequence $\left(\Coll_i, \frac{1}{\| \phi_i \|_{\Coll, v}} d_i, \bp_i \right)$.
\end{definition}

\begin{proposition}[Limiting $\R$--tree] \label{p:Limit R-tree}
The space $\Coll_{\infty, v}^0$ is a (pointed) $\R$--tree equipped with a non-trivial isometric $\V$--action. 
\end{proposition}

\begin{proof}
Most parts of this proposition are standard facts in the theory of ultra-limits and $\mathbb{R}$--trees (see, for example, \cite[$\S4$]{GH} and the references therein).  The only new thing here is to show that $\{\phi_i\}$ induces an isometric action of $\V$ on $\Coll_{\infty, v}^0$. This does not follow from the standard theory since $A$ may not generate $\V$.  However, $\V$ is generated by $A \cup \E(\V)$, where $\E(\V)$ are the edge groups of $\Gdf$ adjacent to $\V$.

Let $\bp=[\{\bp_i\}]\in \Coll_{\infty, v}^0$ be the basepoint of $\Coll_{\infty,v}^0$.   By the definition of $\| \phi_i \|_{\Coll,v}$, for all $g \in A_{v,i}$ 
\begin{equation}\tag{$\dagger$}\label{eq:nondiv}	
\wlim \frac{1}{\| \phi_i \|_{\Coll,v}} d_i(\bp_i, g \cdot \bp_i) < \infty	.	\end{equation}
Let $H\in \E(\V)$, and suppose $h \in H$ and $g\in A\cap H$ (note that such a $g$ exists by Definition~\ref{def:good generating set}). Let $(h_i)$ and $(g_i)$ be \wapproxs of $h$ and $g$, respectively. Then \walmost surely $g_i$ and $h_i$ are parabolic fixing the same point in $\Coll_i$.  By Lemma~\ref{l:almost equal trans len for para}, there exists some constant $D$ so that \walmost surely $\left| d_i(\bp_i,h_i\bp_i) - d_i(\bp_i,g_i\bp_i) \right| < D$.  Therefore, by \eqref{eq:nondiv}
\[	\wlim \frac{1}{\| \phi_i \|_{\Coll,v}} d_i(\bp_i, h_i \cdot \bp_i) < \infty	,	\]
as required.
\end{proof}

\begin{definition} \label{d:Pinfty}
Let $\Coll_{\infty,v}$ be the minimal $\V$--invariant subtree of $\Coll^0_{\infty, v}$.
\end{definition}
Theorem~\ref{t:Rtree summary}.\eqref{eq:rtree} follows immediately from Proposition~\ref{p:Limit R-tree}.

 Recall that $\overline{\V} = \V/\SC(\V)$, where $\SC(\V)$ is given by Definition~\ref{def:stable center}.  Each element of $\SC(\V)$ \walmost surely acts trivially on $\Coll_i$.  Thus we have the following result which completes Item~\eqref{eq:rtree} from Theorem~\ref{t:Rtree summary}.
\begin{corollary}
The $\V$--action descends to a non-trivial isometric $\overline{\V}$--action on $\Coll^0_{\infty,v}$, with minimal $\overline{\V}$--invariant subtree $\Coll_{\infty,v}$.
\end{corollary}

Theorem~\ref{t:Rtree summary}.\eqref{eq:para unique fp} and \eqref{eq:ell unique fp} follow from the next lemma. 

\begin{lemma} \label{l:fixed for ell/par}
Let $g\in \V$ and let $(g_i)$ be an \wapprox of $g$. Suppose that \walmost surely $g_i$ is parabolic (or nontrivial elliptic). Then $g$ fixes a unique point in $\Coll_{\infty, v}$.  If $g^1,g^2 \in \V$, $(g^{1}_i)$ and $(g^2_i)$ are \wapproxs of $g^1$ and $g^2$, respectively, and \walmost surely $g^{1}_i$ and $g^{2}_i$ are parabolic with the same fixed point then $g^1$ and $g^2$ fix the same point in $\Coll_{\infty,v}$.
\end{lemma}
\begin{proof}
Let $C_i$ be the component of the collapsing locus fixed by $g_i$ (or the projection in $\Coll_i$ of the fixed point of the elliptic $g_i$ in $\bH^\ast$). Consider the geodesic triangle with vertices $\bp_i, g_i\cdot \bp_i$ and $C_i$.  Let $x_i$ be the point on $[\bp_i, C_i]$ with $d_i(x_i, C_i)=(\bp_i\mid g\cdot \bp_i )_{C_i}$. Note that $2d_i(x_i, \bp_i)\leq d_i(\bp_i, g_i\cdot \bp_i)+4\delta$. Hence 
\begin{equation*}
\wlim \frac{1}{\| \phi_i \|_{\Coll,v}} d_i(\bp_i, x_i) < \infty
\end{equation*}  
Hence $\{x_i\}$ defines a point in $\Coll^0_{\infty, v}$. 
By hyperbolicity of $\Coll_i$, $d_i(x_i, g_i\cdot x_i)\leq \delta$. Hence $g$ fixes $x$. Therefore $g$ fixes a point in $\Coll_{\infty, v}$. We now show that $g$ fixes exactly one point in $\Coll_{\infty, v}$. Suppose $g$ fixes $y\in\Coll_{\infty, v}$ and $\left( y_i\in \Coll_i\right)$ represents $y$.  By Lemma~\ref{l:almost equal trans len for para},  $\left| 2d_{i}(y_i, C_i) - d_{i}(y_i, g_i\cdot y_i)\right| $ is \walmost surely bounded independent of $i$. Since $g$ fixes $y$, $$\wlim d_{i}(y_i, g_i\cdot y_i)/\| \phi_i \|_{\Coll, v}=0$$ so $\wlim d_{i}(y_i, C_i)/\| \phi_i \|_{\Coll, v}=0$, and $y = [\{C_i\}]\in \Coll_{\infty, v}$. 

The second assertion follows immediately from the above argument and the assumption that \walmost surely $g^1_i$ and $g^2_i$ fix the same point in the collapsed space $\Coll_i$.
\end{proof}

\begin{corollary} \label{c:arc stab loxo}
Suppose $g \in \V \smallsetminus\SC(\V)$ fixes a non-degenerate arc in $\Coll_{\infty,v}$. If $(g_i)$ is an \wapprox of $g$ then \walmost surely $g_i$ is loxodromic.
\end{corollary}

Theorem~\ref{t:Rtree summary}.\eqref{eq:triv trip} follows from Lemma~\ref{l:fixed for ell/par} and the following lemma.

\begin{lemma}\label{l:fixed for loxo}
Let $g\in \V$ and $(g_i)$ be an \wapprox of $g$. Suppose that \walmost surely $g_i$ is loxodromic. The fixed point set of $g$ in $\Coll_{\infty, v}$ is either empty or a single geodesic.   In case it is a single geodesic, it is a limit of images in the collapsed space of geodesic axes in $\bH^\ast$.
\end{lemma}
\begin{proof}
Suppose the fixed point set of $g$ is not empty and let $x=[\{x_i\}]\in \Coll_{\infty, v}$ be a point fixed by $g$. Let $\bH^\ast$ be the domain of the map $q_i=q_{\Gamma_{v_i}}$. Denote the axis of $g_i$ in $\bH^\ast$ by $\widetilde\gamma_i$. Let $\gamma_i=q_i(\widetilde \gamma_i)$. Let $\widetilde{x_i}\in \bH^\ast$ be a point in $q_i^{-1}(x_i)$. By Lemma~\ref{l:almost equal trans len for para}, there is a constant $D$ so that \walmost surely $d_{i}(x_i, g_ix_i)\geq 2d_{i}(x_i, \gamma_i) - D$. Since $x$ is fixed by $g$, we have $\wlim \frac{d_{_i}(x_i, g_ix_i)}{\| \phi_i \|_{\Coll, v}}=0$, and so $\wlim \frac{d_{_i}(x_i, \gamma_i)}{\| \phi_i \|_{\Coll, v}}=0$. Therefore $x$ is on the limit of $\{\gamma_i\}$, which is a geodesic in $\Coll_{\infty, v}$. 
\end{proof}

The goal of the rest of this section is to prove Theorem~\ref{t:Rtree summary}.\eqref{eq:abel seg}, which completes the proof of Theorem~\ref{t:Rtree summary}. 
To that end, fix a non-trivial segment $I = [a,b]$ in $\Coll_{\infty, v}$, and suppose that $\overline{g}, \overline{h} \in \overline{\V}$ stabilize $I$.  Fix lifts $g,h \in V$ of $\overline{g}, \overline{h}$ and let $(g_i)$ and $(h_i)$ be \wapproxs of $g$ and $h$, respectively.  We know that \walmost surely $g_i$ and $h_i$ both lie in the same $\Gamma_{v,i}$, and our goal is to show that \walmost surely $[g_i, h_i] \in Z(\Gamma_{v,i})$.  Recall that if $\Gamma_{v,i}$ is of hyperbolic type then $Z(\Gamma_{v,i}) = \{ 1 \}$, and that in either case $Z(\Gamma_{v,i})$ is the kernel of the action of $\Gamma_{v,i}$ on the associated hyperbolic space $\bH^\ast$.

Observe that $[g_i, h_i]$ is an \wapprox of $[g,h]$. We may suppose $[g_i, h_i]$ is \walmost surely non-trivial, or else there is nothing to prove. Because $[g,h]$ stabilizes the non-trivial segment $I$ in $\Coll_{\infty, v}$, by Corollary~\ref{c:arc stab loxo}  \walmost surely $[g_i, h_i]$ corresponds to a loxodromic isometry of $\bH^\ast$. Let $\widetilde\gamma_i$ be the invariant geodesic for $[g_i, h_i]$ in $\bH^\ast$ and let $\gamma_i=q_i(\widetilde\gamma_i)$. By Lemma~\ref{l:fixed for loxo}, $I$ is a subsegment of the limit of $\gamma_i$. As a result, there are  $a_i, b_i$ on $\gamma_i$ such that $a = [a_i]$ and $b = [b_i]$. Let $I_i$ be a geodesic segment between $a_i$ and $b_i$. 

That segment stabilizers are abelian is proved in the following three lemmas. 

\begin{lemma}\label{l:intersects nbhd of collapsing locus}
Suppose that \walmost surely $\gamma_i$ is not entirely contained in the collapsing locus but that it comes within $5\delta$ of the collapsing locus.  Then \walmost surely $[g_i, h_i] \in Z(\Gamma_{v,i})$.
\end{lemma}
\begin{proof}
Since $a_i$ and $b_i$ are on $\gamma_i$, by Lemma~\ref{l:track geodesics} and the choice of $\delta$, $I_i \subset N_{\delta}(\gamma_i)$.  Thus for $w_i \in I_i$ satisfying $d_i(a_i,w_i), d_i(b_i,w_i )\ge \frac{1}{10} d_i(a_i,b_i)$, we have $d_i \left( [g_i, h_i] \cdot w_i, w_i \right) \le 16\delta$ (See  \cite[Lemma 5.7]{paulin:rtrees}). Hence $[g_i,h_i]$ moves points in the middle part of $\gamma_i$ by less than $30\delta$, contradicting Lemma~\ref{least shift}. Hence $[g_i,h_i]$ acts trivially on $\bH^\ast$ \walmost surely, so \walmost surely $[g_i, h_i] \in Z(\gamma_{v,i})$, as required. 
\end{proof}

\begin{lemma}\label{l:disjoint from collapsing locus}
Suppose that \walmost surely $\gamma_i$ does not come within $5\delta$ of the collapsing locus.  Then \walmost surely $[g_i, h_i] \in Z(\Gamma_{v,i})$.
\end{lemma}

\begin{proof}
We consider an $\omega$--large set of indices $i$ so that $g_i$ and $h_i$ correspond to loxodromic isometries in $\bH^*$, as above, and implicitly concentrate only on such indices.

Suppose $[a_i, g_i\cdot a_i]$ penetrates a component $W$ of the collapsing locus. We claim the distance between the entry and exit points of $[a_i, g_i \cdot a_i]$ in $W$ is bounded independent of $i$. To that end, let $\widetilde a_i=q_i^{-1}(a_i)$ and $\widetilde b_i=q_i^{-1}(b_i)$. If $q_i\left([\widetilde a_i, g_i \cdot \widetilde a_i]\right)$ does not penetrate $W$, the claim follows from Lemma~\ref{l:generalized Farb's 4.8}. Suppose then that $q_i([\widetilde a_i, g_i \cdot \widetilde a_i])$ does penetrate $W$ and let $a'_i$ and $u_i'$ be the points of entry and exit of $q_i([\widetilde a_i, g_i \cdot \widetilde a_i])$ into $W$, respectively. By Lemma~\ref{l:generalized Farb's 4.9}, it suffices to show $d_{\bH^\ast}(a'_i, u_i')$ is bounded independent of $i$. Let $\pi_{\widetilde{W}}$ be the closest-point projection onto $\widetilde{W}=q_i^{-1}(W)$. Note that $d_{\bH^\ast}(a'_i, u_i')\leq d_{\bH^\ast}(\pi_{\widetilde{W}}(\widetilde{a_i}), \pi_{\widetilde{W}}(g_i\cdot\widetilde{a_i}))$. Suppose $q_i([\widetilde b_i, g_i \cdot \widetilde b_i])$ penetrates $W$. Then by Lemma~\ref{l:track geodesics} $[b_i, g_i \cdot b_i] \cap N_\delta(W) \ne \emptyset$. Since $d(a_i, b_i)$ is much bigger than both $d(b_i, g_i\cdot b_i)$ and $d(a_i, g_i \cdot a_i)$ for large $i$, by $\delta$--hyperbolicity, $[a_i, b_i] \cap N_{2\delta}(W) \ne \emptyset$. By Lemma~\ref{l:track geodesics}, every point on $[a_i, b_i]$ is at most $\delta$ away from some point on $\gamma_i$. As a result, $\gamma_i \cap N_{3\delta} (W) \ne \emptyset$, contradicting the assumptions of the lemma. Hence $q_i([\widetilde b_i, g_i \cdot \widetilde b_i])$ does not penetrate $W$. It follows that the projections of $[\widetilde a_i, \widetilde b_i]$, $[\widetilde b_i, g_i \cdot \widetilde b_i]$ and $g_i \cdot [\widetilde a_i, \widetilde b_i]$ to $\widetilde{W}$ have bounded length as none of them intersect $\widetilde{W}$. The above bounds together show that $d_{\bH^\ast}(a'_i, u_i')$ is \walmost surely bounded independent of $i$. 

The number of components penetrated by $[a_i, g_i \cdot a_i]$ is at most $d(a_i, g_i \cdot a_i)$.  Hence $d_{\bH^\ast}(\widetilde a_i, g_i \cdot \widetilde a_i)$ is bounded by $d(a_i, g_i \cdot a_i)D$, where $D$ is independent of $i$.  As a result, the geodesic rectangle $[\widetilde a_i, \widetilde b_i, g_i \cdot \widetilde a_i, g_i \cdot \widetilde b_i]$ is arbitrarily thin in the middle of $[\widetilde a_i,\widetilde b_i]$. Similarly, the same is true for the geodesic rectangle $[\widetilde a_i,\widetilde b_i, h_i \cdot \widetilde a_i, h_i \cdot \widetilde b_i]$. Hence \walmost surely, $[g_i, h_i]$ moves a point in the middle of $[\widetilde a_i, \widetilde b_i]$ by arbitrarily small amount. If $[g_i,h_i] \notin Z(\Gamma_{v_i})$, then by the definition of collapsed spaces $\gamma_i$ is in the collapsing locus \walmost surely. This contradicts the assumption of the lemma. Hence $[g_i,h_i] \in Z(\Gamma_{v_i})$, as required.
\end{proof}

\begin{lemma}\label{l:contained in collapsing locus}
If \walmost surely $\gamma_i$ is contained in the $5\delta$--neighborhood of the collapsing locus then \walmost surely $[g_i, h_i]\in Z(\Gamma_{v,i})$.
\end{lemma}
\begin{proof}
We first show that $\gamma_i$ is entirely contained in the collapsing locus of $\Coll_i$. By the choice of $K$, if $\gamma_i$ is entirely contained in the $5\delta$--neighborhood of the collapsing locus, then it is entirely contained in the $5\delta$--neighborhood of some connected component $C$ of the collapsing locus. So $q_i^{-1}(C)$ is a tube around some geodesic $\beta$ in $\bH^*$.  The $q_i$--pre-image of the $5\delta$ neighborhood of $C$ is the $D'$ neighborhood of $\beta$ for some constant $D'$. The only bi-infinite geodesic contained in this neighborhood is $\beta$, so $\widetilde{\gamma_i}=\beta$. In particular, $\widetilde{\gamma_i}$ is contained in $q_i^{-1}(C)$, so $\gamma_i$ is contained in $C$.
 
 As above, choose subsegments of $\gamma_i$ denoted $[a_i, b_i]$. Let $C_i$ be the component of the collapsing locus containing $\gamma_i$. Suppose $g_i \cdot C_i\neq C_i$. Then $d_{\Coll_{i}}(C_i,g_i \cdot C_i) \ge K$.  For \walmost all $i$, $[a_i, b_i]$ is much longer than both $[a_i, g_i \cdot a_i]$ and $[b_i, g_i \cdot b_i]$.  By $\delta$--hyperbolicity of $\Coll_i$, the middle part of $[a_i, b_i]$ is in a $2\delta$--neighborhood of $g[a_i, b_i]$.  But $K \ge 2\delta$, so $C_i$ intersects the $K$--neighborhood of $g_i \cdot C_i$, a contradiction. Thus, $g_i \cdot C_i=C_i$, and similarly $h_i \cdot C_i=C_i$. Hence $g_i$ and $h_i$ act as translations along $\widetilde\gamma_i$, and \walmost surely $[g_i, h_i] = 1$, proving the lemma.
\end{proof}

The previous three lemmas finish the proof of Theorem~\ref{t:Rtree summary}.\eqref{eq:abel seg}, and hence of Theorem~\ref{t:Rtree summary}.

\section{JSJ-decompositions and modular automorphisms} \label{s:JSJ}
We continue to make Standing Assumption~\ref{ass:fixed things} and use the notation from there. 

\subsection{Summary}
In the last section we showed in the \Cdiv case that if $v$ is the vertex associated to $\V$ and $\wlim \| \phi_i \|_{\Coll, v}=\infty$ then $\V$ admits a nontrivial action on an $\R$-tree. The results of Section \ref{s:R-trees} together with \cite[Theorem 5.1]{Guirardel:Rtrees} imply that this action induces a graph of groups decomposition of $\V$. Associated to any such decomposition is a group of
\emph{modular automorphisms} of $\V$ (see Definition~\ref{defn:modaut}).

To apply the shortening argument in the next section (see Theorem \ref{t:shortening argument}), we need to use modular automorphisms of $\V$ without a priori knowing the particular splitting of $\V$.  For this we need a \emph{JSJ--decomposition} of $\V$, which is a graph of groups decomposition which ``sees" all possible modular automorphisms of $\V$ (See Theorem \ref{thm:decomps}). 

Recall $\SC(\V)$ is the stable center of $\V$ and $\overline{\V}=\V/\SC(\V)$. Let $\pi_v\colon \V\to\overline{\V}$ denote the quotient map. If $\V$ is a hyperbolic-type vertex group, $\SC(\V)=\{1\}$ and $\overline{\V}=\V$ in this case. We first construct a JSJ--decomposition for $\overline{\V}$ which naturally induces a decomposition of $\V$, see Definition \ref{def:J^v splitting}.

The JSJ--decomposition of $\overline{\V}$ we use was essentially constructed by Guirardel-Levitt in \cite{GuiLev2}. If $\overline{\V}$ is finitely generated and $K$--CSA, the existence of the decomposition we need  follows directly from \cite[Theorem 9.14]{GuiLev2}. However, in our case $\overline{\V}$ is only finitely generated relative to stably parabolic subgroups and $\overline{\V}$ is not necessarily $K$--CSA since it can have arbitrarily large finite subgroups. Adapting the proof of \cite[Theorem 9.14]{GuiLev2} to the relatively finitely generated case is straightforward. Also, in the next section we show that $\overline{\V}$ is \emph{weakly $K$--CSA} (see Definition~\ref{defn:wcsa} below), which means that it shares enough of the properties of $K$--CSA groups that the proof of \cite[Theorem 9.14]{GuiLev2} holds with only minor changes. We explain these changes in Appendix~\ref{app:decompositions}, see Theorem~\ref{thm:JSJwithtorsion}. Finally, in Section \ref{s:modaut} we define modular automorphisms, prove Theorem \ref{thm:decomps}, and show how modular automorphisms of $\overline{\V}$ are induced by modular automorphisms of $\V$ (see Lemma \ref{l: lifting the shortening auto}).

\subsection{Weakly $K$--CSA groups}

\begin{definition}
A group is \emph{$K$--virtually abelian} if it has an abelian subgroup of index at most $K$. 
\end{definition}

\begin{restatable}{definition}{defwcsa}\label{defn:wcsa}
Fix $K \ge 1$.  A group is \emph{weakly $K$--CSA} if (i) any element $g$ of order greater than $K$ is contained in a unique maximal virtually abelian subgroup $M(g)$, so that $M(g)$ is $K$--virtually abelian and equal to its normalizer, and (ii) every two infinite, virtually abelian subgroups $A$ and $B$ with $\langle A, B\rangle$ not virtually abelian satisfy $|A\cap B|\leq K$.
\end{restatable}

Note that if a group $H$ is weakly $K$--CSA and all finite subgroups of $H$  have order $\leq K$, then $H$ is $K$--CSA.

\begin{lemma}\label{lem:CSA1}
Let $A$ be a virtually abelian subgroup of $\overline{\V}$. 

\begin{enumerate}
\item If $A$ is not abelian, then $A$ has an abelian subgroup $A^+$ of index two and an element $\tau\in A\setminus A^+$ so that for all $g\in A^+$, $\tau^{-1}g\tau=g^{-1}$. 
\item If $B$ is another virtually abelian subgroup of $\overline{\V}$, then either $|A\cap B|\leq 2$ or $\langle A, B\rangle$ is $2$--virtually abelian. 
\end{enumerate}
\end{lemma}

\begin{proof}
 We first assume that $A$ and $B$ are finitely generated. Let $A_i$ and $B_i$ be subgroups of $\Gamma_i$ generated by $\omega$--approximations of fixed finite generating sets of $A$ and $B$, respectively. There exist vertex groups $V_i$ in the geometric decomposition of $\Gamma_i$ so that \walmost surely $A_i, B_i \le V_i$. Let $\overline{A}_i$ and $\overline{B}_i$ be the images of $A_i$ and $B_i$ in $V_i/Z(V_i)$, respectively.  Then $\overline{A}_i$ and $\overline{B_i}$ are either both subgroups of fundamental groups of orientable finite volume hyperbolic $3$--manifolds or both subgroups of fundamental groups of orientable hyperbolic $2$--orbifolds. Also, $A$ and $B$ are limit groups over the families $\{\overline{A}_i\}$ and $\{\overline{B}_i\}$, respectively.

Assume $|A|\geq 3$ and $|B|\geq 3$, otherwise the result is trivial. Since $A$ and $B$ are finitely generated and virtually abelian, $\overline{A}_i$ and $\overline{B}_i$ are $\omega$--almost surely virtually abelian. Hence each $\overline{A}_i$ is contained in a unique maximal virtually abelian subgroup $M_i$ which is either $\Z$, $\Z^2$, $D_\infty$, or finite cyclic. Let $M_i^+$ be an abelian subgroup of $M_i$ of index at most 2. Then we can define $A^+$ to be the set of $g\in A$ such that if $(g_i)$ is an $\omega$--approximation of $g$, then the image of $g_i$ in $V_i/Z(V_i)$ belongs to $M_i^+$ $\omega$--almost surely. Suppose $g, h\in A$ with $\omega$--approximations $(g_i)$ and $(h_i)$ respectively. Let $\overline{g_i}$, $\overline{h_i}$ be the images of $g_i$ and $h_i$ in $V_i/Z(V_i)$, If $g, h\in A^+$, then $\overline{g}_i, \overline{h}_i\in M_i^+$ $\omega$--almost surely. Then $\overline{g}_i$ and $\overline{h}_i$ \walmost surely commute, and hence $g$ and $h$ commute. Similarly, if $g, h\notin A^+$, then $\overline{g}_i^{-1}\overline{h}_i\in M_i^+$ $\omega$--almost surely and hence $g^{-1}h\in A^+$. Thus, $A^+$ is an abelian subgroup of $A$ of index at most 2. Finally, if $g\in A^+$ and $h\notin A^+$, then we must have $M_i\cong D_\infty$ $\omega$--almost surely. It follows that $\overline{h}_i^{-1}\overline{g}_i\overline{h}_i=\overline{g}_i^{-1}$ $\omega$--almost surely, hence $h^{-1}gh=g^{-1}$.

Similarly, each $\overline{B}_i$ is contained in a unique maximal virtually abelian subgroup $N_i$. Note that in the fundamental group of an orientable finite volume hyperbolic $3$--manifolds, any two maximal virtually abelian subgroups are equal or disjoint. In the fundamental group of an orientable hyperbolic $2$--orbifold, two maximal virtually abelian subgroups are either equal, disjoint, or their intersection contains a single non-trivial element of order 2. In either case, we have $M_i=N_i$  or $|M_i\cap N_i|\leq 2$ whenever $M_i$ and $N_i$ both exist. If the first case happens $\omega$--almost surely, then the same argument as above gives that $\langle A, B\rangle$ is $2$--virtually abelian. If the second happens $\omega$--almost surely, $|A\cap B|\leq 2$.

Now suppose $A$ and $B$ are not finitely generated. We can apply the above construction to any finitely generated subgroup of $A$ and to produce the sequence $M_i$. Note that the sequence $M_i$ is independent of the finitely generated subgroup chosen, since any two finitely generated subgroups are both contained in a larger finitely generated subgroup. In particular, $A^+$ can defined as before and the same arguments show that (1) holds in this case.

Now if $|A\cap B|>2$, then the above argument shows that all finitely generated subgroups of $\langle A, B\rangle$ are $2$--virtually abelian. A diagonal argument shows that this implies that $\langle A, B\rangle$ are $2$--virtually abelian, see \cite[Lemma 9.6]{GuiLev2}. 
\end{proof}

\begin{lemma}\label{lem:CSA2}
Let $g\in \overline{\V}$ be an element of order at least 3. Let
\[
M(g)= \left\langle \{h\in \overline{\V} \;|\; \langle g, h\rangle \text{ is virtually abelian}\}\right\rangle.
\]

Then $M(g)$ is $2$--virtually abelian and $M(g)$ is the unique, maximal virtually abelian subgroup of $ \overline{\V}$ containing $g$. 
\end{lemma}

\begin{proof}
By induction and Lemma~\ref{lem:CSA1}, all finitely generated subgroups of $M(g)$ are $2$--virtually abelian so $M(g)$ is $2$--virtually abelian by \cite[Lemma 9.6]{GuiLev2}. The maximality and uniqueness of $M(g)$ follow.
\end{proof}

\begin{lemma}\label{lem:CSA3}
Let $g\in \overline{\V}$ have order at least 3. Then $M(g)$ is equal to its normalizer.
\end{lemma}

\begin{proof}

Suppose $x\in  \overline{\V}$ and $x^{-1}M(g)x=M(g)$. Let $V_i$ be as in the proof of Lemma~\ref{lem:CSA1}. Let $g_i$ and $x_i$ be $\omega$--approximations of $g$ and $x$ respectively, and let $\overline{g}_i$ and $\overline{x}_i$ be the images of $g_i$ and $x_i$ in $V_i/Z(V_i)$ respectively.

Then $\omega$--almost surely, $\langle \overline{g}_i, \overline{x}_i^{-1}\overline{g}_i\overline{x}_i\rangle$ is virtually abelian, so it is contained in a maximal virtually abelian subgroup $M_i$. If $M_i$ is finite, then it must be a finite cyclic group corresponding to a cone point on an oribfold, where this cone point has order at least 3 \walmost surely. In this case $M_i$ is malnormal, so $\overline{g}_i\in M_i$ and $\overline{x}_i^{-1}\overline{g}_i\overline{x}_i\in M_i$ \walmost surely imply $\overline{x}_i\in M_i$ \walmost surely. 

Now suppose $M_i$ is infinite. Since $M_i$ is a maximal virtually abelian subgroup of a hyperbolic 3--manifold group or a hyperbolic $2$--orbifold group, for all $h\in V_i/Z(V_i)$ either $h\in M_i$ or $|h^{-1}M_ih\cap M_i|<\infty$. Since $M_i$ is either $\Z$, $\Z^2$, or $D_\infty$, $|h^{-1}M_ih\cap M_i|<\infty$ implies that $|h^{-1}M_ih\cap M_i|\leq 2$. Now, \walmost surely $\overline{x}_i^{-1}\overline{g}_i\overline{x}_i\in M_i\cap (\overline{x}_i^{-1}M_i\overline{x}_i)$ and $\overline{x}_i^{-1}\overline{g}_i\overline{x}_i$ has order at least 3, hence $\overline{x}_i\in M_i$ \walmost surely.

Since in either case we have that \walmost surely $\overline{x}_i\in M_i$, it follows that $x\in M(g)$.
\end{proof}

Lemmas~\ref{lem:CSA1},~\ref{lem:CSA2}, and~\ref{lem:CSA3} together imply the following.
\begin{proposition}\label{prop:CSA}
$\overline{\V}$ is weakly $2$--CSA.
\end{proposition}

Using Proposition~\ref{prop:CSA}, Theorem~\ref{thm:JSJwithtorsion} applies to $\overline{V}$.  Let $\mc{A}$ be the family of all virtually abelian subgroups of $\overline{\V}$, and $\mc{H}_{\overline{\V}}$ the family of all stably parabolic subgroups of $\overline{\V}$ (recall Definition~\ref{def:stably para}).

\begin{definition} \label{def:J^v splitting}
Let $\overline{\mathbb J}(v)$ be the tree of cylinders of the $(\mc{A}, \mc{H}_{\overline{\V}})$--JSJ decomposition of $\overline{V}$ given by Theorem \ref{thm:JSJwithtorsion}. 

Let $\mathbb J(v)$ be the splitting of $\V$ induced by  the splitting $\overline{\mathbb J}(v)$ of $\overline{\V}$. That is, $\mathbb J(v)$ and $\overline{\mathbb J}(v)$ have the same underlying graph and each vertex and edge group of $\mathbb J(v)$ is the pre-image of the corresponding vertex or edge group of $\overline{\mathbb J}(v)$.
\end{definition}

By Theorem \ref{thm:JSJwithtorsion}, $\overline{\mathbb J}(v)$ is compatible (in the sense of \cite[Definition 5.19]{GH}) with every $(\mc{A}, \mc{H}_{\overline{\V}})$--splitting of $\overline{V}$. That is, for any $(\mc{A}, \mc{H}_{\overline{\V}})$--splitting $\mathbb{A}$ there exists an $(\mc{A}, \mc{H}_{\overline{\V}})$--splitting $\mathbb B$ together with collapse maps $\mathbb B\to \mathbb A$ and $\mathbb B\to\overline{\mathbb J}(v)$. 

Next we describe the structure of certain important vertex groups of $\overline{\mathbb J}(v)$ and $\mathbb J(v)$, referred to as the flexible vertex groups in Theorem \ref{thm:JSJwithtorsion}.

\begin{definition}\label{def:dihtyp}
We say a group $D$ is a \emph{dihedral type} group if 
there exist (i)  has a subgroup $D^+$ of index 
two in $D$;  and (ii) $\tau\in D\setminus D^+$ so that for all $g\in D^+$, $\tau^{-1}g\tau=g^{-1}$. 
\end{definition}

Note that by Lemma \ref{lem:CSA1}, every virtually abelian subgroup of $\overline{V}$ is either abelian or dihedral type. If $\overline{A}$ is a virtually abelian subgroup of $\overline{V}$ and $A$ is the pre-image of $\overline{A}$ in $V$, then $A$ is either abelian by Lemma \ref{l:preimage abelian} or $A$ is a central extension of a dihedral type group.

For a general group $H$, we will again use $\mc{A}$ to denote the family of virtually abelian subgroups of $H$. We will also use $\mc{H}$ to denote an arbitrary collection of subgroups of $H$.

\begin{definition}\label{defn:QH}
Let $\mathbb A$ be an $(\mc{A}, \mc{H})$ splitting of a group $H$ and let $U$ be a vertex group of $\mathbb A$. $U$ is called a \emph{QH--vertex group with fiber $F$} if there is a short exact sequence of the form
\[
1\to F\to U\to O\to 1
\] 
Where  $O$ is the fundamental group of an orientable hyperbolic $2$--orbifold of finite-type, i.e. $O\in \orb$. Moreover, we require that if $H\leq U$ is an edge group of $\mathbb A$ or $H$ is conjugate into an element of $\mc{H}$, then the image of $H$ in $O$ is either finite or is conjugate into a boundary subgroup of $O$.

\end{definition}

 By Theorem \ref{thm:JSJwithtorsion}, the QH--vertex groups of $\overline{\mathbb J}(v)$ have fiber of size at most 2. In this case, the fiber is a central subgroup of the QH--vertex group. The QH--vertex groups of the corresponding splitting $\mathbb J(v)$  of $\V$ are then central extensions of the QH--vertex groups of $\overline{\V}$, hence these are also QH--vertex groups with central fiber.

\subsection{Modular automorphisms}\label{s:modaut}

We define the group of modular automorphisms relative to a particular graph of groups decomposition and a fixed family of subgroups. These definitions are standard except that some care has to be taken to ensure that the modular automorphisms of $\overline{\V}$ are induced by automorphisms of $\V$, especially in the case of the dihedral type vertex groups.

Suppose $H = A\ast_C B$ and $c \in Z_H(C)$.  The \emph{Dehn twist by $c$} is the automorphism of $H$ defined by fixing each $a \in A$ and mapping each $b \in B$ to $cbc^{-1}$.  If $H = A\ast_C$ and $c \in Z_H(C)$ then the \emph{Dehn twist by $c$} is the automorphism defined by fixing each $a \in A$ and mapping the stable letter $t$ to $tc$.  If $\mathbb A$ is a graph of groups decomposition and $e$ is an edge of $\mathbb A$, a \emph{Dehn twist over $e$} is a Dehn twist in the one edge splitting corresponding to collapsing all edges of $\mathbb A$ other than $e$.

Now suppose $A$ is an abelian group which is finitely generated relative to a family of subgroups $\mc{E}$. Let $P_A$ be the minimal direct factor of $A$ which contains the subgroups in $\mc{E}$ and the torsion of $A$. The subgroup $P_A$ can be characterized as the intersection of the kernels of all homomorphisms $A\to \Z$ which are identically 0 on the elements of $\mc{E}$.  Write $A = A_0 \oplus P_A$, and note that since $A$ is finitely generated relative to $\mc{E}$, $A_0$ is a finitely generated free abelian group. We consider the group of automorphisms of $A$ fixing $P_A$. In particular, this group contains automorphisms sending $g$ to $gh$ for some basis elements $g$ and $h$ of  $A_0$ and fixing all other elements of the basis of $A_0$, and all elements of $P_A$. We call such automorphisms of $A$ \emph{Nielsen transformations relative to $\mc{E}$} and note that they generate a group isomorphic to $\mathrm{SL}_n(\Z)$ where $n$ is the rank of $A_0$.

Now suppose $D$ is a dihedral type group which is finitely generated relative to a family of subgroups $\mc{E}$ which are closed under conjugation.  Let $D^+$ and $\tau$ be as in Definition~\ref{def:dihtyp}. We again define $P_{D^+}$ as the minimal direct factor of $D^+$ which contains the intersection of $D^+$ with the elements of $\mc{E}$ and with the torsion of $D^+$.  Write $D^+ = D_0 \oplus P_{D^+}$, and suppose $g$ and $h$ are distinct elements of a basis for $D_0$. A \emph{squared Nielsen transformation relative to $\mc{E}$} is an automorphism of $D$ fixing $\tau$ and extending the automorphism of $D^+$ sending $g$ to $gh^2$ and fixing $P_{D^+}$ and all other basis elements of $D_0$. As we show in Lemma \ref{l: lifting the shortening auto}, a squared Nielsen transformation of a dihedral type vertex group of a splitting of $\overline{V}$ is induced by an automorphism of the pre-image of that vertex group in $V$, which may not be true for the Nielsen transformation which sends $g$ to $gh$. We also show in Lemma~\ref{l:snt is enough} that the shortening argument still works using these squared Nielsen transformations.

Suppose $\mathbb A$ is a graph of groups decomposition of a group $H$ and $v_0$ is a vertex of $\mathbb A$. Any element of $g$ can be represented (non-uniquely) by $[a_0, e_1, a_1,..., e_n, a_n]$ where $e_1,...,e_n$ is an edge path in $\mathbb A$ from $v_0$ to $v_0$, $a_0\in \mathbb A_{v_0}$, and each $a_i\in \mathbb A_{v_i}$ where $v_i$ is the terminal vertex of $e_i$ for $1\leq i\leq n$. Suppose $v$ is a vertex of $\mathbb A$ and $\sigma \in \Aut(\mathbb A_v)$ acts by conjugation on each adjacent edge group.  Then $\sigma$ can be extended to an automorphism of $G$ as in \cite[Definition 4.13]{ReiWei} as follows:  for each adjacent edge $e$ there is $\gamma_e\in \mathbb A_v$ so that $\sigma(h)=\gamma_e h\gamma_e^{-1}$ for all $h$ in the image of $\mathbb A_e$ in $\mathbb A_v$. Then $\sigma$ extends to $\overline{\sigma} \in \Aut(H)$ by defining
\[
\overline{\sigma}([a_0, e_1, a_1,..., e_n, a_n])= [\overline{a}_0, e_1, \overline{a}_1,..., e_n, \overline{a}_n]
\]

where 
\[
\overline{a_i}=\begin{cases}
a_i\;&\text{if}\;a_i\notin\mathbb A_v\\
\gamma_{e_i}^{-1}\sigma(a_i)\gamma_{e_{i+1}}\;&\text{if}\;a_i\in\mathbb A_v.
\end{cases}
\]

We call $\overline{\sigma}$ a {\em natural extension} of $\sigma$. This natural extension is not unique: it depends on the choice of $\gamma_e$.

\begin{definition}\label{defn:modaut}
Let $\mathbb A$ be a reduced $(\mc{A}, \mc{H})$ splitting of a group $H$. The \emph{modular automorphism group of $H$ relative to ($\mathbb A, \mc{H}$}), denoted $\Mod^{\mc{H}}_{\mathbb A}(H)$, is the subgroup of $\Aut(H)$ generated by:
\begin{enumerate}
\item Inner automorphisms.

\item Dehn twists over edges of $\mathbb A$ with abelian edge groups.

\item Natural extensions of automorphisms of QH-vertex groups of $\mathbb A$ which fix the fiber pointwise, which act by conjugation on elements of $\mc{H}$, and which project to automorphisms of the associated orbifold group induced by homeomorphisms of the underlying orbifold fixing the boundary and cone points. 

\item Natural extensions of Nielsen transformations of maximal abelian vertex groups of $\mathbb A$ relative to $\mc{H}$ and to the adjacent edge groups.

\item Natural extensions of automorphisms of vertex groups $U$ where $U$ is a central extension $1\to Z\to U\to D\to 1$, $Z$ is contained in all edge groups adjacent to $U$, $D$ is a dihedral-type group and the automorphism fixes $Z$ and projects to squared Nielsen transformations of $D$ relative to the image in $D$ of the edge group of $\mathbb A$ adjacent to $U$ and to the elements of $\mc{H}$ contained in $U$. 
\end{enumerate}
\end{definition}

  When working with the $\overline{V}$, the type (5) automorphisms are squared Nielsen transformations of dihedral type vertex groups.  That is, the central extension is trivial in this case.

Note that each type of modular auomorphism listed in Definition \ref{defn:modaut} acts by a (possibly trivial) conjugation on each subgroup in $\mc{H}$. Since this holds for the generators of $\Mod^{\mc{H}}_{\mathbb A}(H)$, it also holds for all the elements of $\Mod^{\mc{H}}_{\mathbb A}(H)$. We record this fact with the following lemma.
\begin{lemma}\label{l:conjugate parabolic}
If $\sigma\in \Mod^{\mc{H}}_{\mathbb A}(H)$ and $P\in \mc{H}$, then $\sigma(P)$ is conjugate to $P$.
\end{lemma}

 The following is one of the main results of this section and is similar to \cite[Proposition 4.17]{ReiWei}, \cite[Theorem 5.4]{GL-splittings}, and  \cite[Theorem 5.23]{GH}.

\begin{theorem}\label{thm:decomps} 
For any $(\mc{A}, \mc{H}_{\overline{\V}})$--splitting $\mathbb A$ of $\overline{\V}$, $\Mod^{\mc{H}_{\overline{\V}}}_{\mathbb A}(\overline{\V})\leq\Mod^{\mc{H}_{\overline{\V}}}_{\overline{\mathbb J}(v)}(\overline{\V})$.
\end{theorem}

The proof of Theorem \ref{thm:decomps} is essentially the same as \cite[Theorem 5.23]{GH}. There are some additional complications in \cite[Theorem 5.23]{GH} which do not arise here, but the differences between the proofs are almost all notational. We provide a sketch of how to derive the proof of Theorem \ref{thm:decomps} from the proof of Theorem  \cite[Theorem 5.23]{GH}. Similar arguments also appear in the proof of \cite[Proposition 4.17 ]{ReiWei}.

\begin{proof}[Proof of Theorem \ref{thm:decomps}]
In our proof, the decomposition ${\overline{\mathbb J}(v)}$ plays the role of $\mathbb A_{JSJ}$ from \cite[Section 5.4]{GH}. In \cite[Section 5.4]{GH}, there is another splitting, $\mathbb A_{M}$, which is built as a refinement of $\mathbb A_{JSJ}$. This refinement is only needed because the splittings considered in \cite[Section 5.4]{GH} do not allow arbitrary virtually abelian edge groups, but only certain types of virtually abelian edge groups. For our purposes, It suffices to read the proofs of \cite[Section 5.4]{GH} with the additional assumption that $\mathbb A_{JSJ}=\mathbb A_{M}$.

Let $\mathbb A$ be an $(\mc{A}, \mc{H}_{\overline{\V}})$--splitting of $\overline{V}$. We consider each type of generator of $\Mod^{\mc{H}_{\overline{\V}}}_{\mathbb A}(\overline{V})$ given by Definition \ref{defn:modaut} in turn. The type (1) generators of $\Mod^{\mc{H}_{\overline{\V}}}_{\mathbb A}(\overline{V})$ clearly belong to $\Mod^{\mc{H}_{\overline{\V}}}_{\overline{\mathbb J}(v)}(\overline{\V})$. The argument that the type (2) generators of $\Mod^{\mc{H}_{\overline{\V}}}_{\mathbb A}(\overline{V})$ belong to $\Mod^{\mc{H}_{\overline{\V}}}_{\overline{\mathbb J}(v)}(\overline{\V})$ is the same as \cite[Lemma 5.25]{GH}. The argument that the type (3) generators of $\Mod^{\mc{H}_{\overline{\V}}}_{\mathbb A}(\overline{V})$ belong to $\Mod^{\mc{H}_{\overline{\V}}}_{\overline{\mathbb J}(v)}(\overline{\V})$ is the same as \cite[Lemma 5.27]{GH}. The type (4) and (5) generators of $\Mod^{\mc{H}_{\overline{\V}}}_{\mathbb A}(\overline{V})$ both correspond to virtually abelian vertex groups (note that the central extensions of dihedral type vertex groups are all extensions with trivial kernel in this case, so these vertex groups are actually dihedral type), and the argument that these generators belong to $\Mod^{\mc{H}_{\overline{\V}}}_{\overline{\mathbb J}(v)}(\overline{\V})$ is the same as \cite[Lemma 5.26]{GH}.
\end{proof}

Next we show modular automorphisms of $\overline{\V}$ are induced by modular automorphisms of $V$. By Lemma~\ref{l:preimage abelian}, the pre-image in $\V$ of any abelian subgroup of $\overline{\V}$ is abelian. Recall that $\mathbb J(v)$ is the splitting of $\V$ induced by the splitting $\overline{\mathbb J}(v)$ of $\overline{\V}$ and $\pi_v\colon \V\to \overline{\V}$ is the quotient map.

\begin{lemma}\label{l: lifting the shortening auto}
Let $\alpha\in\Mod_{\overline{\mathbb J}(v)}^{\mc{H}_{\overline{\V}}}(\overline{\V})$ be a modular automorphism of type (1), (2), (4), or (5) according to the notation of Definition \ref{defn:modaut}. Then there exists $\widetilde{\alpha}\in \Mod_{\mathbb J(v)}^{\mc{H}_{{\V}}}(\V)$ so that $\widetilde{\alpha}$ acts trivially on the stable center and for all $g \in \V$, $\pi_v(\widetilde{\alpha}(g))=\alpha(\pi_v(g))$. 
\end{lemma}
\begin{proof}
First, the inner automorphisms of $\overline{\V}$ are induced by inner automorphisms of $\V$ since $\V$ is a central extension of $\overline{\V}$.

Suppose $\alpha$ is a Dehn twist by $c$ over an edge $e$ of $\overline{\mathbb J}(v)$. Then $V$ has a one-edge splitting $\mathbb B$ with edge $f$ such that $\mathbb B_f=\pi_v^{-1}(\overline{\mathbb J}(v)_e)$, and $\mathbb B_f$ is abelian, by Lemma \ref{l:preimage abelian}. Let $\widetilde{c}\in \pi_v^{-1}(c)$. Define $\widetilde{\alpha}$ to be the Dehn twist by $\widetilde{c}$ over $f$ by the element.

If $\alpha$ is a type (4) automorphism for a maximal abelian vertex group $\overline{A}$ in $\overline{\mathbb J}(v)$. Then the corresponding vertex group $A$ of $\mathbb J(v)$ is also maximal abelian by Lemma \ref{l:preimage abelian}. Let $\overline{P}_{\overline{A}}$ be the subgroup from the definition of relative Nielsen transformations, so $\overline{A}=\overline{A}_0\oplus\overline{P}_{\overline{A}}$. Let $\{\overline{a}_1,...,\overline{a}_n\}$ be a basis for $\overline{A}_0$. Then $P_A=\pi_v^{-1}(\overline{P})$ is the minimal direct factor of $A$ which contains all adjacent edge groups, elements of $\mc{H}_\V$, and the stable center. If we choose $a_i\in\pi^{-1}_v(\overline{a_i})$ for each $1\leq i\leq n$, and let $A_0$ be the subgroup generated by $\{a_1,...,a_n\}$ then we have that $A\cong A_0\oplus P_A$. Hence if $\alpha$ is the relative Nielsen transformation which sends $\overline{a}_i$ to $\overline{a}_i\overline{a}_j$, then $\widetilde{\alpha}$ can be defined as the relative Nielsen transformation which sends $a_i$ to $a_ia_j$.

Suppose now that $\overline{D}$ is a dihedral type vertex group of $\overline{\mathbb J}(v)$, let $D$ be the pre-image of $D$ in $V$, and let $D^+$ be the pre-image of $\overline{D}^+$. By Lemma \ref{l:preimage abelian} $D^+$ is abelian. Let $t$ be a pre-image of $\tau$. For each $x\in D^+$, let $c_x=t^{-1}xtx$. Then $\pi_v(c_x)=1$, hence $c_x$ belongs to $\ker(\pi_v)\cap D$, which we denote by $Z^{\omega}(D)$. Next we show that the map  $D^+\to Z^{\omega}(D)$ defined by $x\to c_x$ is a homomorphism.

Let $x, y\in D^+$. Recall that $D^+$ is abelian so $x$ and $y$ commute, and $c_x$ and $c_y$ belong to the center of $D$. It follows that 
\[
c_{xy}=t^{-1}xytxy=(t^{-1}xt)(t^{-1}yt)xy=x^{-1}c_xy^{-1}c_yxy=c_xc_y
\]

 Define $P_{\overline{D}^+}$ as the minimal direct factor of $\overline{D}^+$ which contains the intersection of $\overline{D}^+$ with all edge groups adjacent to $\overline{D}$, all stably parabolic subgroups of $\overline{D}^+$ and all the torsion of $\overline{D}^+$. Then as in the abelian case, we can write $\overline{D}^+=\overline{D}^+_0\oplus\overline{P_{D^+}}$ and $D^+\cong D^+_0\oplus P$ with $P=\pi_v^{-1}(P_{\overline{D}^+})$, where $D^+_0$ is a finite rank free abelian group which maps isomorphically to $\overline{D}_0^+$. Suppose $\overline{g}, \overline{h}$ are basis elements of $\overline{D}^+$, $g$ and $h$ are pre-images of $\overline{g}, \overline{h}$ in $D^+$, and $\alpha$ is a squared Nielsen transformation of $\overline{D}$ of the form $\overline{g}\to \overline{g}\overline{h}^2$. Let $\widetilde{\alpha}$ be the map on $D$ which is the identity on $t$, $P$, and on all basis elements of $D^+_0$ except $g$  where $\widetilde{\alpha}(g) = gh^2c_{h^{-1}}$. The map $\widetilde{\alpha}$ is invertible and induces $\alpha$ on $\overline{D}$, so it only remains to verify that $\widetilde{\alpha}$ is a homomorphism. It clearly induces a homomorphism on the abelian group $D^+$, so it remains to show that it preserves the relations $t^{-1}xt=x^{-1}c_x$ for all $x\in D^+$. By our definition of $c_x$, $t^{-1}\widetilde{\alpha}(x)t=\widetilde{\alpha}(x)^{-1}c_{\widetilde{\alpha}(x)}$. Since $\widetilde{\alpha}$ fixes $t$, this means we need to show that $c_x=c_{\widetilde{\alpha}(x)}$  for all $x\in D^+$. Observe that

 \begin{align*}
c_{\widetilde{\alpha}(g)}&=t^{-1}gh^2c_{h^{-1}}tgh^2c_{h^{-1}}\\
&=t^{-1}gh^2tgh^2c_{h^{-1}}^2\\
&=t^{-1}gt(t^{-1}h^2th^2)gc_{h^{-1}}^2\\
&=t^{-1}gtgc_{h^2}c_{h^{-1}}^2\\
&=c_g.
\end{align*}

The map $x\to c_x$ is a homomorphism and $c_{\widetilde{\alpha}(y)}=c_y$ for all $y$ which are either in $P$, or in a basis for $D^+_0$. Therefore  for all $x\in D^+$ $c_{\widetilde{\alpha}(x)}=c_x$, and so $\widetilde{\alpha}$ is an automorphism of $D$ inducing $\alpha$ on $\overline{D}$.  

Finally, we observe that $\widetilde{\alpha}$ acts by conjugation on the adjacent edge groups of $D$. Suppose $E$ is such an edge group. Then $E\cap D^+$ is contained in $P$ and hence fixed by $\widetilde{\alpha}$. If $E\subset D^+$ then we are done, so assume that $E\not\subset D^+$. The subgroup $D^+$ is index two in $D$, so there exists $d\in D^+$ such that $E=\langle td, E\cap D^+\rangle$. Then $d=d_0p$ for some $d_0\in D^+_0$ and $p\in P$. Suppose that $g^n$ is the power of $g$ which appears when $d_0$ is written  in terms of the basis of $D_0^+$ given above. Then $\widetilde{\alpha}(d_0)=d_0h^{2n}c_{h^{-1}}^n$. Since $\widetilde{\alpha}(p)=p$ and $p$ commutes with $d_0$ and $c_{h^{-1}}$, $\widetilde{\alpha}(d)=d_0h^{2n}c_{h^{-1}}^np=dh^{2n}c_{h^{-1}}^n$. Using the fact that $\widetilde{\alpha}(t)=t$, we have that

\[
\widetilde{\alpha}(td)=tdh^{2n}c_{h^{-1}}^n.
\] 

On the other hand, the definition of $c_{x}$ gives that $h^{-n}th^{n}=th^{2n}c_{h^{-1}}^n$. Since $h^n$ commutes with $d$, we have

\[
h^{-n}tdh^{n}=tdh^{2n}c_{h^{-1}}^n.
\]

Since $h$ commutes with $E\cap D^+$, we have that for all $e\in E$, $\widetilde{\alpha}(e)=h^{-n}eh^n$.
\end{proof}

\begin{lemma}\label{l: lifting the shortening auto2}
For each QH-vertex group $U$ of $\mathbb J(v)$ with underlying orbifold $\mc{O}$ and each homeomorphism $\phi$ of $\mc{O}$ fixing the boundary and cone points, there exists $\widetilde{\alpha}\in \Mod_{\mathbb J(v)}^{\mc{H}_\V}(\V)$ so that $\widetilde{\alpha}(U)=U$,$\widetilde{\alpha}$ acts trivially on the stable center and on the fiber of $U$, and $\widetilde{\alpha}$ induces $\phi$ on $\mc{O}$.
\end{lemma}
\begin{proof}
 Let $U$ be a QH-vertex group of $\mathbb J(v)$ and let $\overline{U}$ be the image of $U$ in $\overline{V}$. Then $\overline{U}$ is a QH-vertex group of $\overline{\mathbb J}(v)$. Let $F$ the fiber of $\overline{U}$ so $\overline{U}/F\cong O$, where $O$ is an orbifold fundamental group as in Definition \ref{defn:QH}. By Theorem \ref{thm:JSJwithtorsion}, we can also assume that $|F|\leq 2$. Since $F$ is normal, this means that $F$ must also be central in $\overline{U}$. Now let $\gamma\in O$ be an element which is represented by a simple closed curved on the underlying orbifold, and let $\overline{A}_{\gamma}$ be the pre-image of $\langle \gamma\rangle$ in $\overline{U}$. Then $\overline{A}_{\gamma}$ is a central extension of a cyclic group, hence $\overline{A}_{\gamma}$ is abelian.
 
Now let $A_\gamma=\pi_v^{-1}(\overline{A}_{\gamma})$. Note that $U$ has fiber $\pi_v^{-1}(F)$, and $A_\gamma$ is an abelian subgroup of $U$ by Lemma \ref{l:preimage abelian}. Since $\gamma$ is represented by a simple closed curve, the group $O$ splits over $\langle \gamma\rangle$ and this induces a splitting of $U$ over $A_\gamma$. The dehn twist by $\gamma$ is an automorphism of $O$ which is induced by a Dehn twist on $U$ over any element in $A_\gamma$ which maps to $\gamma$. This Dehn twist on $U$ is a type (3) element of $\Mod_{\mathbb J(v)}^{\mc{H}}(\V)$. Since the pure mapping class group of the orbifold is generated by Dehn twists, we are done.

\end{proof}

\section{Resolutions and factoring}\label{s:res-short}

In this section we complete the proof of Theorem~\ref{t:technical outline}.  Given the setup in the previous two sections, the purpose of this section is to undertake the version of Sela's shortening argument which works in the setting of $\R$--trees arising from the construction in Section~\ref{s:R-trees}.

Throughout this section we continue to make Standing Assumption~\ref{ass:fixed things} and use the notation from there. In particular, we consider a sequence $(\phi_i \co G \to \Gamma_i = \pi_1(M_i))$ from $\Hom(G,\Mgen)$, which is not \Tdiv. In addition, we also assume throughout this section that the sequences $(\phi_i)$ is \Cdiv.

\subsection{Shortening quotients}

We first use the decompositions $\mathbb J(v)$ from Definition~\ref{def:J^v splitting} to refine the decomposition $\Gdf$ of $L$.

\begin{definition} \label{def:J splitting}
 Since the edge groups of $\Gdf$ are elliptic in each $\mathbb J(v)$, we can refine $\Gdf$ by replacing each vertex $v$ by the splitting $\mathbb J(v)$ (see \cite[Lemma 4.12]{GuiLev2}).  Denote the resulting splitting of $L$ by $\mathbb J$.
 \end{definition}

As in \cite{ReiWei} (see also \cite{GH}), we create a sequence of graphs of groups approximating $\mathbb J$ so that the $\phi_i$ factor through the terms in the approximating graphs of groups. This is essentially the same as \cite[Lemma 7.1]{ReiWei}, and the proof from there works in our situation without change.  The idea of approximating splittings of non-finitely generated limit groups with associated splittings of finitely presented approximations comes from \cite[Section 3]{sela:dio1}.
\begin{lemma} \label{lem:ugly hack}
Let $\mathbb J$ be the splitting of $L$ from Definition~\ref{def:J splitting}. There exists a sequence of finitely presented groups $G=W_0, W_1, \ldots$ and epimorphisms $f_i\colon W_i\to W_{i+1}$ and $h_i\colon W_i\to L$ for $i\geq 0$ so that:
\begin{enumerate}
\item $\phi_\infty=h_0$;
\item for all $i \ge 1$ we have $h_i=h_{i+1}\circ f_i$;
\item\label{eq:direct limit}  $L$ is the direct limit of the sequence $G\twoheadrightarrow W_1\twoheadrightarrow...$.  Equivalently, 
$$\wker(\phi_i)=\bigcup_{k=1}^\infty\ker(f_{k-1} \circ \ldots \circ f_0)$$
\item Each $W_i$ has a graph of groups decomposition $\mathbb A_i$. For each $i$ there exists morphisms $f^i: \mathbb A_i\rightarrow \mathbb A_{i+1}$ and $h^i: \mathbb A_i\rightarrow \mathbb J$ such that the underlying graph morphisms are isomorphims,   $f^i$ induces $f_i$ and $h^i$ induces $h_i$.
\item Let $V_i$ be a vertex group  of $\mathbb A_i$ and $V$ be the corresponding vertex group of $\mathbb J$ under $h^i$. Then we have 
\begin{enumerate}
\item $V=\bigcup\limits_{i=1}^\infty h_i(V_i)$
\item  $V_i$ has a central subgroup $Z_i$ which is contained in all edge groups adjacent to $V_i$ such that $h_i$ maps $Z_i$ injectively into $\SC(\V)$ and $\SC(\V)=\bigcup\limits_{i=1}^\infty h_i(Z_i)$
\item If $\overline{V}$ is QH--vertex group, then $h_i$ induces an isomorphism from $V_i/Z_i$ to $\overline{V}$.
\item If $\overline{V}$ is a virtually abelian vertex group, then $h_i$ induces an injective map from $V_i/Z_i$ to $\overline{V}$.
\end{enumerate}

\item If $E$ is an edge group of $\mathbb J$ and $E_i$ is the corresponding edge group of $\mathbb A_i$, then $h_i$ maps $E_i$ injectively into $E$ and $E=\bigcup_{i=1}^\infty h_i(E_i)$.
\item The edge groups and vertex groups of $\mathbb{A}_i$ are finitely generated. 
\end{enumerate}

\end{lemma}

Fix a vertex $v$ of $\Gdf$. There is a collapse map $\mathbb J\to\Gdf$, such that the pre-image of $v$ is $\mathbb J(v)$. For each $\mathbb A_i$, the underlying graph is isomorphic to $\mathbb J$.  Let $\mathbb A_{v, i}$ denote the sub-(graph of groups) corresponding to $\mathbb J(v)$. Let $\mathbb W_i$ be the splitting of $W_i$ obtained from $\mathbb A_i$ by collapsing each $\mathbb A_{v, i}$.  Let $W_{v, i}$ be the fundamental group of $\mathbb A_{v, i}$; note that $W_{v, i}$ is identified with a subgroup of $W_i$ by Section \ref{sec:not_cov}. It follows from the construction that $f_i|_{W_{v, i}}$ maps into (though possibly not \emph{onto}) $W_{v, i+1}$ and that $\V$ is the direct limit of the $W_{v, i}$.

 We use modular automorphisms associated to these graphs of groups to ``shorten" $\phi_i$ relative to the vertex $v$, as we now explain. 

Let $\xi_i\colon G\to W_i$ be the natural map, that is $\xi_i=f_{i-1} \circ...\circ f_0$. Since $W_j$ is finitely presented, for fixed $j$, $\ker(\xi_j)\subseteq \ker(\phi_i)$ for an $\omega$ large set of $i$. Hence after passing to a subsequence of $\phi_i$ and re-indexing, we may assume $\ker(\xi_j)\subseteq\ker(\phi_i)$ for all $j$ and all $i\geq j$.  Thus, for all $j$ and all $i\geq j$ the map $\phi_i$ factors through $\xi_j$, so there is $\lambda_i^j \co W_j \to \Gamma_i$ so
\[
\phi_i=\lambda_i^j\circ\xi_j.
\]
Associated to $v$ is a sequence $\{v_i\}$ of vertices in the geometric trees of $\{\Gamma_i\}$, and $g\in \V$ if and only if for some \wapprox $(g_i)$ to $g$ \walmost surely $g_i \in \Gamma_{v_i}$. Since each $W_{v, j}$ is finitely generated, after again passing to a subsequence of $\{\phi_i\}$ and hence $\{\lambda_i^j\}$ and re-indexing we can assume $\lambda_i^j$ is in fact a morphism of graphs of groups from $\mathbb W_j$ to the geometric splitting of $\Gamma_i$. Let $\lambda_i=\lambda_i^i$. Summarizing, we have

\begin{lemma}\label{lambda_i}
\ 
\begin{enumerate}
\item $\phi_i=\lambda_i\circ\xi_i$;
\item $\lambda_i$ is a morphism of graphs of groups from $\mathbb W_i$ to the geometric splitting of $\Gamma_i$.
\end{enumerate}
\end{lemma}

 Let $\widetilde{A}_v$ be a lift to $G$ of the good relative generating set $A_v$ of $\V$ and for each $i$ let $\widetilde{A_{v,i}} = \xi_i(\widetilde{A}_v)$.
Lemma~\ref{lambda_i} allows us 
 make the following definition. 
 
 \begin{definition}\label{d:stretch factor} 
\[
\|\lambda_i\|_{\Coll, v}=\inf_{x\in\Coll(\Gamma_{v_i})}\max_{a\in \widetilde{A_{v, i}}}d_i(\lambda_i(a) \cdot x, x), 
\] 
For \walmost every $i$, choose a point $\bp_i \in \Coll(\Gamma_{v_i})$ which satisfies
\[ \max_{a \in \widetilde{A_{v,i}}} \left\{ d_i( \lambda_i(a) \cdot \bp_i, \bp_i) \right\} \le \| \lambda_i \|_{\Coll, v} + \frac{1}{i}	.	\]
Define
\[\| \lambda_i \|_{\Coll, v, \bp}=\max_{a \in \widetilde{A_{v,i}}} \left\{ d_i( \lambda_i(a) \cdot \bp_i, \bp_i) \right\} \]
\end{definition}

Let $\mc{H}_{i}$ denote the family of subgroups $H$ of $W_{i}$ such that $H$ is elliptic in $\mathbb A_{i}$ and $h_i(H)$ is a stably parabolic subgroup of $L$. For a vertex $v$ of $\Gdf$ let $\mc{H}_{v, i}$ denote the family of subgroups of $W_{v, i}$ which belong to $\mc{H}_i$. Note that any edge group of $\mathbb W_i$ which belongs to $W_{v, i}$ maps to a stably parabolic subgroup of $L$ and hence belongs to $\mc{H}_{v, i}$. By Lemma \ref{l:conjugate parabolic} $\Mod^{\mc{H}_{v, i}}_{\mathbb A_{v, i}}(W_{v, i})$ acts on these subgroups by conjugation, and hence can be extended to automorphisms of $W_i$. Let $\Mod_v(W_i)$ be the subgroup of $\Aut(W_i)$ generated by these extensions. Denote by $\Mod(W_i)$ the subgroup of $\Aut(W_i)$ generated by all subgroups $\Mod_v(W_i)$ over all vertices $v$ of $\Gdf$. Note that  $\Mod(W_i)$ acts on subgroups in $\mc{H}_i$ by conjugation.  We define an equivalence relation on $\Hom(W_i, \Gamma_i)$ by setting $\lambda\sim\lambda'$ if there is an $\alpha \in \Mod(W_i)$ so $\lambda'=\lambda\circ\alpha$. 

\begin{definition}
A sequence of homomorphisms $\{\widehat{\lambda}_i\}$ is called \emph{almost short} if for any vertex $v$ of $\Gdf$ and any sequence $\{\lambda'_i\sim \widehat{\lambda}_i\}$, we have $\|\widehat{\lambda}_i\|_{\Coll, v}\leq \|\lambda_i'\|_{\Coll, v}+\frac{1}{i}$.
\end{definition}

\begin{lemma}\label{l:shortest}
There exists an almost short sequence of homomorphisms $\{\widehat{\lambda}_i\}$ such that $\widehat{\lambda}_i\sim \lambda_i$ for all $i$. 
\end{lemma}

\begin{proof}
For each vertex $v$ of $\Gdf$, choose $\alpha_i^v\in \Mod(W_i)$ so that $\|\lambda_i\circ\alpha_i^v\|_{\Coll, v}\leq \|\lambda_i'\|_{\Coll, v}+\frac{1}{i}$ for any $\lambda'_i\sim \lambda_i$. Since $\Mod(W_i)$ is generated by all subgroups of the form $\Mod_{v'}(W_i)$, where $v'$ is a vertex of $\Gdf$, we have $\alpha_i^v=\alpha_1\cdots\alpha_k$, where each $\alpha_j\in \Mod_{v'}(W_i)$ for some vertex $v'$ of $\Gdf$. Note that if $v'\neq v$, then $\alpha_j$ acts on $W_{v, i}$ by conjugation and hence does not change the translation length $\|\lambda_i\|_{\Coll, v}$. Therefore we can delete all such $\alpha_j$ in $\alpha_i^v=\alpha_1\cdots\alpha_k$ while keeping $\|\lambda_i\circ\alpha_i^v\|_{\Coll, v}\leq \|\lambda_i'\|_{\Coll, v}+\frac{1}{i}$ true. As a result, we can assume that $\alpha_i^v\in \Mod_v(W_i)$ without loss of generality.  Let $\alpha_i$ be a product of all these $\alpha_i^v$ and $\widehat{\lambda}_i=\lambda_i\circ\alpha_i$.   Then $\widehat{\lambda}_i$ realize the ``almost minimum'' of $\|\cdot\|_{\Coll, v}$ for all vertices $v$ simultaneously. 
 \end{proof}

\begin{definition}\label{d:shortening quotient}
Let $\{\widehat{\lambda}_i\}$ be as in Lemma \ref{l:shortest} and define $\eta_i\colon=\widehat{\lambda}_i\circ\xi_i$. We call the limit group $S=G/\wker(\eta_i)$  associated to the sequence $(\eta_i\colon G\to \Gamma_i)$ a \emph{$\Coll$--shortening quotient of $L$} and $(\eta_i\colon G\to \Gamma_i)$ the \emph{defining sequence} of $S$. 
\end{definition}

Note that $S$ is indeed a quotient of $L$ since by Lemma~\ref{lem:ugly hack}.\eqref{eq:direct limit} and construction we have 
\begin{equation*}
\wker(\phi_i)=\wker(\xi_i)\subseteq\wker(\eta_i) ,
\end{equation*} 
so there is a natural quotient map $\pi\colon L\to S$. 

\begin{lemma} \label{l:sq not tdiv}
Suppose $\{\phi_i:G\rightarrow \Gamma_i\}$ is not \Tdiv. Let $\{\eta_i\}$ be the defining sequence of a $\Coll$--shortening quotient of $G/\wker(\phi_i)$. Then $\eta_i$ is not \Tdiv. 
\end{lemma} 

\begin{proof}
Let $g\in G$.  Represent $\phi_\infty(g)$ as a reduced $\Gdf$--loop $[a_0, e_1,..., e_n, a_n]$ based at $v$. For \walmost every $i$, $\xi_i(g)$ is a lift of $\phi_\infty(g)$ in $W_i$ and $\xi_i(g)=[\widetilde a_0, e_1,..., e_n, \widetilde a_n]$, where $\widetilde a_k$ is a lift of $a_k$ in the corresponding vertex group of the geometric splitting of $W_i$. By construction, the factoring map $\lambda_i$ is a morphism from the geometric splitting of $W_i$ to the geometric splitting of $\Gamma_i$. Combining this with the definition of the $\alpha_i$, $\widehat{\lambda_i}=\lambda_i\circ \alpha_i$ is also a morphism from the geometric splitting of $W_i$ to the geometric splitting of $\Gamma_i$. Hence $\eta_i(g)=\widehat{\lambda_i}(\xi_i(g))$ moves $v_i\in T_i$ by a distance not bigger than the amount by which $\phi_i(g)$ moves $v_i$. Therefore, $\eta_i(g)$ moves a point in $T_i$ independent of $g$ by an amount independent of $i$, so $\eta_i$ is not \Tdiv. 
\end{proof}

Recall $\mc{H}_{L,(\phi_i)}$ is the family of stably parabolic subgroups of $L$ with respect to $(\phi_i)$.  This family includes all subgroups fixing edges in $T_\infty$.

\begin{lemma}\label{lem:shortq}
 Each $P \in \mc{H}_{L,(\phi_i)}$ maps injectively into an element of $\mc{H}_{S,(\eta_i)}$. 
\end{lemma}  
\begin{proof}
If \walmost surely $\phi_i$ maps $g\in G$ to a non-trivial element in a parabolic subgroup of $\Gamma_i$, then $\phi_i(g)$ and $\eta_i(g)$ are conjugate \walmost surely, and hence $\eta_\infty(g)$ is a non-trivial element in a stably parabolic subgroup of $S$.  Since modular automorphisms act by conjugation on stably parabolic subgroups, it is easy to see that $\pi$ takes stably parabolic \emph{subgroups} of $L$ injectively to stably parabolic subgroups of $S$.
\end{proof}

We first consider the case where the quotient map $\pi$ is injective, which implies that $L\cong S$. Hence we can consider $(\eta_i)$ as a defining sequence for $L$ in this case. The following lemma is the reason why we consider GGDs instead of geometric decompositions. It follows directly from the construction of $\eta_i$ from $\phi_i$. 
 
\begin{lemma} \label{lem:still ggd}
If $\pi\colon L\to S$ is injective, then $\Gdf$ is the Linnell refinement of a GGD with respect to $(\eta_i )$.  
\end{lemma}

If $\pi\colon L\to S$ is injective then for a vertex $v$ of $\Gdf$ we define 

\[
\|\eta_i\|_{\Coll, v}=\inf_{x\in\Coll(\Gamma_{v_i})}\max_{s\in \widetilde{A}_v}d_i(\eta_i(s)x, x).
\]

Recall that in our definition of modular automorphisms, we only included squared Nielsen transformations for the dihedral-type vertex groups. The next lemma is what guarantees that this still provides enough automorphisms for the shortening argument to work. If we allowed all Nielsen transformations, then the conclusion of Lemma~\ref{l:snt is enough} below follows from the Euclidean algorithm. 

For a dihedral type group $D$ which is finitely generated relative to a family of subgroups $\mc{E}$, let ${\rm Aut}^{\mc{E}}_s(D)$ denote the subgroup of ${\rm Aut}(D)$ generated by squared Nielsen transformations relative to $\mc{E}$. We also define $D^+$ and $P_{D^+}$ as in Section \ref{s:modaut}.

\begin{lemma}\label{l:snt is enough}
Suppose $D$ is a dihedral type group which is finitely generated relative to a collection of subgroup $\mc{E}$ and  $g_1, \dots, g_k\in D$ project to form a basis of $D^+/P_{D^+}$. Suppose $D$ acts on the real line $T$ so that $\langle g_1,...g_k\rangle$ has indiscrete orbits. Then there exists a sequence $\alpha_i\in {\rm Aut}_s^{\mc{E}}(D)$  so that for all $1\leq j\leq k$, the translation length of $\alpha_i(g_j)$ goes to zero as $i\rightarrow \infty$. 
\end{lemma}

\begin{proof}
Without loss of generality, we can assume $k=2$ and that $g_1$ and $g_2$ act on $T$ as rationally independent translations in the same direction. We choose an orientation on $T$ and, for each $g\in D$ which acts as a translation on $T$, we set $\tau_g$ to be the signed translation length of $g$. That is, for any $x\in T$, $d_T(x, gx)=\tau_g$ when $g$ translates in the positive direction and $d_T(x, gx)=-\tau_g$ when $g$ translates in the negative direction, so the translation length of $g$ is $|\tau_g|$.

Without loss of generality, we can assume that $\tau_{g_1}>\tau_{g_2}>0$. To complete the proof, it suffices to find $\alpha\in {\rm Aut}_s^{\mc{E}}(D)$ such that $|\tau_{\alpha(g_1)}|<\frac12|\tau_{g_1}|$ and $|\tau_{\alpha(g_2)}|\leq |\tau_{g_2}|$.

Case 1: $0< \tau_{g_2}<\frac{1}{2}\tau_{g_1}$. Then we can choose $n\geq 1$ such that $0<\tau_{g_1}-2n\tau_{g_2}<\frac12\tau_{g_1}$. Hence we can choose $\alpha=\beta^n$, where $\beta$ is the the squared Nielsen transformation which fixes $g_2$ and sends $g_1$ to $g_1g_2^{-2}$.

Case 2: $\frac{1}{2}\tau_{g_1}<\tau_{g_2}<\tau_{g_1}$.
In this case, we let $\beta_1$ be the squared Nielsen transformation which fixes $g_2$ and sends $g_1$ to $g_1g_2^{-2}$. Let $x_1=\beta_1(g_1)^{-1}$, and note that $0<\tau_{x_1}<\tau_{g_1}$; moreover, 
\[
|\tau_{g_1}-\tau_{x_1}|=\tau_{g_1}-(-\tau_{g_1}+2\tau_{g_2})=2(\tau_{g_1}-\tau_{g_2}).
\]

Letting $y_1=\beta_1(g_2)=g_2$, we have that $|\tau_{y_1}-\tau_{x_1}|=\tau_{y_1}-\tau_{x_1}=\tau_{g_1}-\tau_{g_2}$. Now, if $\tau_{x_1}>\frac12\tau_{y_1}$, then we can let $\beta_2$ be the squared Nielsen transformation which fixes $x_1$ and sends $y_1$ to $y_1x_1^{-2}$ and set $y_2=\beta(y_1)^{-1}$ and $x_2=\beta_2(x_1)=x_1$. Then we will again have $|\tau_{y_2}-\tau_{x_2}|=|\tau_{y_1}-\tau_{x_1}|=\tau_{g_1}-\tau_{g_2}$ by the same calculation as before. Since we are decreasing by a fixed amount each time, after finitely many steps we will obtain $x_m$ and $y_m$ where either $0<\tau_{x_m}<\frac12\tau_{y_m}$ or $0<\tau_{y_m}<\frac12\tau_{x_m}$. From the construction, we also have that $0<\tau_{x_m}\leq\tau_{x_{m-1}}\leq...\leq\tau_{x_1}<\tau_{g_1}$, $0<\tau_{y_m}\leq\tau_{y_{m-1}}\leq...\leq\tau_{y_1}\leq\tau_{g_2}$, and there are squared Nielsen transformations $\beta_1,...,\beta_m$ such that $\beta_m\circ...\circ\beta_1(g_1)=x_m^{\pm 1}$ and $\beta_m\circ...\circ\beta_1(g_2)=y_m^{\pm 1}$. 

Case 2a: $0<\tau_{y_m}<\frac12\tau_{x_m}$. Then choosing $n$ and $\beta$ as in Case 1, and setting $x_{m+1}=\beta^n(x_m)$ we have that $\tau_{x_{m+1}}<\frac12\tau_{x_m}$. Then we can choose $\alpha=\beta^n\circ\beta_m\circ...\circ\beta_1$ since $|\tau_{\alpha(g_2)}|=|\tau_{y_m}|\leq|\tau_{g_2}|$ and
\[
|\tau_{\alpha(g_1)}|=|\tau_{x_{m+1}}|<\frac12|\tau_{x_m}|<\frac12|\tau_{g_1}|.
\]

Case 2b: $0<\tau_{x_m}<\frac12\tau_{y_m}$. Then we can again choose $\beta$ and $n$ such that $0<\tau_{y_{m+1}}<\frac12\tau_{y_m}$ where $y_{m+1}=\beta^n(y_m)$ and $x_{m+1}=\beta^n(x_m)=x_m$. In this case, we have 
\[
0<\tau_{y_{m+1}}<\frac12\tau_{y_m}\leq\frac12\tau_{g_2}<\frac12\tau_{g_1}.
\]

Hence, if $\tau_{x_{m+1}}<\tau_{y_{m+1}}$, we can set choose $\alpha=\beta^n\circ\beta_m\circ...\circ\beta_1$. If $\frac12\tau_{x_{m+1}}<\tau_{y_{m+1}}<\tau_{x_{m+1}}$, then, similar the beginning of Case 2, we can choose a squared Nielsen transformation $\theta$ which fixes $y_{m+1}$ sends $x_{m+1}$ to $x_{m+1}y_{m+1}^{-2}$ and set $x_{m+2}=\theta(x_{m+1})^{-1}$ and $y_{m+2}=\theta(y_{m+1})=y_{m+1}$. Then we have 
\[
0<\tau_{x_{m+2}}<\tau_{y_{m+2}}=\tau_{y_{m+1}}<\frac12\tau_{g_1}.
\]
So we can choose $\alpha=\theta\circ\beta^n\circ\beta_m\circ...\circ\beta_1$.

Finally, if $0<\tau_{y_{m+1}}<\frac12\tau_{x_{m+1}}$, arguing as in Case 1 we can find $\theta$ such that $0<\tau_{\theta(x_{m+1})}<\frac12\tau_{x_{m+1}}$ and $\theta$ is a power of a squared Nielsen transformation which fixes $y_{m+1}$. Since $\tau_{x_{m+1}}<\tau_{g_1}$, we can choose $\alpha=\theta\circ\beta^n\circ\beta_m\circ...\circ\beta_1$ in this case.

\end{proof}

\begin{theorem}\label{t:shortening argument}
If $\pi$ is injective then for every vertex $v$ of $\Gdf$ $\|\eta_i\|_{\Coll, v}$ does not diverge. 
\end{theorem}

\begin{proof}
The proof is based on Sela's shortening argument, which is now well understood by the experts in the field.  We give a sketch.
For sake of contradiction, suppose $\|\eta_i\|_{\Coll, v}$ diverges for some vertex $v$ of $\Gdf$, and let $\bp_i\in \Coll(\Gamma_{v_i})$ be as in Definition~\ref{d:stretch factor}.  Note that instead of each $\bp_i$ being centrally located and each $\widehat{\lambda}_i$ being short we only have  $\{\bp_i\}$ being a sequence of almost centrally located points and  $\{\widehat{\lambda}_i\}$ being an almost short sequence. As a result, we need the shortening automorphisms to shorten the translation length by an amount bounded below independently of $i$. Fortunately, the shortening automorphisms in Sela's shortening argument do have this property. 

By Theorem~\ref{t:Rtree summary}(\ref{eq:rtree}), $(\eta_i)$ induces an action of $\overline{\V}$ on the $\R$--tree $\Coll_{v, \infty}$ of $\Coll(\Gamma_{v_i})$.  Let $T$ be the minimal $\overline{\V}$-invariant subtree of $\Coll_{v, \infty}$, and $\bp\in \Coll_{v, \infty}$ the point defined by $\{\bp_i\}$. 

\begin{claim}\label{claim:intersect} For some $s\in \widetilde A$, the geodesic segment $[\bp, s\cdot\bp]$ intersects $T$ in a non-degenerate segment.
\end{claim}

\begin{proof} [Proof of Claim~\ref{claim:intersect}]
If $\bp\in T$, this follows by construction, as the action of $\langle \widetilde A \rangle$ on $T$ is non-trivial. Now suppose $\bp\notin T$. Let $o$ be the projection of $\bp$ in the closure (in $\Coll_{v, \infty}$) of $T$.  Observe $\langle \widetilde A \rangle$ does not fix $o$. Let $s\in \widetilde A$ such that $s\cdot o\neq o$. Then $[s\cdot o, o]$ has non-degenerate intersection with $T$ and hence $[s\cdot \bp, \bp]=[s\cdot\bp, s\cdot o]\cup[s\cdot o, o]\cup[o, \bp]$ also does.  This completes the proof of Claim~\ref{claim:intersect}.
\end{proof}

It follows from Theorem~\ref{t:Rtree summary}(\ref{eq:triv trip}) and (\ref{eq:abel seg}) and a standard argument that the action of $\overline{\V}$ on $T$ is super-stable. Since $\overline{\V}$ is freely indecomposable, by \cite[Theorem 5.1]{Guirardel:Rtrees}, there is a graph of actions decomposition of the $\overline{\V}$-tree $T$, whose vertex actions are either simplicial, axial, or Seifert type. Let $\mathbb A$ be the splitting of $\overline{\V}$ induced by this decomposition. Given a generator $s\in A$, consider the path $[\bp, s\cdot \bp]$. Choosing a generator for which this path is the longest, it has a non-trivial intersection with at least one piece in the graph of actions decomposition. Note that  stably parabolic subgroups of $\V$ fix points in $T$ by Lemma \ref{l:fixed for ell/par}. Let $\mc{H}_{\overline{\V}}$ be the collection of images of stably parabolic subgroups of $\V$ in $\overline{\V}$.  For the axial and Seifert type pieces, we can directly find a modular automorphism in $\Mod^{\mc{H}_{\overline{\V}}}_{\mathbb A}(\overline{\V})$ which shortens the segment of $[\bp, s\cdot\bp]$ in this piece without changing the lengths of the segments in any other pieces and without changing the lengths of any other generators.  In the axial case, the shortening automorphisms in $\Mod^{\mc{H}_{\overline{\V}}}_{\mathbb A}(\overline{\V})$ are more restrictive than those used in other applications of the shortening argument.  However, Lemma~\ref{l:snt is enough} ensures these automorphisms suffice. In the proof of \cite[Theorem 5.8]{ReiWei}, it is shown that the shortening automorphism shortens the translation length on $T$ by at least $\frac{r}{2}$. Here $r$ is the length of the shortest non-degenerate segment of the form  $[\bp, s\cdot\bp]\cap T'$, where $T'$ is some axial piece and $s\in \widetilde A$. Note that even though $r$ depends on the sequence $(\eta_i)$, it is independent of $i$.  The same proof works verbatim in our setting using Lemma~\ref{l:snt is enough} in place of the use of the Euclidean algorithm in the proof of \cite[Proposition 5.14]{ReiWei}. In the Seifert case, the shortening automorphism is completely standard, and also reduces the translation length by an amount independent of $i$.  By Theorem~\ref{thm:decomps} this modular automorphism belongs to $\Mod^{\mc{H}_{\overline{\V}}}_{\overline{\mathbb J}(v)}(\overline{\V})$ and by Lemma~\ref{l: lifting the shortening auto} it can be lifted to an automorphism in $\Mod^{\mc{H}_{\V}}_{\mathbb{J}(v)}(\V)$, where $\mc{H}_{\V}$ is the collection of stably parabolic subgroups of $V$. By construction \walmost surely this automorphism can be lifted to an element of $\Mod_v(W_i)$. Since the scaled (down) action of $W_{v, i}$ on $\Coll(\Gamma_{v_i})$ approximates the action of $\V$ on $T$, whose translation length is reduced by an amount independent of $i$ by the shortening automorphism, this automorphism of $W_i$ \walmost surely shortens $\widehat{\lambda}_i$ by an amount bounded below independent of $i$, a contradiction. In case $[\bp, s\cdot \bp]$ intersects a simplicial piece, we cannot find a automorphism which shortens the action of $\V$ on $T$. However, by Theorem~\ref{t:Rtree summary}.\eqref{eq:para unique fp} and \eqref{eq:ell unique fp} any non-trivial element in a segment stabilizer is stably loxodromic.  Thus, similarly as in \cite[Proposition 5.19]{ReiWei}, we can find an automorphism which shortens the action of $W_{v, i}$ on $\Coll(\Gamma_{v_i})$, One can see at the end of the proof of \cite[Proposition 5.19]{ReiWei} that the shortening automorphism reduces the translation length by at least $l(e)/\|\eta_i\|_{\Coll, v}$ , up to an error $o(i)$. Here l(e) is the length of an edge in a simplicial piece, which is independent of $i$. So we again reach a contradiction. 
\end{proof}

\subsection{The descending chain condition}
The proof of the following result is inspired by Sela's proof of \cite[Theorem 1.12]{Sela:hyp}.
Note that if $\{\phi_i:G\rightarrow \Gamma_i\}$ is a \Tdiv sequence of homomorphisms, then there is an analogous construction of shortening quotients of the limit group $G/\wker(\phi_i)$ using the actions of $\Gamma_i$ on $T_i$, where $T_i$ is the tree dual to the geometric decomposition of $\Gamma_i$. We refer to these types of shortening quotients as \emph{$\mc{T}$--shortening quotients}, see \cite[Section 6]{GH} for the details of their construction. 

\begin{theorem}\label{dcc}
There is no infinite sequence of $\Mgen$--limit groups
\[
L_1\overset{\alpha_1}{\onto} L_2\overset{\alpha_2}{\onto}...
\]
such that each $\alpha_i$ is a proper quotient map.
\end{theorem}

\begin{proof}
Towards a contradiction, suppose there is an infinite descending sequence of $\Mgen$--limit groups as above. Choose a sequence of $\Mgen$--limit groups $R_1\overset{\beta_1}{\onto} R_2...$ with each $R_n=F_k/\wker(\phi_i^n)$ and $\phi^n_\infty\circ\beta_n=\phi^{n+1}_\infty$ for all $n$, and further assume that, for each $n\geq 1$, $R_{n+1}$ is chosen such that if $R_n\onto S\onto...$ is any infinite descending sequence of $\Mgen$--limit groups with $S=F_k/\wker(\rho_i)$, then
\[
|\ker(\rho_\infty)\cap B_n|\leq |\ker(\phi^{n+1}_\infty)\cap B_n|
\]  
where $B_n$ is the ball of radius $n$ in $F_k$ with respect to the word metric.

Choose a diagonal sequence $(\psi_n=\phi^n_{j_n})$, where $j_n$ is so that $\ker(\psi_n)\cap B_n=\ker(\phi^n_\infty)\cap B_n$ and for some $g\in\ker(\phi^{n+1}_\infty)$, $g\notin\ker(\psi_n)$. Let $R_\infty=G/\wker(\psi_n)$; by construction, $\ker(\psi_\infty)=\bigcup\limits_{j=1}^\infty\ker(\phi^j_\infty)$, so $R_\infty$ is also the direct limit of the sequence  $F_n\onto R_1\onto R_2\onto...$.

We claim every descending sequence of $\Mgen$--limit groups 
\[
R_\infty\onto L_1\onto L_2...
\]
with proper quotient maps is finite. Indeed, if some element $g\in B_n$ maps trivially to $L_1$ but not to $R_\infty$, then $g$ maps non-trivially to $R_{n+1}$. But then there is an infinite descending sequence $R_n\onto L_1\onto...$, so this contradicts our choice of $R_{n+1}$.

Since there is no infinite descending chain starting at $R_\infty$, any sequence of proper $\Mgen$--limit quotients of $R_\infty$ is finite.   Consider the case where the defining sequence of homomorphisms for $R_\infty$ is \Tdiv.  In this case, by \cite[Lemmas 6.4 and 6.6]{GH} $R_\infty$ has a proper $\mc{T}$--shortening quotient $S_1$.  Since $R_\infty$ does not admit an infinite descending sequence of proper $\Mgen$--limit quotients, repeating this argument finitely many times we obtain a sequence of proper  quotients 
\[
R_\infty\twoheadrightarrow S_1\ldots \twoheadrightarrow S_q
\]
where the defining sequence of homomorphisms for $S_q$ is not \Tdiv.  If the defining sequence of homomorphisms for $S_q$ is \Cdiv then by applying the shortening argument above finitely many times we obtain a sequence of quotient maps
\[	S_q \twoheadrightarrow U_1 \ldots \twoheadrightarrow U_s	\]
of $\Mgen$--limit quotients, each of whose defining sequence of homomorphisms is not \Tdiv (by Lemma~\ref{l:sq not tdiv}) and terminating in an $\Mgen$--limit group $U_s$ which has a GGD $\Gd$ so that the defining sequence of homomorphisms for $U_s$ is not \Cdiv with respect to $\Gd$ (a sequence of proper quotients terminates by the construction of $R_\infty$, and once a $\Coll$--shortening quotient is not a proper quotient, the defining sequence of the quotient is not \Cdiv by Theorem~\ref{t:shortening argument}).  Each map $U_i \twoheadrightarrow U_{i+1}$ maps the stably parabolic subgroups of $U_i$ injectively into the stably parabolic subgroups of $U_{i+1}$, by Lemma~\ref{lem:shortq}.  By Theorem~\ref{t:NEW non-divergent} and the construction of $U_s$, all of the stably parabolic subgroups of $U_s$ are finitely generated.  It follows that all of the stably parabolic subgroups of $S_q$ are finitely generated.  It now follows from Theorem~\ref{t:fingenedges} that the defining sequence of homomorphisms for $S_q$ $\omega$--factors through the limit.  We now apply the arguments of \cite[$\S6$]{GH}.  Since $R_\infty$ does not admit an infinite descending sequence of proper $\Mgen$--limit quotients, and hence neither do any of the $S_i$, the hypotheses of \cite[Lemmas 6.5 and 6.6]{GH} are satisfied for $R_\infty$ and each of the $S_i$, and so these lemmas may be applied inductively from the bottom of the sequence to prove that the defining sequence for $R_\infty$ $\omega$--factors through the limit.
Then, from the construction of $R_\infty$, there exists $i$ such that $\ker(\psi_n)\subseteq\ker(\psi_\infty)\subseteq\ker(\phi^{n+1}_\infty)$, contradicting our construction of $\psi_n$.
\end{proof}

Finally, we prove Theorem~\ref{t:technical outline} from Section~\ref{s:geo decomp}

As explained in Section~\ref{s:geo decomp}, except for the results proved in the appendices, this completes the proof of Theorem~\ref{t:main}, and hence the paper.

\technical*

\begin{proof}
If $L$ is an $\Mgen$--limit group whose defining sequence is not \Tdiv then $L$ has a shortening quotient $\eta_1\colon L\onto S_1$ as defined in Definition~\ref{d:shortening quotient}. We first assume that $\eta_1$ is a proper quotient map. By Lemma \ref{l:sq not tdiv}, $S_1$ has a defining sequence of homomorphisms which is not \Tdiv. With respect to this defining sequence of $S_1$, we can construct the geometric decomposition of $S_1$ and its Linnell refinement as in Definitions \ref{def:gd} and \ref{def:refined ggd}. We can now apply the constructions of Sections \ref{s:JSJ} and \ref{s:res-short} to $S_1$ equipped with this refined geometric decomposition to produce a $\Coll$--shortening quotient $\eta_2\colon S_1\onto S_2$. 

We now inductively continue this process to produce a sequence of shortening quotients $L=S_0\onto S_1\onto S_2\onto...$, each of which is an $\Mgen$--limit group by construction. By Theorem~\ref{dcc}, any sequence of proper  quotients of $\Mgen$--limit groups eventually terminates. This means that we eventually produce a $\Coll$--shortening quotient $\eta_{k+1}\colon S_k\onto S_{k+1}$ which is not a proper quotient map, i.e. $\eta_{k+1}$ is an isomorphism. We choose the index $k\geq 0$ such that $\eta_{k+1}$ is the first map in this sequence which is not a proper quotient map, which means that $\eta_i$ is a proper quotient map for each $1\leq i\leq k$. Note that we may have $k=0$, which happens when there are no proper quotient maps and $\eta_1$ is an isomorphism. By Lemma \ref{lem:shortq}, for each $1\leq i\leq k+1$, each stably parabolic subgroup of $S_{i-1}$ maps injectively into a stably parabolic subgroup of $S_i$. Finally, since $\eta_{k+1}$ is injective, it follows from Lemmas \ref{lem:still ggd} and \ref{t:shortening argument} that there is a GGD for $S_{k+1}$ with respect to which $S_{k+1}$ is not \Cdiv.
\end{proof}

\appendix

\section{Edge-twisted graphs of groups} \label{app:edge-twist}
In this appendix we prove the following result from Section~\ref{s:non-divergent}.

\edgetwistfg*

We repeat the definition of an edge-twisted splitting for convenience.

\defedgetwist*

\subsection{Abelian vertex groups in graphs of groups}

We begin with some results about graphs of groups with abelian vertex groups.

\begin{lemma}\label{w}
Suppose $A$ is an abelian group and $K, L \le A$ satisfy $K\cap L = \{ 1 \}$. Let $N \le A$ be finitely generated. Then $\langle N, K\rangle\cap L$ is finitely generated. 
\end{lemma}

\begin{proof}
Let 
\begin{equation*}
N_0=\{g\in N\mid \text{there exists } g'\in K, \text{such that } gg'\in L\}
\end{equation*}
Then $N_0$ is a subgroup of the finitely generated abelian group $N$, so $N_0$ is finitely generated. Define $\phi: N_0\rightarrow \langle N, K\rangle\cap L$ by $\phi(g)=gg'$ where $g' \in K$ is so that $gg'\in L$. 
If $gg'\in L$ and $gg''\in L$ for $g', g''\in K$ then $g'(g'')^{-1}\in L$. But $g'(g'')^{-1}\in K$ and $K\cap L=\{1\}$. So $g'=g''$, and $\phi$ is well-defined. It is easy to check that $\phi$ is a homomorphism. For any $g_0\in \langle N, K\rangle\cap L$, we have $g_0=gg'$ for some $g\in N$ and $g'\in K$. Then $g_0=\phi(g)$, so $\phi$ is surjective.  Since $\langle N,K \rangle \cap L$ is the image of the finitely generated group $N_0$, it is finitely generated, as required.
\end{proof}

\begin{lemma}\label{amalgam}
Let $G=M*_{K}A$, where $A$ is abelian. Let $S_1\subset M$, $S_2\subset A$ and let $S_3$ be a generating set of $K \cap \langle S_2\rangle$.  Suppose $g$ is a word in $S_1\cup S_2$. Then
\begin{enumerate}
\item\label{eq:M gen S1S3} If $g\in M$,  then $g$ can be written as a word in $S_1\cup S_3$. 
\item\label{eq:g=ma} If $g\in A$, then $g=ma$, where $m\in K \cap \langle S_1\cup S_3 \rangle$ and $a\in \langle S_2 \rangle$. 
\end{enumerate}
In particular, if $G$ is finitely generated, then $M$ is finitely generated. 
\end{lemma}

\begin{proof}
By the assumptions of the lemma, we have  
\begin{equation*}
g= r_1s_1\dots r_k s_k
\end{equation*}
where each $r_i\in \langle S_1 \rangle$ (and hence in $\langle S_1\cup S_3 \rangle$) and each $s_i\in \langle S_2 \rangle$. If some $i > 1$ we have $s_i \in K$ then $s_i \in \langle S_3 \rangle$, so $r_i s_i r_{i+1} \in \langle S_1 \cup S_3 \rangle$.  We can repeat this reduction until no $s_i$ lies in $K$.  If $r_i\in K$ for some $i > 1$ then $r_i\in A$. Since $A$ is abelian, we can write $r_{i-1}s_{i-1}r_is_i$ as $r_{i-1}r_is_{i-1}s_i$, and reduce the number of syllables in our description of $g$.  After repeating finitely many times, this result is an expression:
\begin{equation}\label{e}
	g = p_1q_1 \ldots p_lq_l	
\end{equation}
where each $p_i \in \langle S_1 \cup S_3 \rangle$, each $q_i \in \langle S_2 \rangle$.  Moreover, none of the $p_i$ or $q_i$ lies in $K$, except possibly that $q_l$ is trivial, and we cannot control $p_1$ since the above reduction on $s_i$ required $i > 1$.

Consider the case where $g \in M$.  In this case, the above expression cannot contain any $q_i$, so $g = p_1 \in \langle S_1 \cup S_3 \rangle$, as required.

Now suppose that $g\in A$. If $g \in K$ then $g \in M$ and we are in the first case.  In this case, we can take $a = 1$ and $m = g$ in the conclusion of the lemma.  Suppose then that $g \not\in K$.  In this case, we have $l = 1$, since otherwise it is not possible that $g \in A$.  This proves \eqref{eq:g=ma}.

For the last assertion of the lemma, take a finite generating set for $G$, and let $S_1, S_2$ be all of the elements of $M,A$ that appear in normal forms for these generators.  Then $S_1$ and $S_2$ are finite, and generate $G$.   It follows that $S_3$ can be chosen to be finite, since $K \cap \langle S_2 \rangle$ is a subgroup of the finitely generated abelian group $\langle S_2 \rangle$.  It follows from Item~\eqref{eq:M gen S1S3} that $M$ is finitely generated, as required. 
\end{proof}

The next two results deal with the simplest cases of Theorem~\ref{t:fin gen}: those $\E$ with two edges.  The general case follows quickly.

\begin{lemma}\label{non-sep}
Suppose $\E$ is an edge-twisted graph of groups so $G = \pi_1(\E)$ is finitely generated.  Suppose further that $\E$ has two edges and two vertices.  Let $M$ be the type B vertex group of $\E$.  Then $M$ is finitely generated.
\end{lemma}

\begin{proof}
Let $K$ and $L$ be the edge groups of $\E$, corresponding to edges $e_K$ and $e_L$ respectively.  Let $A$ be the type A vertex group of $\E$.  Choose the edge associated to $e_K$ as a maximal tree, which allows us to consider $K$ as a subgroup of both $A$ and $M$.  Consider $L$ to be a subgroup of $A$, let $t$ be the stable letter of the splitting, and let $L' := t^{-1}Lt \le M$.

Let $K_0 \le K$ and $L_0 \le L$ be the subgroups of $K$ and $L$ guaranteed by Definition~\ref{def:edge-twist}.  Thus,
\begin{enumerate}
\item $K_0 \cap L_0 = \{ 1 \}$; and
\item There are finite sets $T_K \subset K$ and $T_L \subset L$ so that $K = \langle T_K, K_0 \rangle$ and $L = \langle T_L, L_0 \rangle$.
\end{enumerate}

Choose finite subsets $S_A \subset A$ containing $T_K\cup T_L$ and $S_M \subset M$ containing $T_K$ so that $G = \langle S_A, S_M, t \rangle$.  Let $N_0 = \langle S_A \rangle$.
Since $K_0\cap L_0 = \{ 1 \}$, by Lemma~\ref{w}, $\langle N_0, K_0\rangle\cap L_0$  and $\langle N_0,L_0\rangle \cap K_0$ are finitely generated, by $R_1$ and $R_2$, respectively, say. 
Let $N_1 = \langle N_0, R_1, R_2 \rangle$.

We claim $\langle N_1 , K_0 \rangle \cap L_0 \le N_1$.
Suppose $g \in \langle N_1, K_0 \rangle \cap L_0$.  Then $g = nkr_1r_2$ where $n \in N_0$, $k \in K_0$, $r_1 \in \langle R_1 \rangle$ and $r_2 \in \langle R_2 \rangle$.
Since $g \in L_0$ and $\langle R_1 \rangle \le L_0$, we have $nkr_2 \in L_0$.  But $k, r_2 \in K_0$, so $nkr_2 \in \langle N_0, K_0 \rangle \cap L_0 =  \langle R_1 \rangle$.  Thus, $g = (nkr_2)r_1 \in \langle R_1 \rangle \le N_1$.  
A similar argument shows $\langle N_1, L_0 \rangle \cap K_0 \le \langle R_2\rangle\le N_1$.

We further claim that 
 $\langle N_1, K\rangle\cap L\le  N_1$.
Indeed, since $T_K \subset S_A\subset N_1$ we have $\langle N_1,K \rangle \cap L_0 = \langle N_1, K_0 \rangle \cap L_0 \le N_1$.  Now suppose that $g \in \langle N_1, K \rangle \cap L$, and write $g = ng_0$ where $n \in \langle T_L \rangle \le N_1$ and $g_0 \in L_0$.  We have $g \in \langle N_1 , K \rangle$ and $n \in N_1$, so $g_0 \in \langle N_1 , K \rangle \cap L_0 \le N_1$. Therefore, $g = ng_0 \in N_1$, as required.   
Similarly, we have $\langle N_1,L\rangle \cap K\subset N_1$.

Since $N_1$ is a finitely generated abelian group, $N_1\cap K$ and $N_1\cap L$ are both finitely generated. So 
\begin{equation}
M'=\langle S_M, N_1\cap K,  t^{-1}(N_1\cap L)t\rangle
\end{equation}
is finitely generated.  We claim that $M = M'$, which proves the lemma.  We clearly have $M' \le M$, so we have to prove the opposite inclusion.

To that end, let 
\begin{equation}
N_2 =\langle N_1, M'\cap K, t(M'\cap L')t^{-1}\rangle\subset A. 
\end{equation}
(Recall that $L' = t^{-1}Lt \le M$.)

We claim that 
\begin{enumerate}
\item $N_2\cap K=\langle N_1\cap K, M'\cap K\rangle$; and
\item $N_2\cap L=\langle N_1\cap L, t(M'\cap L')t^{-1}\rangle$. 
\end{enumerate}
From the definition of $N_2$, it is clear that $\langle N_1 \cap K, M' \cap K\rangle \le N_2 \cap K$.  For the reverse inclusion, since $t(M' \cap L')t^{-1} \le L$, we have 
\[	\langle N_1, t(M' \cap L')t^{-1} \rangle \cap K \le \langle N_1, L \rangle \cap K \le N_1 \cap K.	\]
Since $A$ is abelian, for any subgroup $J \le K$ we have $\langle N,J \rangle \cap K = \langle N \cap K, J \rangle$.  Therefore,
\begin{eqnarray*}
N_2 \cap K &=& \langle N_1, t(M'\cap L')t^{-1}, M'\cap K \rangle \cap K \\
&=& \langle \langle N_1, t(M' \cap L')t^{-1} \rangle \cap K, M' \cap K \rangle \\
&\le& \langle N_1 \cap K, M' \cap K \rangle	,
\end{eqnarray*}
as required.   The second equality of the claim is entirely similar.

As a result, we have $M' = \langle S_M,  N_2\cap K, t^{-1}(N_2\cap L)t\rangle$.
To finish the proof, we now show that $M \le M'$. 
Choose sets $S_1 \supset S_M$ and $S_2 \supset S_A$ so that $\langle S_1 \rangle = M'$ and $\langle S_2 \rangle = N_2$. 
 Let $S_3$ be a generating set for $K \cap N_2$.  By the above $\langle S_3 \rangle \le M'$.

Consider the subgroup $\langle M,A \rangle \le G$ (which is isomorphic to $M \ast_K A$).  Since $S_M \subset S_1$ and $S_A \subset S_2$ we have $G = \langle S_1, S_2, t \rangle$.

Suppose that $g \in M$, and write $g = w_1 \dots w_j$
where each $w_i$ is either $t$ or $t^{-1}$, in $\langle S_1 \rangle$, or in $\langle S_2 \rangle$.
Since $g \in M$, if there are any occurrences of $t$ or $t^{-1}$ then by Britton's Lemma there is a subword of one of the following two forms:
\begin{enumerate}
\item $twt^{-1}$ where $w \in \langle S_1 \cup S_2 \rangle \cap L'$; or
\item $t^{-1}ut$ where $u \in \langle S_1 \cup S_2 \rangle \cap L$.
\end{enumerate}
In the first case, $w \in L' \le M$, so by Lemma~\ref{amalgam} $w \in \langle S_1 \cup S_3 \rangle$, and so $w \in M'$.  Then $twt^{-1} \in t(M' \cap L')t^{-1} \le N_2$, so $twt^{-1}$ can be replaced by an element in $\langle S_2 \rangle$.  In the second case, $u \in L \le A$.  By Lemma~\ref{amalgam}, $u$ can be written as $u = ma$ where $m \in K \cap \langle S_1 \cup S_3 \rangle$ and $a \in  \langle S_2 \rangle$. So $m \in K \cap M' \le N_2$ and $a \in N_2$, so $u \in N_2 \cap L$.  Therefore, $t^{-1}ut \in t^{-1}(N_2 \cap L)t \le M'$, so we can replace $t^{-1}ut$ by an element of $\langle S_1 \rangle$.

Repeating until there are no occurrences of $t$ or $t^{-1}$, we obtain $g \in \langle S_1 \cup S_2 \rangle$.  By Lemma~\ref{amalgam}, $g \in \langle S_1 \cup S_3 \rangle = M'$, so $M = M'$ as required.  
\end{proof}

The proof of the following lemma is very similar to that of Lemma~\ref{non-sep}, and we omit it.
\begin{lemma}\label{sep}
Let $\E$ be an edge-twisted splitting with two edges, two type B vertices and one type A vertex so that $\pi_1(\E)$ is finitely generated.  Then the type B vertex groups are finitely generated.
\end{lemma}

We now give the proof of the main result of this appendix.

\edgetwistfg*

\begin{proof}
We proceed by induction on the number of type A vertex groups of $\E$.  If there are none, then $\E$ is a single type B vertex group, and the result is trivial.  If there is a single type A vertex group, the result follows from 
Lemmas~\ref{non-sep} and~\ref{sep}.

Suppose now that $\E$ has $k$ type A vertices for $k > 1$ and that the result is true for any finite edge-twisted graph of groups with finitely generated fundamental group and at most $k -1$ type A vertices.  Let $A$ be a type A vertex group of $\E$. Let $\E_0$ be the graph of groups obtained from $\E$ by collapsing all of the edges that are not adjacent to $A$.  It is clear that $\E_0$ is an edge-twisted graph of groups with a single type A vertex group, and so the by induction the type B vertex groups of $\E_0$ are finitely generated by the case $k = 1$.  However, the type $B$ vertex groups of $\E_0$ are the fundamental groups of (edge-twisted) sub-graphs of groups of $\E$, with fewer type A vertex groups than $\E$.  Since these sub-graphs of groups have finitely generated fundamental group, the inductive hypothesis applies to them to prove that their type B vertex groups are finitely generated.  But these are the type B vertex groups of $\E$, so the result is proved.
\end{proof}

\section{Relative Linnell and JSJ decompositions} \label{app:decompositions}

The purpose of this section is to prove Corollary~\ref{cor:lindecomp}, about the existence of the relative $C$--Linnell decomposition which was used in Definition~\ref{def:refined ggd}, and to prove Theorem~\ref{thm:JSJwithtorsion} which is used in Section~\ref{s:JSJ}.  

\subsection{Relative acylindrical accessibility}

We first construct the decomposition $\mathbb L$ from Theorem~\ref{thm:decomps}. We refer to this decomposition as a  \emph{relative Linnell decomposition}. The existence of such a decomposition is proved in \cite{Linnell-Access} when $G$ is finitely generated, and it can also be derived from Weidmann's version of acylindrical accessibility \cite{Weidmann-Access}. Here we modify Weidmann's argument to the case where $G$ is finitely generated relative to a finite collection of subgroups $\mc{H}$ and all splittings considered are rel $\mc{H}$. 

\begin{lemma}\label{l:relative lemma}
Let $T$ be a minimal $G$--tree. Let $G_0$ be a subgroup of $G$ with no global fixed point in $T$ and let $T_0$ be the minimal $G_0$--invariant subtree of $T$. Let $H$ be a subgroup of $G$ which fixes a point $x$ such that $d(x, T_0)=k$. Let $E=\langle G_0, H\rangle$ and let $T_E$ be the minimal $E$--invariant subtree of $T$. The number of edges in $T_E/E$ is at most the number of edges in $T_0/G_0$ plus $k$. 
\end{lemma}

\begin{proof}
Let $T'$ be the union of $T_0$ and the path from $x$ to $T_0$. he $E$--orbit of $T'$, denoted by $T''$, is connected since $gT'\cup T'$ is connected for each $g\in G_0\cup H$, and $T''/E$ has at most  $k$ more edges than $T_0/G_0$. On the other hand, $T_E\subset T''$, so the lemma follows. 
\end{proof}

\begin{definition}
A $G$--action on a tree is \emph{$(k, C)$--acylindrical} (for $k \ge 0$ and $C \ge 1$) if the stabilizer of all paths of length $\geq k+1$ has size at most $C$. 
\end{definition}

\begin{proposition}
Fix $k \ge 0$ and $C \ge 1$.  Suppose $G$ is finitely generated relative to a finite collection of subgroups $\mc{H}$.  There is a bound on the number of edges in minimal $(k, C)$--acylindrical splittings of $G$ rel $\mc{H}$. 
\end{proposition}

\begin{proof}
Let $\mc{H}=\{H_1,...,H_n\}$. It suffices to assume that each $H_i$ is infinite. Let $S\subset G$ be a finite set such that $S\cup H_1\cup\cdots\cup H_n$ generates $G$ and $|S\cap H_i|\geq C+1$ for each $i$. Let $G_0=\langle S\rangle$. By Weidmann's $(k, C)$--acylindrical-accessibility \cite[Theorem 1]{Weidmann-Access}, the number of edges of any $(k, C)$--acylindrical splitting of $G_0$ is bounded by a constant $B$. Let $\mathbb A$ be a $(k, C)$--acylindrical splitting of $G$ rel $\mc{H}$, and let $T$ be the corresponding Bass-Serre tree. First suppose that $G_0$ does not fix a point in $T$, and let $T_0$ be the minimal $G_0$ invariant subtree. For each $i$, let $x_i$ be a fixed point of $H_i$ and let $y_i$ be the closest point in $T_0$ to $x_i$. Since $S\cap H_i$ fixes $[x_i, y_i]$ and $|S\cap H_i|\geq C+1$, we have $d_T(x_i, T_0)=d_T(x_i, y_i)\leq k$

Let $G_i$ be the subgroup of $G$ generated by $G_0$ and $H_1, \dots, H_i$. Then $G_n=G$. Let $T_i$ be the $G_i$--minimal invariant subtree. Clearly we have $ T_i\subset T_{i+1}$.  Each $H_i$ fixes a point at distance at most $k$ from $T_{i-1}$, so by Lemma~\ref{l:relative lemma} $T_i/G_i$ has at most $k$ more edges than $T_{i-1}/G_{i-1}$. Then $T/G=T_n/G_n$ has at most $B+kn$ edges. 

In case $G_0$ does fix a point in $T$ the above argument shows that if $y \in \Fix(G_0)$ and $x_i \in \Fix(H_i)$ then for each $i$ $d(y,x_i) \le k$.  Let $\mc{F}$ be the union of the geodesics $[y,x_i]$ and note $\mc{F}$ contains at most $kn$ edges.  Since the $G$--orbit of $\mc{F}$ covers $T$, we are done in this case also.
\end{proof}

\begin{corollary}\label{cor:lindecomp}
Suppose $G$ is finitely generated relative to a finite collection of subgroups $\mc{H}$. For all $C\geq 1$, $G$ has a splitting rel $\mc{H}$ in which all edge groups have size $\leq C$ and no vertex group splits  rel $\mc{H}$ over a subgroup of size $\leq C$.
\end{corollary}

\subsection{JSJ-decompositions for relatively finitely generated groups with torsion}
The purpose of this section is to prove Theorem~\ref{thm:JSJwithtorsion}.  We largely follow Guirardel--Levitt \cite{GuiLev2, GL11}, and explain the changes in our situation.  We refer to \cite{GuiLev2} for terminology.

The JSJ decomposition we use is essentially \cite[Theorem 9.14]{GuiLev2}, but we need to weaken hypotheses. First, we need to accommodate relatively finitely generated groups. In addition, we allow groups that are not $K$--CSA in the sense of \cite{GuiLev2} but instead satisfy are weakly $K$--CSA as in Definition \ref{defn:wcsa}. We repeat the definition here for convenience.

\defwcsa*

If $G$ is finitely generated and $K$--CSA the next result follows from \cite[Theorem 9.14]{GuiLev2}. We briefly sketch how to modify their proof. 

\begin{theorem}\label{thm:JSJwithtorsion}
Let $G$ be a weakly $K$--CSA group, let $\mc{H}$ be a finite collection of subgroups of $G$, and let $\mc{A}$ be the collection of virtually abelian subgroups of $G$. Suppose $G$ is finitely generated rel $\mc{H}$ and $G$ does not split over a subgroup of order $\leq 2K$ rel $\mc{H}$. The $(\mc{A}, \mc{H})$--JSJ decomposition of $G$ exists and all flexible vertex stabilizers are virtually abelian or QH with fiber of size at most $K$. Its tree of cylinders is compatible with every $(\mc{A}, \mc{H})$--tree. 
\end{theorem}

\begin{proof}
The proof of \cite[Theorem 9.14]{GuiLev2} is based on \cite[Sections 7 and 8]{GuiLev2}.  In \cite[Section 7]{GuiLev2} the finite generation assumption is never used. While it is used in \cite[Section 8]{GuiLev2}, there is a remark at the beginning of that section that it suffices to assume $G$ is finitely generated rel $\mc{H}$ \cite[p. 80]{GuiLev2}. Generalizing \cite[Section 8]{GuiLev2} to the relatively finitely generated setting is straightforward and we omit the details.

Now we discuss replacing $K$--CSA groups with weakly $K$--CSA groups in the proof of \cite[Theorem 9.14]{GuiLev2}. In \cite[Section 7]{GuiLev2}, they construct the tree of cylinders of a given $(\mc{A}, \mc{H})$--tree $T$.  In \cite{GuiLev2} all groups in $\mc{A}$ are infinite, whereas our $\mc{A}$ may contain finite groups. However, the construction works the same way in this case, as can be seen in \cite{GL11}. We next discuss this construction and verify it has the properties we need even when $\mc{A}$ contains finite groups.

For $A, B\in \mc{A}$, define $A\sim B$ if $\langle A, B\rangle$ is virtually abelian. Since $G$ is weakly $K$--CSA, this is an admissible equivalence relation on $\mc{A}$ (see \cite[Lemma 9.13]{GuiLev2}), so for any $(\mc{A}, \mc{H})$--tree $T$, we can form the tree of cylinders $T_c$  (see \cite{GL11}). If $A$ and $B$ are not equivalent, then by assumption $|A\cap B|\leq K$, so the action of $G$ on $T_c$ is $(2, K)$--acylindrical (See the proof of \cite[Lemma 7.7]{GuiLev2} or \cite[Proposition 5.13]{GL11}). Now, $T$ dominates $T_c$, and vertex stabilizers of $T_c$ which are not elliptic in $T$ are (maximal) non-(virtually cyclic) virtually abelian subgroups. That is, $T$ smally dominates $T_c$. Then $\mc{A}$ contains all virtually cyclic and all finite subgroups of $G$, so by \cite[Theorem 8.7]{GuiLev2} there exists an $(\mc{A} ,\mc{H})$--JSJ tree whose flexible vertices are either virtually abelian or QH with fiber of size at most $K$.

If $T_{JSJ}$ is the JSJ-tree its tree of cylinders is compatible with every $(\mc{A} ,\mc{H})$--tree. This follows from \cite[Lemmas 7.14 and 7.15]{GuiLev2}. Note that \cite[Lemma 7.15]{GuiLev2} assumes one-endedness. Here is how this is used: Suppose $S\to T$ collapses a single $G$--orbit of edges into the $G$--orbit of a vertex $v$ of $T$. Let $H=\Stab_G(v)$. In case $H$ is small in $S$ the proof from \cite{GuiLev2} works as written. Assume that $H$ is QH with fiber $F$, where $|F|\leq K$. Let $S_v$ be the minimal subtree of $S$ for the stabilizer $H$. Since $G$ does not split over any subgroup of order $\leq 2K$, by \cite[Lemma 5.16]{GuiLev2} all boundary components of the associated orbifold are used. Hence by \cite[Lemma 5.18]{GuiLev2} the splitting of $H$ corresponding to the action on $S_v$ is dual to a family of geodesics on the orbifold. 

The claim is that any cylinder of $S$ containing an edge of $S_v$ is contained in $S_v$. Suppose there are edges $e$ and $f$ with stabilizers $A=\Stab_G(e)$ and $B=\Stab_G(f)$ so that $e$ is in $S_v$ and $f$ has exactly one endpoint in $S_v$. One-endedness would imply that $B$ is infinite, and in this case the proof in \cite{GuiLev2} works as written. So assume that $B$ is finite. Since $B/F$ is a subgroup of the orbifold corresponding to the QH vertex $v$, it is in fact cyclic and contained (or conjugate into) a cone point subgroup of the orbifold (it does not come from a mirror since $|B/F|\geq 3$). By \cite[Proposition 5.4]{GuiLev2}, $A/F$ is cyclic subgroup corresponding to a geodesic on the orbifold. Hence $\langle A, B\rangle$ is not virtually abelian, so these groups are not equivalent.

The rest of the proof of \cite[Theorem 9.14]{GuiLev2} works verbatim in our situation, proving Theorem~\ref{thm:JSJwithtorsion}. 
\end{proof}

\small
\bibliographystyle{abbrv}

\end{document}